\documentclass[aos,preprint]{imsart}
\usepackage{amsthm,amsmath}
\RequirePackage[numbers]{natbib}

\usepackage{smile}
\usepackage{footmisc}
\RequirePackage[colorlinks,citecolor=blue,urlcolor=blue]{hyperref}
\numberwithin{equation}{section}
\numberwithin{theorem}{section}
\numberwithin{corollary}{section}
\numberwithin{asmp}{section}
\numberwithin{definition}{section}
\newtheorem{prop}{Proposition}

\usepackage{algorithm}
\usepackage{algpseudocode}
\algnewcommand\algorithmicinput{\textbf{Input:}}
\algnewcommand\INPUT{\item[\algorithmicinput]}
\algnewcommand\algorithmicoutput{\textbf{Output:}}
\algnewcommand\OUTPUT{\item[\algorithmicoutput]}
\usepackage{bibunits}

\startlocaldefs
\def\reals{{\mathbb R}}
\def\P{{\mathbb P}}
\def\E{{\mathbb E}}

\def\supp{\mathop{\text{supp}\kern.2ex}}

\def\supp{\mathop{\text{supp}}}

\endlocaldefs

\begin{document}
\begin{bibunit}
\begin{frontmatter}

\title{A Unified Framework for Testing High Dimensional Parameters: A Data-Adaptive Approach}

\begin{aug}
\author{\fnms{Cheng} \snm{Zhou}\thanksref{t1}\ead[label=e1]{chengzhmike@gmail.com}},
\author{\fnms{Xinsheng} \snm{Zhang}\thanksref{t1}\ead[label=e2]{xszhang@fudan.edu.cn}},
\author{\fnms{Wenxin} \snm{Zhou}\thanksref{t2}\ead[label=e2]{wenxinz@princeton.edu}}
\and
\author{\fnms{Han} \snm{Liu}\thanksref{t2}\ead[label=e3]{hanliu@princeton.edu}}

\affiliation{Department of Statistics, Fudan University\thanksmark{t1} and Department of Operation Research and
	Financial Engineering, Princeton University\thanksmark{t2}}
\end{aug}




\begin{abstract}
	High dimensional hypothesis test  deals with models in which the  number of parameters is significantly larger than the sample size. Existing literature develops a variety of individual tests. Some of them are sensitive to the dense and small disturbance, and others are sensitive to the sparse and large disturbance. Hence, the powers of these tests depend on the assumption of the alternative scenario. This paper provides a unified framework for developing new tests which are adaptive to a large variety of  alternative scenarios in high dimensions. In particular, our framework includes arbitrary  hypotheses which can be  tested using  high dimensional $U$-statistic based vectors. Under this framework,  we  first develop  a broad family of tests based on a novel variant of the $L_p$-norm with $p\in \{1,\dots,\infty\}$. We then combine these tests to  construct a data-adaptive test that is simultaneously powerful under various  alternative scenarios. To obtain the asymptotic distributions of these tests, we utilize the  multiplier bootstrap for $U$-statistics. 
	In addition, we consider the computational aspect of the bootstrap method and propose a novel low cost scheme. We prove the optimality of the proposed tests. Thorough numerical results on simulated and real datasets are provided to support our theory. 
\end{abstract}
\begin{keyword}
	\kwd{High dimensional hypothesis tests}
	\kwd{$U$-statistics}
	\kwd{multiplier bootstrap methods}
	\kwd{data-adaptive tests}
\end{keyword}

\end{frontmatter}

\section{Introduction}
Modern data acquisition routinely produces massive datasets in many scientific areas, e.g. genomics, astronomy, functional Magnetic Resonance Imaging (fMRI), and image processing. Effective analysis of such data requires us to test high dimensional parameters (\cite{
	lee2012optimal,pan2014powerful, yu2007altered,zou2008improved, fama2012size,cakici2015five}). Though specific methods have been developed to infer high dimensional mean and covariance parameters.  It is unclear how to choose the best test when the parameter of interest has a complex structure and the  pattern of possible alternative hypothesis is unknown. In particular,
we need a unified framework for constructing tests of high dimensional parameters which are simultaneously powerful under a large variety of alternative assumptions. This paper provides such a framework.
\subsection{General setup}

Our framework considers a generic setup for high dimensional inference.  More specifically, let $\bX=(X_1,\ldots,X_d)^\top$ and $\bY=(Y_1,\ldots,Y_d)^\top$ be two $d$-dimensional  random vectors
independent of each other. $\bX_1,\ldots,\bX_{n_1}$ are independent and identically distributed (i.i.d.)
random samples from $\bX$ with $\bX_k=(X_{k1},X_{k2},\ldots,X_{kd})^\top$. Similarly,  $\bY_1,\ldots, \bY_{n_2}$ are i.i.d. random
samples from $\bY$ with $\bY_k=(Y_{k1},Y_{k2},\ldots,Y_{kd})^\top$. We set $\mathcal{X}=\{\bX_1,\ldots,\bX_{n_1}\}$, $\mathcal{Y}=\{\bY_1,\ldots,\bY_{n_2}\}$, and  
\begin{equation}\label{def:uhat}
\begin{array}{rl}
\hat{u}_{1,s}&=\binom{n_1}{m}^{-1}\sum\limits_{1\le k_1<\cdots<k_m\le n_1}\Phi_{s}(\bX_{k_1},\ldots,\bX_{k_m}),\\
\hat{u}_{2,s}&=\binom{n_2}{m}^{-1}\sum\limits_{1\le k_1<\cdots<k_m\le n_2}\Phi_{s}(\bY_{k_1},\ldots,\bY_{k_m}),
\end{array}
\end{equation}
where $s=1,\ldots,q$, and $\Phi_s$ is a $m$-order symmetric kernel function. We assume that $\Phi_s$ is symmetric and that each kernel function is of the same order $m$ only for notational simplicity.\footnote{If $\Phi_s$ is an asymmetric kernel function, it gives a  $U$-statistic $\hat{u}_{1,s}=\frac{1}{m!}\binom{n_1}{m}^{-1}\sum\Phi_{s}(\bX_{\ell_1},\ldots,\bX_{\ell_m})$, where the summation is over all permutations of distinct elements 
	$\{\ell_1,\ldots,\ell_m\}$ from $\{1,\ldots,n_1\}$. By setting  $\Phi_s^{0}(\xb_1,\ldots,\xb_m)=(m!)^{-1}\sum\Phi_s(\bx_{k_1},\ldots,\bX_{k_m})$, where the summation is over all permutations of
	$\{1,\ldots,m\}$, we rewrite $\hat{u}_{1,s}$ as a $U$-statistic with a symmetric kernel $\Phi_s^0$. For $\bY$, we can rewrite $\hat{u}_{2,s}$ as a $U$-statistic with a symmetric kernel similarly.}\textsuperscript{,}\footnote{ If  $\{\Phi_s\}_{s=1,\ldots,q}$ have different kernel orders, we require that the kernel orders are uniformly bounded.  } 

We then define two $U$-statistic  based vectors as  
\begin{equation}\label{def:uvector}
\hat{\bu}_1:=(\hat{u}_{1,1},\hat{u}_{1,2},\ldots,\hat{u}_{1,q})^\top\hspace{1em}{\rm and}\hspace{1em} \hat{\bu}_2:=(\hat{u}_{2,1},\hat{u}_{2,2},\ldots,\hat{u}_{2,q})^\top.
\end{equation}
We use $\bu_{\gamma}$ to denote the expectation of $\hat{\bu}_{\gamma}$, i.e., $\bu_{\gamma}=(u_{\gamma,1},u_{\gamma,2},\ldots,u_{\gamma,q})^\top$ with $u_{\gamma,s}=\E[\hat{u}_{\gamma,s}]$ for $\gamma=1,2$ and $s=1,\ldots,q$. We are interested in testing  the hypotheses:
\begin{itemize}
	\item[(i)](One-sample problem) For a given $\bu_0\in \reals^q$,
	\begin{equation}\label{def:hone-sampleu}
	\Hb_0: \bu_1=\bu_0\hspace{2em}{\rm v.s.}\hspace{2em}\Hb_1: \bu_1\neq\bu_0;
	\end{equation}
	\item[(ii)](Two-sample problem)
	\begin{equation}\label{def:htwo-sampleu}
	\Hb_0: \bu_1=\bu_2\hspace{2em}{\rm v.s.}\hspace{2em}\Hb_1: \bu_1\neq\bu_2.
	\end{equation}
\end{itemize}

We consider the high dimensional setting  that $d/n$ (or $q/n$) does not necessarily go to zero.   These two kinds of  hypotheses are quite general and include most existing studies as special cases.

\subsection{Special cases and applications}

In this section, we provide several special cases of the above general testing problem.

\begin{itemize}
	\item  Matrix-based one-sample test:
	\begin{equation}\label{def:hone-sampleumat}
		\Hb_0: \Ub_1=\Ib_d \hspace{2em}{\rm v.s.}\hspace{2em}\Hb_1: \Ub_1\neq\Ib_d,
	\end{equation}
	where   $\Ub_1$'s entries are estimated  by $U$-statistics and $\Ib_d$ is an identity matrix of size $d$. The hypothesis (\ref{def:hone-sampleumat}) is often used to infer the independence of random variables. This problem plays a fundamental role in many fields including multiple testing (\cite{benjamini1995controlling}), naive Bayes classification (\cite{tibshirani2002diagnosis,fan2008high}), and independent component analysis(\cite{comon1994independent}). Under the Gaussian setting, testing  (\ref{def:hone-sampleumat}) with $\Ub_1$ as  covariance  matrix is well studied both  in low (\cite{roy1957some, nagao1973some,anderson2003introduction}) and high (\cite{ledoit2002some, jiang2004asymptotic,birke2005note,schott2005testing,bai2009corrections,chen2010two,cai2012phase,jiang2013central,cai2013optimal}) dimensions. Moreover,  \cite{jiang2004asymptotic,li2006some,zhou2007asymptotic,liu2008asymptotic,cai2011limiting,cai2012phase,shao2014necessary} consider  the high dimensional independence test under more general distribution. Considering robustness,  rank-based $U$-statistics such as Kendall' s tau and Spearman's rho are introduced to describe the dependence of random variables. As for their definitions and basic theoretical properties, we refer to the book \cite{1990u}. Recently, \cite{han2014distribution,bao2015spectral} study  how to utilyze general $U$-statistics for high dimensional independence test.
	
	\item Matrix-based two-sample  test:
	\begin{equation}\label{def:htwo-sampleumat}
		\Hb_0: \Ub_1=\Ub_2\hspace{2em}{\rm v.s.}\hspace{2em}\Hb_1: \Ub_1\neq\Ub_2,
	\end{equation}
	where $\Ub_1$ and $\Ub_2$ are matrices such that their entries are estimated by $U$-statistics. The hypothesis (\ref{def:htwo-sampleumat}) is often used before the discriminant analysis (\cite{anderson2003introduction,shao2011sparse, cai2011direct, mai2012direct, fan2012road,  mai2013semiparametric, han2013coda}) to simplify the test statistics. For low dimensional  two-sample covariance matrix test, we refer its theoretical properties to \cite{anderson2003introduction}. In recent years, \cite{schott2007test, srivastava2010testing, li2012two, cai2013, chang2015} study how to perform the two-sample covariance matrix test in  high dimensions.  Moreover, \cite{1990u, han2014distribution, bao2015spectral, zhou2015extreme} consider how to use general U-statistics  to replace covariance coefficients. 
	
	\item Means  test:
	\begin{itemize}
		\item[(i)](One-sample problem) 
		\begin{equation}\label{def:hone-sampleu}
		\Hb_0: \bmu_1=\zero \hspace{2em}{\rm v.s.}\hspace{2em}\Hb_1: \bmu_1\neq \zero;
		\end{equation}
		\item[(ii)](Two-sample problem)
		\begin{equation}\label{def:htwo-sampleu}
		\Hb_0: \bmu_1=\bmu_2\hspace{2em}{\rm v.s.}\hspace{2em}\Hb_1: \bmu_1\neq\bmu_2,
		\end{equation}
	\end{itemize}
	where $\bmu_1$ and $\bmu_2$ are  mean vectors of $\bX$ and $\bY$.
	Testing the mean vector is a   special case of (\ref{def:hone-sampleu}) and (\ref{def:htwo-sampleu}). The testing of  mean values is very fundamental. We refer  their low dimensional properties to  \cite{anderson2003introduction}. Recently, a large amount of literature work on high dimensional means test (\cite{bai1996effect,
		srivastava2008test, srivastava2009test, chen2010two, tony2014two, chang2014simulation}). 
\end{itemize}

For (\ref{def:hone-sampleumat}) and  (\ref{def:htwo-sampleumat}),  we
can convert the matrix into a column vector by vectorization to obtain  equivalent tests with the same form as (\ref{def:hone-sampleu}) or (\ref{def:htwo-sampleu}). Therefore,  (\ref{def:hone-sampleumat}) and (\ref{def:htwo-sampleumat})  fall in our  framework.

Testing high dimensional  $U$-statistic parameters also has many important practical applications. For example, in gene selection,  we use it to detect gene differences \citep{ho2008differential,hu2009detecting,hu2010new, cai2013, tony2014two} or rare variants \citep{basu2011comparison,lin2011general,wu2011rare,
	lee2012optimal,pan2014powerful}
between  the diseased and non-diseased population. In finance, we  use  it 
to detect anomalies (\citep{castagna2003instantaneous}) and  test the market efficiency (\cite{fama1993common, fama2015, fama2016international}).

\subsection{Background and existing work}

In the low dimensional setting with $d<n$ fixed,  the Hotelling's $T^2$ test enjoys certain  kind of optimality and has been widely used. To test two-sample mean vectors, the Hotelling's $T^2$ is defined as
\[
\frac{n_1n_2}{n_1+n_2}(\overline{\bX}-\overline{\bY})^\top S_{1,2}^{-1}(\overline{\bX}-\overline{\bY}),
\]
where $\overline{\bX}=n_1^{-1}\sum_{k=1}^{n_1}\bX_{k}$, $\overline{\bY}=n_2^{-1}\sum_{k=1}^{n_2}\bY_{k}$, and 
\[
S_{1,2}=\frac{1}{n_1+n_2-2}\Big(\sum_{k=1}^{n_1}(\bX_k-\overline{\bX})(\bX_k-\overline{\bX})^\top+\sum_{k=1}^{n_2}(\bY_k-\overline{\bY})(\bY_k-\overline{\bY})^\top\Big).
\]
As for the limiting distribution, large and moderate deviations of Hotelling's  $T^2$, we refer to \cite{anderson2003introduction,dembo2006large,liu2013cramer}.

In the high dimensional setting,  many tests have been proposed to test high dimensional vectors and matrices.  These tests fall in two categories:  the $L_2$-type versus $L_\infty$-type tests. Specifically, for (\ref{def:hone-sampleu}) and (\ref{def:htwo-sampleu}), the $L_2$-type tests  are based on $\|\Ab(\bu_1-\bu_0)\|_{2}$ or $\|\Ab(\bu_1-\bu_2)\|_{2}$, and the $L_\infty$-type tests are based on $\|\Ab(\bu_1-\bu_0)\|_\infty$ or $\|\Ab(\bu_1-\bu_2)\|_\infty$ for some operator $\Ab$. 
On one hand, the $L_2$-type tests \citep{bai1996effect,schott2007test,srivastava2009test,srivastava2010testing,
	chen2010two,li2012two}  aim to detect relatively dense signals, as the $L_2$-norm accumulates  small deviations of all entries. On the other hand, the $L_{\infty}$-type tests \citep{cai2013,tony2014two} are more  sensitive to  sparse signals, where some strong perturbations exist on a small number of entries.
\cite{liu2013cramer,cai2013,tony2014two} illustrate that the $L_{\infty}$-type tests  are reasonably more powerful than the $L_2$-type tests and enjoy  certain kind of optimality when the alternative is sparse.
 \subsection{Our contributions}

Theoretically, there is no uniformly most powerful test under different scenarios of the alternatives (\cite{cox1979theoretical}). Therefore, depending on the unknown truth of alternatives, a given and fixed test may or may not be powerful. In this paper, we aim to   develop a broad family of  tests such that
at least one of them is powerful enough in a given situation. We then 
combine these tests to obtain a data-adaptive test
that will maintain high power across a wide range of alternative scenarios.  We develop our family of tests based on a new family of adjusted $L_p$-norms with $p=1,2,\ldots,\infty$, so that there is at least one test in our family is powerful no matter the signal is dense or sparse. The
limiting distribution of the data-adaptive test  is very  complex that we cannot obtain its explicit  form. Therefore, we use the bootstrap method to approximate the limiting distribution, so that we can obtain the critical value and valid $P$-value of the test.  More specifically, to  obtain a better approximation in the high dimensional setting, we adjust $L_p$-norm while building the test statistics. In detail, we introduce it as follows.
\begin{definition}\label{def:s0pnorm}
	 For  $\vb =(v_1,\ldots, v_d)^\top\in\reals^d$, we define
	$\|{\vb}\|_{(s_0,p)}:=\big(\sum_{j=d-s_0+1}^d(v^{(j)})^p\big )^{1/p}$, where $v^{(1)}, v^{(2)},\ldots, v^{(d)}$ are the order statistics of $|v_1|,\ldots, |v_d|$ with $0\le v^{(1)}\le v^{(2)}\le \ldots\le v^{(d)}$.
\end{definition}
By this definition, for any positive integer $s_0$, we have $\|\vb\|_{(s_0,\infty)}=\|\vb\|_{\infty}$, where $\|\vb\|_{\infty}=\max_{j=1,\dots, d} |v_d|$. Moreover, the following proposition shows that $\|\cdot \|_{(s_0,p)}$ is a norm  for any $1\le p\le \infty$.
\begin{prop}\label{proposition:s0pnorm}
For any $1\le p\le \infty$,  $\|\cdot \|_{(s_0,p)}$ is a norm on $\reals^d$.
\end{prop}
The detailed proof of Proposition \ref{proposition:s0pnorm} is in Appendix \ref{proof:s0pnorm} of supplementary materials. 
In this paper, we assume $1\le p\le \infty$ to make $\|\cdot\|_{(s_0,p)}$ a norm. Therefore, similarly to $L_p$-norm, we can call $\|\vb\|_{(s_0,p)}$ the $(s_0,p)$-norm of $\vb$ in this paper. To construct the above family of tests, we use the $(s_0,p)$-norm as the adjusted $L_p$-norm.
More details on this testing procedure is in Section \ref{sec:testProc}.
This paper has four major contributions:
\begin{itemize}
	\item First, we introduce a new family of tests based on the $(s_0,p)$-norm. As is shown in the simulation experiment of Section \ref{section:Exper}, the power of traditional $L_p$-norm based test decreases tremendously (especially for small $p$) as $q\rightarrow\infty$.  The reason is that the $L_p$-norm with small $p$ is easy to accumulate the noise of all  entries. Therefore, we introduce $s_0$ to increase the signal-noise ratio of  test statistics. The introduction of $s_0$ is also crucial in establishing our theoretical results for high dimensional multiplier bootstrap.  Moreover, we  obtain the required scaling between $s_0$, $p$, $q$, and $n$ for the proposed bootstrap methods. 
	
	\item Secondly, as it is hard to obtain the joint distribution of test statistics with various $(s_0,p)$-norm, we use the multiplier bootstrap method to obtain its  asymptotic distribution. In low dimensions, this bootstrap method is well studied for both the sum of random variables (\cite{rubin1981bayesian,lo1987large,parzen1994resampling,barbe2012weighted}) and $U$-statistics (\cite{korolyuk2013theory,arcones1993limit,mason1992rank,huskova1993consistency,huvskova1993generalized,gombay2002rates}).  In high dimensions, the multiplier bootstrap  is also useful for approximating the sum of random vectors (\cite{chernozhukov2013gaussian}).  Motivated by these results,
	we generalize multiplier bootstrap method for $U$-statistics to the high dimensional setting with  theoretical guarantees.
	\item Thirdly, for adapting to the
	possible alternatives, we propose a new approach to combine these
	$(s_0,p)$-norm based tests. Our combined test  automatically chooses the most powerful test within the chosen combination according to the data. Therefore, we call this test the data-adaptive combined test. However, to 
	obtain the $P$-value for the combined test, we originally need a double-loop bootstrap procedure, which suffers from high computational cost. To avoid this, we propose a novel computationally efficient scheme which generates  nonindependent bootstrap samples. We also provide theoretical guarantees for this new bootstrap scheme in the high dimensional setting.
	
	\item Finally, 
	combining the developed theory for the proposed methods and exiting lower bounds in the literature, we  present that our methods are rate-optimal in many settings. 
\end{itemize}

\subsection{Notation}
We set $\|{\vb}\|_p$ as the  $L_p$-norm of a vector $\vb = (v_1,\dots,v_d)^\top\in \reals^d$. We denote the spherical surface in $\reals^d$ by
$\mathbb{S}^{d-1}:=\{\vb\in\reals^d:\|\vb\|_2=1\}$. For two sequences of real numbers $\{a_n\}$ and $\{b_n\}$, we write $a_n = O(b_n)$ if there exists a constant $C$ such that $|a_n| \le C|b_n|$ holds for all sufficiently large $n$, write $a_n = o(b_n)$ if $a_n/b_n \rightarrow 0$, and write $a_n\asymp b_n$ if there exist constants $C \ge c >0$ such that $c|b_n| \le |a_n| \le C|b_n|$ for all sufficiently large $n$. For a sequence of random variables $\{\xi_1,\xi_2,\ldots\}$, we use $\xi_n\rightarrow \xi$ to denote that the sequence $\{\xi_n\}$ converges in probability towards $\xi$ as $n\rightarrow\infty$. For  simplicity,  we also use $\xi_n=o_p(1)$ to denote $\xi_n\rightarrow 0$. 
\subsection{Paper organization}
The rest of this paper is organized as follows. In Section \ref{section:GenProc} we propose the new testing procedures: the individual $(s_0,p)$-norm based test and the data-adaptive combined test. In Section \ref{section:Theo},  we develop a theory to analyze the size and power of the proposed tests. Section \ref{section:Exper} provides some numerical results on simulated data to justify our proposed methods' size and power. In Section \ref{section:Diss}, we discuss some potential future work. Supplementary materials provide both proofs and additional numerical results on both simulated and real data.

\section{Methodology}\label{section:GenProc}
\label{sec:testProc}
This section introduces the $(s_0,p)$-norm based individual tests and the data-adaptive combined test for testing high dimensional $U$-statistic based parameters. We also introduce how to exploit the multiplier bootstrap method to obtain the critical values and 
$P$-values for both individual and combined tests. In the following,  we introduce individual tests based on the $(s_0, p)$-norm in Section \ref{sec:s0ptestProc} and  the data-adaptive  combined test in Section \ref{sec:adTestProc}.


\subsection{ Individual tests based on the $(s_0,p)$-norm}\label{sec:s0ptestProc}
We  introduce the $(s_0,p)$-norm based tests which are basic components of the data-adaptive combined test.
First, we explain the construction motivation in Section \ref{sec:indMot} and describe the test statistics in Section \ref{sec:indBuild}. We then introduce  bootstrapping scheme for $U$-statistics in high dimensions in   Section \ref{sec:bootsProc} and use it to obtain critical values and $P$-values for the proposed tests.

\subsubsection{Motivation of the construction of the $(s_0,p)$-norm}\label{sec:indMot}
We first introduce the motivation of  the proposed individual tests. In the existing literature,  there are two types of tests ($L_2$-type and $L_\infty$-type tests) to test high dimensional vectors or matrices.
The $L_2$-type tests are sensitive to  dense signals and the $L_\infty$-type tests are sensitive to  sparse signals. Therefore, the performance  of these tests
depends on  the pattern of possible alternatives. 
If such pattern is unknown,  it is more desirable to construct a data-adaptive test which is simultaneously powerful under various alternative scenarios. For this, we need to construct a family of versatile tests so that for a given alternative at least one test wiithin the family is powerful. Inspired by the existing  $L_2$-type and $L_\infty$-type tests, we build the test family based on the $L_p$-norm. Importantly, as $p$ increases, the $L_p$-norm puts more weight
on the larger entries  while gradually ignoring the remaining smaller entries. As $p\rightarrow\infty$, we have
$
\|\vb\|_p\rightarrow\|\vb\|_{\infty}$ for any $\vb\in\reals^d$,
where  $\|\vb\|_{\infty}$'s value only depends on the largest entry of $\vb$. More generally, as $p$ increases,
we put more weight on the larger  entries, eventually realizing the $L_\infty$-type test. Hence, by properly choosing  $p$ from the proposed test family, there exists at least one test within the family that is powerful in each alternative situation. 

However, it is problematic to directly use the $L_p$-norm ($p<\infty$)   to construct the test statistics in  high dimensions. For example, when $d/n\not\rightarrow 0$,  Hotelling's $T^2$ test ($L_2$-type) performs
poorly, as the Pearson's sample covariance matrices no longer converge to their population counterparts under the spectral norm (\cite{bai1993limit}). For high dimensional testing problems, we need to adjust the test statistics or make structured assumptions on the population covariance matrix to obtain better  asymptotic distributions of the test statistics. We face the same problem while using $L_p$-norm ($p<\infty$)  to construct the test statistics. Hence, to avoid making unnecessary assumptions on the covariance structure of the random vector, we introduce the $(s_0,p)$-norm to adjust the original $L_p$-norm. As is shown by numerical simulations in Section \ref{section:Exper}, the $L_p$-norm based test with small $p$ has significant power loss when the dimension of the parameter of interest $q\rightarrow\infty$. 
The introduction of $s_0$ can boost the power of $L_p$-norm based test especially for  small $p$. More specifically, when $p$ is small, the 
$L_p$-norm  accumulates noise from all the entries, which leads to significant power loss.  By exploiting the $(s_0,p)$-norm,  we can enhance the signal-noise ratio  for the obtained test statistics.  When $p$ is large, the choice of $s_0$ becomes less critical. In theory, for the bootstrap scheme to work properly under any $1\le p\le \infty$, we require that $s_0^2\log(qn) = O(n^\delta)$ holds for some $0<\delta<1/7$. Therefore, 
 $s_0$ can also go to the infinity as $n\rightarrow\infty$. By simulation, $s_0$ close to $s$, which is the  true unknown number of entries violating $\Hb_0$, is preferable. More details on the choice of $s_0$ are provided in Section \ref{section:Theo} and \ref{section:Exper}.

\subsubsection{The $(s_0,p)$-norm based test statistics}\label{sec:indBuild}
Before presenting the test statistics,  we first introduce the following jackknife variance estimator for the $U$-statistic $\hat{u}_{\gamma, s}$  defined in \eqref{def:uhat} with $\gamma=1,2$ and $s=1,2\ldots,q$. As $m\ge 2$, we define 
\begin{equation}\label{def:vhat}
\hat{v}_{1,s}=m^2n_1^{-1}\sum_{k=1}^{n_1}(Q_{1k,s}-\hat{u}_{1,s})^2,~
\hat{v}_{2,s}=m^2n_2^{-1}\sum_{k=1}^{n_2}(Q_{2k,s}-\hat{u}_{2,s})^2,
\end{equation}
where we set
\begin{equation}\label{def:qalphas}
\begin{array}{rl}
Q_{1k,s}&:=\binom{n_1-1}{m-1}^{-1}\sum\limits_{1\le \ell_1<\cdots< \ell_{m-1}\le n_1\atop \ell_j\neq k, j=1,\ldots,m-1}\Phi_{s}(\bX_{k}, \bX_{\ell_1},\ldots,\bX_{\ell_{m-1}}),\\
Q_{2k,s}&:=\binom{n_2-1}{m-1}^{-1}\sum\limits_{1\le \ell_1<\cdots< \ell_{m-1}\le n_2\atop \ell_j\neq k, j=1,\ldots,m-1}\Phi_{s}(\bY_{k}, \bY_{\ell_1},\ldots,\bY_{\ell_{m-1}}).
\end{array}
\end{equation}
We use $\hat{v}_{\gamma,s}$  to estimate the variance of $\sqrt{n_\gamma}\hat{u}_{\gamma,s}$. Therefore, $\hat{v}_{\gamma,s}/n_\gamma$ is the variance estimator for $\hat{u}_{\gamma,s}$. 
As $m=1$, $\hat{u}_{\gamma,s}$ and $\hat{v}_{\gamma,s}$ are reduced to 
\begin{equation}\label{def:mequalone}
\left\{
\begin{array}{c}
\hat{u}_{1,s}=n_1^{-1}\sum\limits_{k=1}^{n_1}\Phi_s(\bX_k),\\ 
\hat{u}_{2,s}=n_2^{-1}\sum\limits_{k=1}^{n_2}\Phi_s(\bY_k),
\end{array}
\right. 
~
\left\{
\begin{array}{c}
\hat{v}_{1,s}=n_1^{-1}\sum\limits_{k=1}^{n_1}(\Phi_s(\bX_{k})-\hat{u}_{1,s})^2,\\
\hat{v}_{2,s}=n_2^{-1}\sum\limits_{k=1}^{n_2}(\Phi_s(\bX_{k})-\hat{u}_{2,s})^2.
\end{array}
\right.
\end{equation}

After introducing these notations, we  present our $(s_0,p)$-norm  based test statistics. For this, we define $\bW=(W_1,\ldots,W_q)^\top$ and $\bN=(N_1,\ldots,N_q)^\top$, where we set $W_s$ and $N_s$ as
\begin{equation}\label{def:WN}
\begin{aligned}
W_{s}&:={(\hat{u}_{1,s}-u_{0,s})}/{\sqrt{\hat{v}_{1,s}/n_1}},\\
N_{s}&:={(\hat{u}_{1,s}-\hat{u}_{2,s})}/{\sqrt{\hat{v}_{1,s}/n_1+\hat{v}_{2,s}/n_2}}.
\end{aligned}
\end{equation}
For the one-sample problem in (\ref{def:hone-sampleu}), we propose the test statistic
$W_{(s_0,p)}:=\|\bW\|_{(s_0,p)}$. Similarly, for the  two-sample problem in (\ref{def:htwo-sampleu}), we propose the test statistic
$N_{(s_0,p)}:=\|\bN\|_{(s_0,p)}$. Throughout this paper, if not specially specified, we require $1\le p\le \infty$ to make  $\|\cdot\|_{(s_0,p)}$ a norm, which is also required by the theory.

\subsubsection{Bootstrap  procedure for the asymptotic distribution}\label{sec:bootsProc}
In the high dimensional setting, \cite{chernozhukov2013gaussian} introduce the  multiplier bootstrap method for the sum of independent random vectors. In detail, let $\bZ_1, \ldots,\bZ_{n}$ be independent random vectors in $\reals^d$ with 
$\bZ_k=(Z_{k1},\ldots,Z_{kd})^\top$  and $\E[\bZ_k]=\zero$ for $k=1,\ldots,n$. Let
$\varepsilon_1, \varepsilon_2,\ldots,\varepsilon_n$ be independent standard normal random variables, the multiplier bootstrap sample for $\bZ_1, \ldots,\bZ_{n}$ is $\varepsilon_1\bZ_1,\ldots,\varepsilon_n\bZ_n$. The bootstrap sample for the sample mean 
$n^{-1}\sum_{k=1}^n\bZ_k$ then becomes $n^{-1}\sum_{k=1}^n\varepsilon_k\bZ_k$.
To fully utilyze  this result, we use multiplier bootstrap scheme for for high dimensional $U$-statistics.
In detail, we generate independent samples $\varepsilon_{1,1}^b,\ldots,\varepsilon_{1,n_1}^b$ and
$\varepsilon_{2,1}^b,\ldots,\varepsilon_{2,n_2}^b$ from $\varepsilon\sim N(0,1)$ for $b=1,\ldots,B$ and set
\[\label{def:uhatbs}
\begin{array}{rl}
\hat{u}_{1,s}^b&\hspace{-1em}=\binom{n_1}{m}^{-1}\sum\limits_{1\le k_1<\cdots<k_m\le n_1}(\varepsilon^b_{1,k_1}+\cdots
+\varepsilon^b_{1,k_m})\big(\Phi_{s}(\bX_{k_1},\ldots,\bX_{k_m})-\hat{u}_{1,s}\big),\\
\hat{u}_{2,s}^b&\hspace{-1em}=\binom{n_2}{m}^{-1}\sum\limits_{1\le k_1<\cdots<k_m\le n_2}(\varepsilon^b_{2,k_1}+\cdots+\varepsilon^b_{2,k_m})\big(\Phi_{s}(\bY_{k_1},\ldots,\bY_{k_m})-\hat{u}_{2,s}\big).
\end{array}
\]
Correspondingly, we set $\hat{\bu}^b_\gamma:=(\hat{u}^b_{\gamma,1},\ldots,\hat{u}^b_{\gamma,q})^\top$ for $\gamma=1,2$. After introducing $\hat{\bu}^b_\gamma$, we define $
\bW^b=(W^b_1,\ldots,W^b_q)^\top$ and $\bN^b=(N^b_1,\ldots,N^b_q)^\top$, 
where
\begin{equation}\label{def:WNB}
W_{s}^b={\hat{u}^b_{1,s}}/{\sqrt{\hat{v}_{1,s}/n_1}},~
N_{s}^b={(\hat{u}^b_{1,s}-\hat{u}^b_{2,s})}/{\sqrt{\hat{v}_{1,s}/n_1+\hat{v}_{2,s}/n_2}}.
\end{equation}
Bootstrap samples become $\{N^b_{(s_0,p)}\}_{b=1,\ldots,B}$ and $\{W^b_{(s_0,p)}\}_{b=1,\ldots,B}$ with 
\begin{equation}\label{def:WNbs0p}
W_{(s_0,p)}^b=\|\bW^b\|_{(s_0,p)}\hspace{1em}{\rm  and }\hspace{1em}
N_{(s_0,p)}^b=\|\bN^b\|_{(s_0,p)}.
\end{equation}
Given the significance level $\alpha$ and the bootstrap samples, we set the critical values of   $W_{(s_0,p)}$  and $N_{(s_0,p)}$ as 
\begin{align*}
\hat{t}^{W}_{\alpha,(s_0,p)}&=\inf\Big\{t\in \reals: \frac{1}{B}\sum_{b=1}^B\ind \{W^b_{(s_0,p)}\le t\} > 1-\alpha\Big\},~\\
\hat{t}^{N}_{\alpha,(s_0,p)}&=\inf\Big\{t\in \reals: \frac{1}{B}\sum_{b=1}^B\ind \{N^b_{(s_0,p)}\le t\}> 1-\alpha\Big\}.
\end{align*}
Therefore, we obtain the $(s_0,p)$-norm based tests for (\ref{def:hone-sampleu}) and (\ref{def:htwo-sampleu}) as
\begin{align}\label{def:TNW}
T^{W}_{\alpha,(s_0,p)}:=\ind\big\{W_{(s_0,p)}\ge \hat{t}^W_{\alpha,(s_0,p)}\big\},~T^{N}_{\alpha,(s_0,p)}:=\ind\big\{N_{(s_0,p)}\ge \hat{t}^N_{\alpha,(s_0,p)}\big\}
.
\end{align}
We reject $\Hb_0$ of (\ref{def:hone-sampleu}) if and only if 
$T^W_{\alpha,(s_0,p)}=1$ and reject $\Hb_0$ of (\ref{def:htwo-sampleu}) if and only if 
$T^N_{\alpha,(s_0,p)}=1$.
Accordingly, we estimate  $W_{(s_0,p)}$ and $N_{(s_0,p)}$'s oracle  $P$-values $P^W_{(s_0,p)}$ and $P^N_{(s_0,p)}$ by
\begin{equation}\label{def:PhatWN}
\begin{aligned}
\hat{P}^W_{(s_0,p)}&=(B+1)^{-1}{\sum_{b=1}^B\ind\{W^b_{(s_0,p)}> W_{(s_0,p)}\}}\\
\hat{P}^N_{(s_0,p)}&=(B+1)^{-1}{\sum_{b=1}^B\ind\{N^b_{(s_0,p)}> N_{(s_0,p)}\}}.
\end{aligned}
\end{equation}
Therefore, given a significance level $\alpha$,  we reject $\Hb_{0}$ of (\ref{def:hone-sampleu}) if and only if $\hat{P}^W_{(s_0,p)}\le
\alpha$ and reject $\Hb_{0}$ of (\ref{def:htwo-sampleu}) if and only if $\hat{P}^N_{(s_0,p)}\le
\alpha$.

\subsection{Data-adaptive combined  test }\label{sec:adTestProc}
We now introduce the data-adaptive combined test. In Section \ref{sec:Adptest}, we present the test procedure. In Section \ref{sec:doubleboots}, we introduce a  double-loop bootstrap procedure to obtain the $P$-value of the data-adaptive test. To reduce the expensive computation cost of the double-loop bootstrap procedure, in Section \ref{sec:lowcost} we introduce a low cost bootstrap procedure which obtains  nonindependent bootstrap samples.  The theory of this new low cost bootstrap procedure is provided in Section \ref{sec:theroAd}.

\subsubsection{Test statistics}\label{sec:Adptest}
$W_{(s_0,p)}$ and $N_{(s_0,p)}$ have different  powers for different $p$ and alternative scenarios. For example, $W_{(s_0,\infty)}$ and $N_{(s_0,\infty)}$ are  sensitive to large  perturbations on a small number of entries of $\bu_1-\bu_0$ and $\bu_1-\bu_2$. Moreover,  $W_{(s_0,2)}$ and $N_{(s_0,2)}$ are sensitive to small perturbations on  a large number of entries of $\bu_1-\bu_0$ and $\bu_1-\bu_2$. We aim to  combine these tests to construct a data-adaptive test which is simultaneously powerful under different alternatives.

For the one-sample problem, as small  $P$-values of $W_{(s_0,p)}$  lead to the rejection of $\Hb_0$ in (\ref{def:hone-sampleu}), we  construct the data-adaptive test statistic $W_{\rm ad}$ by taking the minimum of $P$-values of all individual tests, i.e.,
\begin{equation}\label{def:Wad}
W_{\rm ad}=\min_{p\in\mathcal{P}}\hat{P}^W_{(s_0,p)},
\end{equation}
where $\mathcal{P}\subset\{1,2,\ldots,\infty\}$ is a candidate  set of $p$.  A bootstrap procedure to  obtain $W_{\rm ad}$ is described in Algorithm \ref{alg:first}.
For the two-sample problem in (\ref{def:htwo-sampleu}), we construct the data-adaptive test statistic $N_{\rm ad}$ as
\begin{equation}\label{def:Nad}
N_{\rm ad}=\min_{p\in\mathcal{P}}\hat{P}^N_{(s_0,p)}.
\end{equation}

Throughout this paper, we require that $\#(\mathcal{P})<\infty$ is a fixed constant, which is also required by the theory and discussed in Section \ref{sec:theroAd}. If the alternative pattern is unknown, we recommend using the balanced $\mathcal{P}$ including both small and large values of $p\in [1,\infty]$.  
For example, $\mathcal{P}=\{1,2,\ldots,5,\infty\}$ is used in the  later simulation experiments. If the alternative pattern is known, we can boost the power of the data-adaptive combined test by choosing $\mathcal{P}$ accordingly. For example, for  possible sparse alternatives, $\mathcal{P}$ should consist of large values of $p$.  

\begin{algorithm}
	\caption{A bootstrap procedure to obtain $W_{\rm ad}$}\label{alg:first}
	\begin{algorithmic}[1]
		\INPUT $\mathcal{X}$.
		\OUTPUT $W^1_{(s_0,p)},\ldots,W^B_{(s_0,p)}$ with $p\in \mathcal{P}$, and $W_{\rm ad}$.
		\Procedure{}{}
		\State 
		$W_{(s_0,p)}=\|{\bm W}\|_{(s_0,p)} \text{ with } {\bm W}=(W_1,\ldots,W_q)^\top \text{ and } W_{s}={(\hat{u}_{1,s}-u_{0,s})}/{\sqrt{\hat{v}_{1,s}/n_1}}.$
		\For{ $b\leftarrow 1$ {\bf to} $B$}
		\State  Sample independent standard normal random variables $\{\varepsilon^b_{1,1},\ldots,\varepsilon^b_{1,n_1}\}$.
		\State $\hat{u}_{1,s}^b\!=\!\binom{n_1}{m}^{-1}\hspace*{-1em}\sum\limits_{1\le k_1<\cdots<k_m\le n_1}\hspace*{-1em}(\varepsilon^b_{1,k_1}+\cdots
		+\varepsilon^b_{1,k_m})\big(\Phi_{s}({\bm X}_{k_1},\ldots,{\bm X}_{k_m})-\hat{u}_{1,s}\big).$
		\State $W_{s}^b={\hat{u}^b_{1,s}}/{\sqrt{\hat{v}_{1,s}/n_1}}$ for $s=1,\ldots,q$.
		\For{$p$ {\bf in} $\mathcal{P}$}
		\State $W^b_{(s_0,p)}=\|{\bm W}^b\|_{(s_0,p)}\text{ with } 
		{\bm W}^b=(W^b_1,\ldots, W^b_q)^\top$.
		\EndFor
		\EndFor
		\State $\hat{P}^W_{(s_0,p)}={\sum_{b=1}^B\ind\{W^b_{(s_0,p)}> W_{(s_0,p)}\}}/{(B+1)}$ for $p\in \mathcal{P}$.
		\State $W_{\rm ad}=\min_{p\in\mathcal{P}}\hat{P}^W_{(s_0,p)}$.
		\EndProcedure
	\end{algorithmic}
\end{algorithm}

\subsubsection{Double-loop bootstrap procedure}\label{sec:doubleboots}
We  present how to  obtain $P$-value of $W_{\rm ad}$. By setting   $F_{W, {\rm ad}}(x)$   as the  distribution function
of $W_{\rm ad}$,  $W_{\rm ad}$'s oracle $P$-value becomes  $F_{W,{\rm ad}}(W_{\rm ad})$. As $F_{W,{\rm ad}}(x)$ is unknown, we need to use the bootstrap method to estimate it, which leads to a double-loop bootstrap procedure. In the outer loop, by  Algorithm \ref{alg:first} we  obtain the bootstrap samples for  $W_{(s_0,p)}$, i.e, $\big\{W^1_{(s_0,p)},\ldots,W^B_{(s_0,p)}\big\}$.  In the inner loop,  for each $b\in\{1,\ldots, B\}$, we use Algorithm 
\ref{alg:second} to obtain bootstrap samples for $W^b_{(s_0,p)}$, i.e., $\big\{W^{b,1}_{(s_0,p)},\ldots,W^{b, L}_{(s_0,p)}\big\}$, and   construct the bootstrap samples  for $W_{\rm ad}$ as
\[
W_{\rm ad}^b=\min_{p\in\mathcal{P}}\frac{\sum_{\ell=1}^L\ind\{W^{b,\ell}_{(s_0,p)}> W^b_{(s_0,p)}\}}{L+1} \hspace{2em}\text{ for }\hspace{2em} b=1,\ldots,B.
\] 
With the bootstrap samples, we can estimate the oracle $P$-value of $W_{\rm ad}$ by 
\[
\frac{1}{B+1}\Bigg(\Big(\sum_{b=1}^{B}\ind\{W^b_{\rm ad}\le W_{\rm ad}\}\Big)+1\Bigg).
\]
Figure \ref{fig:flowchart} illustrates this  double-loop bootstrap method. By this double-loop bootstrap procedure, to guarantee the independence of  $W^1_{\rm ad},\ldots,W^B_{\rm ad}$,  we totally need  $LB+B$ samples from (\ref{def:uhatbs}), which  is computationally expensive when $L$ and $B$ are large.

\begin{algorithm}
	\caption{A double-loop bootstrap procedure to obtain bootstrap samples of $W_{\rm ad}$}\label{alg:second}
	\begin{algorithmic}[1]
		\INPUT $\mathcal{X}$ and $W^1_{(s_0,p)},\ldots,W^B_{(s_0,p)}$ for $p\in \mathcal{P}$.
		\OUTPUT $W_{\rm ad}^1,\ldots,W_{\rm ad}^B$.
		\Procedure{}{}
		\For{$b \leftarrow 1$ {\bf to} $B$ }
		\For{$\ell \leftarrow 1$ {\bf to} $L$ }
		\State  Sample independent standard normal random variables $\{\varepsilon^{b,\ell}_{1,1},\ldots,\varepsilon^{b,\ell}_{1,n_1}\}$.
		\State
		$\hat{u}_{1,s}^{b,\ell}\!=\!\binom{n_1}{m}^{-1}\hspace*{-1em}\sum\limits_{1\le k_1<\cdots<k_m\le n_1}\hspace*{-1em}(\varepsilon^{b,\ell}_{1,k_1}+\cdots
		+\varepsilon^{b,\ell}_{1,k_m})\big(\Phi_{s}({\bm X}_{k_1},\ldots,{\bm X}_{k_m})-\hat{u}_{1,s}\big).$
		\State $W_{s}^{b,\ell}={\hat{u}^{b,\ell}_{1,s}}/{\sqrt{\hat{v}_{1,s}/n_1}}$ for $s=1,\ldots, q$.
		\For{p {\bf in } $\mathcal{P}$}
		\State $W^{b,\ell}_{(s_0,p)}=\|{\bm W}^{b,\ell}\|_{(s_0,p)}\text{ with } 
		{\bm W}^{b,\ell}=(W^{b,\ell}_1,\ldots, W^{b,\ell}_q)^\top$.
		
		\EndFor
		\EndFor
		\State $\hat{P}_{W^b,(s_0,p)}={\sum_{\ell=1}^L\ind\{W^{b,\ell}_{(s_0,p)}> W^b_{(s_0,p)}\}}/{(L+1)}$ for $p\in\mathcal{P}$.
		\State $W^b_{\rm ad}=\min_{p\in\mathcal{P}}\hat{P}_{W^b,(s_0,p)}$
		\EndFor
		\EndProcedure
	\end{algorithmic}
\end{algorithm}
\begin{figure}[htp!]
	\begin{center}
		\includegraphics[width=0.9\textwidth,angle=0]{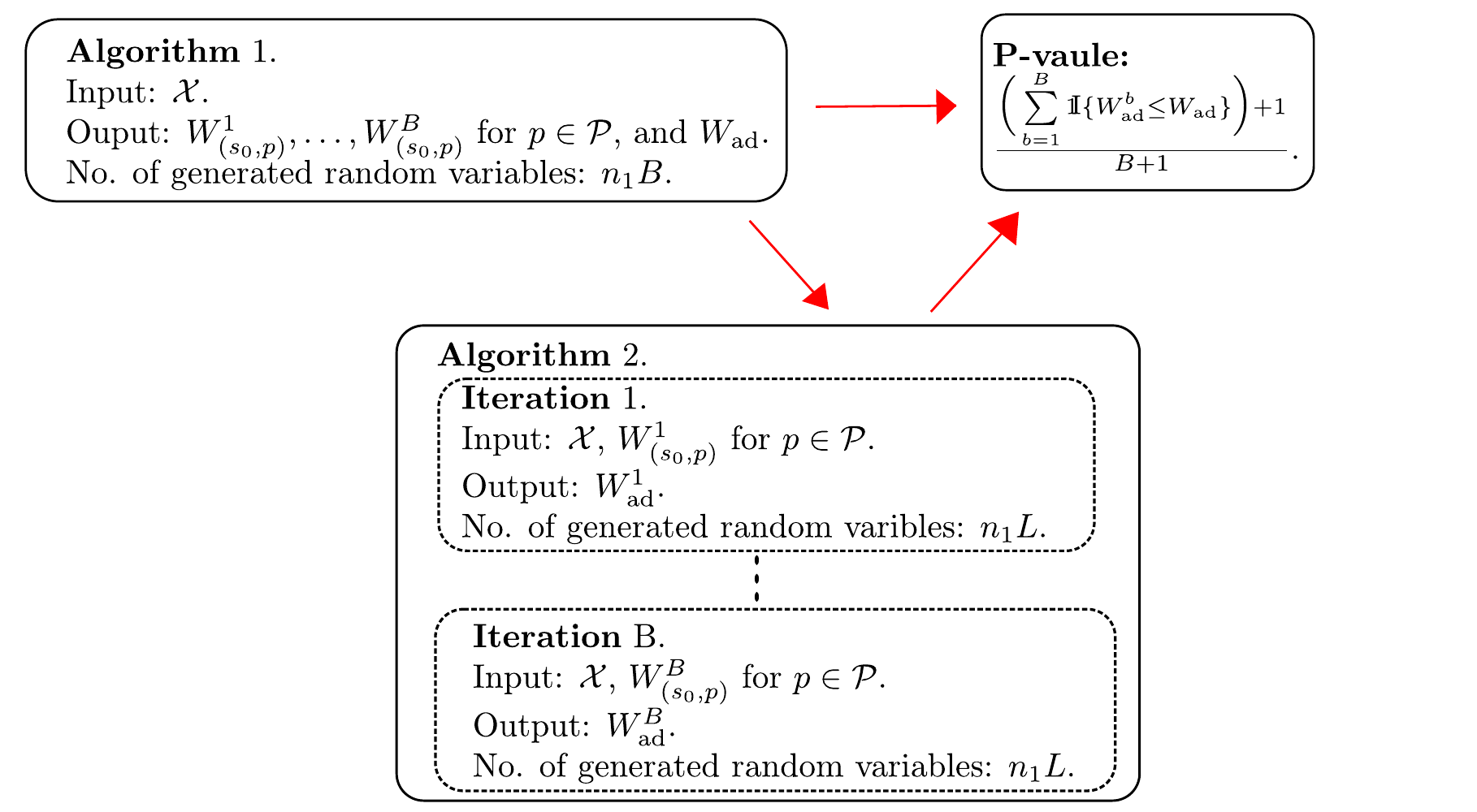}
	\end{center}
	\caption{\small Flowchart for the double-loop bootstrap procedure and total number of generated standard normal random  variables. } \label{fig:flowchart}
\end{figure}

\subsubsection{A low cost bootstrap procedure}\label{sec:lowcost}

To handle the computational bottleneck of the double-loop bootstrap, we propose to  replace Algorithm \ref{alg:second} with  Algorithm \ref{alg:adjust}, which is computationally more efficient but obtains  nonindependent bootstrap samples for $W_{\rm ad}$, denoted as $\{W^1_{\rm ad'},\ldots,W^B_{\rm ad'}\}$.

In detail, in Algorithm \ref{alg:first} by (\ref{def:uhatbs}), (\ref{def:WNB}), and (\ref{def:WNbs0p})  we generate  bootstrap samples for $W_{(s_0,p)}$, i.e., $W_{(s_0,p)}^1,\ldots,W_{(s_0,p)}^B$. To avoid the double-loop bootstrap procedure, we  need to more effectively  utilize the generated bootstrap samples $W^1_{(s_0,p)},\ldots,W^B_{(s_0,p)}$. For this, we  set 
\[
\hat{P}^{b,W}_{(s_0,p)}=\frac{\sum_{b_1\ne b} \ind\{W^{b_1}_{(s_0,p)}> W^{b}_{(s_0,p)}\}}{B}~{\rm for}~b=1,\ldots,B\hspace{1em}{\rm and}\hspace{1em}p\in\mathcal{P}.
\]
We use $W^b_{\rm ad'}=\min_{p\in\mathcal{P}
}\hat{P}^{b,W}_{(s_0,p)}$  as the bootstrap sample for $W_{\rm ad}$, and
estimate the oracle $P$-value  by
\begin{equation}\label{def:hatPWad}
\hat{P}_{\rm ad}^W = \frac{\big(\sum_{b=1}^{B}\ind\{W^b_{\rm ad'}\le W_{\rm ad}\}\big)+1}{B+1}.
\end{equation}
The samples  $W^1_{\rm ad'},\ldots,W^B_{\rm ad'}$ are nonindependent. However, we can prove that they are  asymptotically independent  as $n_1,B\rightarrow\infty$, which plays a pivotal role in proving the consistency of
$\hat{P}^W_{\rm ad}$.

Figure \ref{fig:flowchartAdj} illustrates 
the process of the low cost bootstrap procedure. To obtain the $P$-value of $W_{\rm ad}$, we don't need to generate new bootstrap samples. In total, to perform the data-adaptive test we only need to generate $B$ bootstrap samples from (\ref{def:uhatbs}).

We similarly deal with the two-sample problem. By generating  bootstrap samples for $N_{(s_0,p)}$, i.e., $N^1_{(s_0,p)},\ldots,N^B_{(s_0,p)}$ and setting
\begin{equation}\label{def:hatbns0p}
\hat{P}^{b,N}_{(s_0,p)}=\frac{\sum_{b_1\ne b} \ind\{N^{b_1}_{(s_0,p)}> N^{b}_{(s_0,p)}\}}{B}~{\rm for}~b=1,\ldots,B~\text{and}~p\in\mathcal{P},
\end{equation}
we use $N^b_{\rm ad'}=\min_{p\in\mathcal{P}
}\hat{P}^{b,N}_{(s_0,p)}$ as the bootstrap sample of $N_{\rm ad}$. Therefore, we can similarly  
estimate the oracle $P$-value of $N_{\rm ad}$  by
\begin{equation}\label{def:hatPNad}
\hat{P}_{\rm ad}^N= \frac{\big(\sum_{b=1}^{B}\ind\{N^b_{\rm ad'}\le N_{\rm ad}\}\big)+1}{B+1}.
\end{equation}

With the estimated $P$-values of the data-adaptive tests $W_{\rm ad}$ and $N_{\rm ad}$, given significance level $\alpha$, we reject $\Hb_0$ of (\ref{def:hone-sampleu}) if and only if $\hat{P}^W_{\rm ad}\le\alpha$ and reject $\Hb_0$ of (\ref{def:htwo-sampleu}) if and only if $\hat{P}^N_{\rm ad}\le\alpha$. Therefore, we set 
\begin{equation}\label{def:TWNAd}
T^W_{\rm ad}=\ind\{\hat{P}^W_{\rm ad}\le\alpha\}\hspace{2em}{\rm and}\hspace{2em}T^N_{\rm ad}=\ind\{\hat{P}^N_{\rm ad}\le\alpha\}.
\end{equation}

\begin{algorithm}[ht]
	\caption{A low cost  bootstrap procedure}\label{alg:adjust}
	\begin{algorithmic}[1]
		\INPUT $\mathcal{X}$ and $W^1_{(s_0,p)},\ldots,W^B_{(s_0,p)}$ for $p\in \mathcal{P}$.
		\OUTPUT $W_{\rm ad'}^1,\ldots, W_{\rm ad'}^B$.
		\Procedure{}{}
		\For{$b \leftarrow 1$ {\bf to} $B$}
		\For{ $p$ {\bf in} $\mathcal{P}$}
		\State $\hat{P}^{b,W}_{(s_0,p)}={\sum_{b_1\ne b} \ind\{W^{b_1}_{(s_0,p)}> W^{b}_{(s_0,p)}\}}/{B}$
		\EndFor
		\State $W^b_{\rm ad'}=\min_{p\in\mathcal{P}
		}\hat{P}^{b,W}_{(s_0,p)}$.
		\EndFor
		\EndProcedure
	\end{algorithmic}
\end{algorithm}

\begin{figure}[htp!]
	\begin{center}
		\includegraphics[width=0.9\textwidth,angle=0]{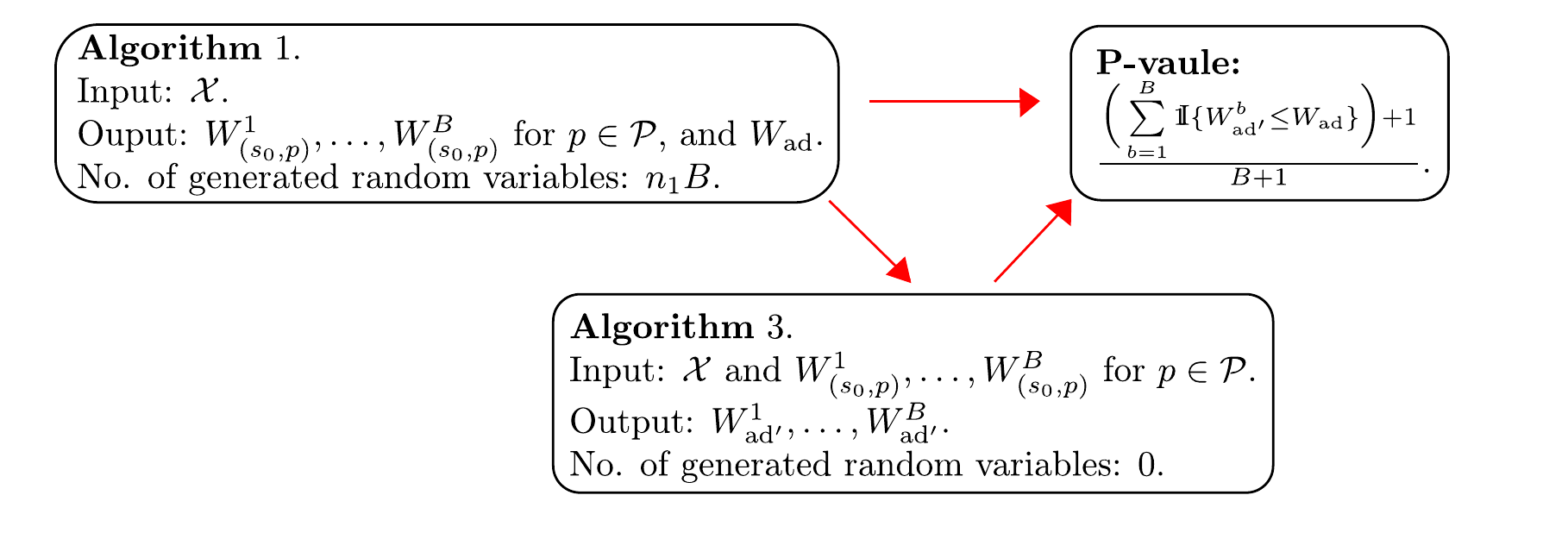}
		
	\end{center}
	\caption{\small Flowchart for the low cost  bootstrap procedure with low computation cost and total number of generated standard normal random  variables. } \label{fig:flowchartAdj}
\end{figure}

\begin{remark}
	To construct test statistics $W_{(s_0,p)}$  and $N_{(s_0,p)}$,  we normalize $\hat{u}_{1,s}-u_{1,s}$  and $\hat{u}_{1,s}-\hat{u}_{2,s}$  by dividing their standard deviation estimators. 
	If we assume that $U$-statistics have the same variance under the null hypothesis (homogeneity  assumption),  we can build $W_{(s_0,p)}$ and $N_{(s_0,p)}$ without the normalization to avoid introducing unnecessary estimation error. Therefore, 
	$W_s$ and $N_s$ become
	\[
	W_s:=\hat{u}_{1,s}-u_{0,s}\hspace{2em}{\rm and}\hspace{2em}
	N_s:=\hat{u}_{1,s}-\hat{u}_{2,s} .
	\] 
	For the same reason, we set $W_s^b=\hat{u}^b_{1,s}$ and 
	$N_s^b=\hat{u}^b_{1,s}-\hat{u}^b_{2,s}$ when performing bootstrap procedure of Sections \ref{sec:bootsProc} and \ref{sec:adTestProc}.  As the proof is similar for the test
	statistics without normalization, in Section \ref{section:Theo}  we only analyze
	the theoretical properties of the test statistics with  normalization.
\end{remark}

\section{Theoretical properties}\label{section:Theo}
In this section, we discuss the theoretical properties of the proposed testing methods including the $(s_0,p)$-norm based  test and  data-adaptive combined test. We first introduce several  assumptions  in Section \ref{sec:assumptions}. We then analyze the asymptotic size and power of the $(s_0,p)$-norm based test in Section \ref{sec:theorys0lp}. At last, we analyze the data-adaptive combined test in Section \ref{sec:theroAd}. 

\subsection{Assumptions}\label{sec:assumptions}
Before presenting the theoretical properties, we introduce the assumptions that are needed in this paper.   We  also explain the intuitions of these assumptions. Throughout this paper, for the two-sample problem, we assume  $n_1\asymp n_2\asymp n:=\max(n_1,n_2)$, which  means that $n_1, n_2,$ and $n$ are of the same order. We then introduce some other assumptions.
Assumption {\bf (A)}  characterizes the scaling of $s_0$, $q$, and $n$. 
Assumptions {\bf (E)},  {\bf (M1)} and {\bf (M2)} specify the requirements of the kernel functions. In detail, we introduce Assumption 
{\bf (A)} as follows.
\begin{itemize}
	\item{\bf (A)} For the one-sample problem in (\ref{def:hone-sampleu}), we assume that there is some   $0<\delta<1/7$ such that  $s_0^2\log (q)= O(n_1^\delta)$ holds.  For the two-sample problem in (\ref{def:htwo-sampleu}), we similarly assume that there is some $0<\delta<1/7$ such that $s_0^{2}\log q= O(n^\delta)$ holds.
\end{itemize}
 Assumptions {\bf (A)} also allows $q$ and $s_0$ to go to the infinity, as long as $s_0^2\log(qn)= o(n^{\delta})$ holds with some $0<\delta<1/7$.

We then introduce the assumptions on the kernel functions of  the $U$-statistics.  For  $\xb,\xb_1,\ldots,\xb_m\in\reals^d$, define
\[
\begin{aligned}
\bPsi(\xb_1,\ldots,\xb_m)&:=\big(\Psi_1(\xb_1,\ldots,\xb_m),\ldots,\Psi_q(\xb_1,\ldots,\xb_m)\big)^\top\\
\bh(\xb)&:=\big(h_1(\xb),\ldots,h_q(\xb)\big)^\top,
\end{aligned}
\]
where $\Psi_s$ and $h_s$ are
 \begin{equation}\label{def:PsiAnd}
 \begin{aligned}
\Psi_s(\bX_{k_1},\ldots,\bX_{k_m}) &= \Phi_s(\bX_{k_1},\ldots,\bX_{k_m})-u_{1,s}\\
h_s(\bX_{k})&=\E[\Psi_s(\bX_{k_1},\ldots,\bX_{k_m})|\bX_k].
\end{aligned}
\end{equation} Also, set $\mathcal{V}_{s_0}:=\{\vb\in \mathbb{S}^{q-1}: \|\vb\|_0\le s_0\}$.
With these introduced notations, by setting $0<K,b<\infty$ as some positive constants,
we are now ready to state Assumptions {\bf (E)}, {\bf (M1)}, and {\bf (M2)}.
\begin{itemize}
	\item{\bf (E)} For different indexes $0<i_1,\ldots, i_m<n_1$ and $0<j_1,\ldots, j_m<n_2$, we require
	\begin{align*}
	\max_{1\le s\le q}&\E\big[\exp\big(|\Psi_s(\bX_{i_1}, \ldots, \bX_{i_m})|/K\big)\big]\le 2,\\
	\max_{1\le s\le q}&\E\big[\exp\big(|\Psi_s(\bY_{j_1},\ldots,\bY_{j_m})|/K\big)\big]\le 2.
	\end{align*}
	\item{\bf (M1)}  $\E[|\vb^\top \bh(\bX)|^2]\ge b$ and $\E[|\vb^\top\bh(\bY)|^2]\ge b$ hold for any $\vb\in\mathcal{V}_{s_0}$. 
	
	\item{\bf (M2)} For   $\ell=1,2$, we require 
	\[
	\max_{1\le s\le q}\E[|h_s(\bX)|^{2+\ell}]\le K^\ell,~
	\max_{1\le s\le q}\E[|h_s(\bY)|^{2+\ell}]\le K^\ell.
	\]

\end{itemize} 
Assumption {\bf (E)} requires that  $\Psi_s(\bX_{i_1},\ldots,\bX_{i_m})$ and $\Psi_s(\bY_{j_1},\ldots,\bY_{j_s})$ follow the sub-exponential distribution.  Especially,  bounded $\Psi_s$  including useful rank-based $U$-statistics such as Kendall's tau and Spearman's rho satisfy this condition. Assumption {\bf (M1)} excludes degenerate $U$-statistics. Moreover, it also requires that the inner product of  $\bh(\bX)$ (or $\bh(\bY)$) and any $\vb \in \mathcal{V}_{s_0}$  is not degenerated. The distribution assumptions {\bf (E)}, {\bf (M1)}, and {\bf(M2)} are useful for applying high-dimensional
central limiting theorem (CLT) in Lemma \ref{lemma:CCK2}. These assumptions are also justified by \cite{chernozhukov2014central}.

\subsection{Theoretical properties of  $(s_0,p)$-norm based test statistics}\label{sec:theorys0lp}
After introducing the assumptions in Section \ref{sec:assumptions}, we now state the theoretical properties of the $(s_0,p)$-norm based test. Firstly, we consider the asymptotic size. The following theorem justifies the multiplier bootstrap for $W_{(s_0,p)}$ and $N_{(s_0,p)}$, which is crucial for the size control.
\begin{theorem}\label{therom:CoreWN1+} Suppose all assumptions in Section \ref{sec:assumptions} hold. Under  $\Hb_0$ of (\ref{def:hone-sampleu}),  we have 
	\begin{equation}\label{ineq:CoreW1+}
	\sup_{z\in(0,\infty)}\Big|\P(W_{(s_0,p)}\le z)-\P(W^b_{(s_0,p)}\le z|\mathcal{X})\Big|=o_p(1),~\text{as}~n_1\rightarrow\infty.
	\end{equation}
	Similarly, under $\Hb_0$  of (\ref{def:htwo-sampleu}) we have
	\begin{equation}\label{ineq:CoreN1+}
	\sup_{z\in(0,\infty)}\Big|\P(N_{(s_0,p)}\le z)-\P(N^b_{(s_0,p)}\le z|\mathcal{X}, \mathcal{Y})\Big|=o_p(1),~\text{as}~n\rightarrow\infty.
	\end{equation}
\end{theorem}
\begin{proof}
	The proof of  (\ref{ineq:CoreW1+}) is similar to that of (\ref{ineq:CoreN1+}). For simplicity,
	we only present the proof of (\ref{ineq:CoreN1+}), which consists of three steps. We first analyze the approximate distribution of $\bN$. We then obtain the distribution of the bootstrap sample $\bN^b$ given $\mathcal{X}$ and $\mathcal{Y}$. At last, we analyze the approximation error between $\bN$ and $\bN^b|\mathcal{X},\mathcal{Y}$ to yield (\ref{ineq:CoreN1+}). We only sketch the proof here. More detailed proof  is presented in Appendix \ref{proof:theoremCoreNW1+} of supplementary materials. 
	
	{\bf Step (i) (Sketch).} In this step, we aim to obtain the approximate distribution of $\bN$ under the null hypothesis.
	Under the null hypothesis we have $u_{1,s}=u_{2,s}$. Therefore, we rewrite  $N_{s}$ as
	\[
	N_s = (\tilde{u}_{1,s}-\tilde{u}_{2,s})/{\sqrt{\hat{v}_{1,s}/n_1
			+\hat{v}_{2,s}/n_2}},
	\]
	where $\tilde{u}_{\gamma,s}:=\hat{u}_{\gamma,s}-u_{\gamma,s}$ is the centered version of
	$\hat{u}_{\gamma,s}$. As $\tilde{u}_{\gamma,s}$ is also a  $U$-statistic, by the Hoeffding's decomposition we  can approximate  $\tilde{u}_{\gamma,s}$ by a sum of independent random variables. In detail, we  use $(m/n_1)\sum_{k=1}^{n_1}h_s(\bX_k)$ and  $(m/n_2)\sum_{k=1}^{n_2}h_s(\bY_k)$ to approximate $\tilde{u}_{1,s}$ and $\tilde{u}_{2,s}$. By setting 
	\begin{equation}\label{def:sigmast}
	\sigma_{1,st}=\E\big(h_s(\bX)h_t(\bX)\big)\hspace{2em}{\rm and}\hspace{2em}\sigma_{2,st}=\E\big(h_s(\bY)h_t(\bY)\big)
	\end{equation}
	for $1\le s,t\le q$, as $n\rightarrow\infty$ we have 
	$\hat{v}_{\gamma,s}\rightarrow m^2\sigma_{\gamma,ss}$, which motivates us to define
	\begin{equation}\label{def:Hs}
	H_s^N={\Big(\dfrac{1}{n_1}\sum\limits_{k=1}^{n_1}h_s(\bX_k)-\dfrac{1}{n_2}\sum\limits_{k=1}^{n_2} h_s(\bY_k)\Big)} \Big/ {\sqrt{\sigma_{1,ss}/n_1+\sigma_{2,ss}/n_2}}.
	\end{equation}
	By setting  $\bH^N=(H_1^N,\ldots,H_q^N)^\top$, we use $\bH^N$ as an approximation of $\bN$.
	However, we  don't know the exact distribution of $\bH^N$. As $\bH^N$ is a sum of independent random vectors with zero mean, by the central limit theorem we can use a normal random vector
	to further approximate $\bH^N$.

	Let $\bG^N$ be a Gaussian random vector with the same mean vector and covariance  matrix as $\bH^N$.  By setting $\bSigma_{1}:=(\sigma_{1,st}), \bSigma_{2}:=(\sigma_{2,st})\in \reals^{q\times q}$, we have 
	\begin{equation}\label{def:GNR12}
	\bG^N\sim N(\zero, \Rb_{12})\hspace{2em}{\rm with}\hspace{2em}
	\Rb_{12}:=\Db_{12}^{-1/2}\bSigma_{12}\Db_{12}^{-1/2},
	\end{equation}
	where we set
	\begin{equation}\label{def:SigmaD12}
	\bSigma_{12}=\bSigma_1/n_1+\bSigma_2/n_2\hspace{2em}{\rm and}\hspace{2em}\Db_{12}={\rm Diag}(\bSigma_{12}).
	\end{equation}
	We then use the distribution of  $\bG^N$ to approximate that of $\bN$.
	
	{\bf Step (ii) (Sketch).} In this step, we aim to obtain the  distribution of  $\bN^b|\mathcal{X},\mathcal{Y}$. We rewrite $\hat{u}^b_{1,s}$ and $\hat{u}^b_{2,s}$  in (\ref{def:uhatbs}) as
	\begin{equation}\label{def:uhatb2}
	\hat{u}^b_{1,s}=\frac{m}{n_1}\sum_{k=1}^{n_1}(Q_{1k,s}-\hat{u}_{1,s})\varepsilon^b_{1,k},~\hat{u}^b_{2,s}=\frac{m}{n_2}\sum_{k=1}^{n_2}(Q_{2k,s}-\hat{u}_{2,s})\varepsilon^b_{2,k},
	\end{equation}
	where $Q_{1k,s}$ and $Q_{2k,s}$ are defined in (\ref{def:qalphas}). Considering that $\varepsilon^b_{\gamma,1},\ldots,\varepsilon^b_{\gamma,n_1}$ are independent standard normal random variables,  by (\ref{def:uhatb2}) we have $\hat{\bu}^b_{\gamma}:=(\hat{u}^b_{\gamma,1},\ldots,\hat{u}^b_{\gamma, q})|\mathcal{X},\mathcal{Y}\sim N(\zero, m^2\hat{\bSigma}_{\gamma}/n_\gamma)$ with 
	$\hat{\bSigma}_{\gamma}:=(\hat{\sigma}_{\gamma,st})\in\reals^{q\times q}$ and
	\begin{equation}\label{def:sigmahat}
	\hat{\sigma}_{\gamma,st}=\frac{1}{n_\gamma}\sum_{k=1}^{n_1}(Q_{\gamma k,s}-\hat{u}_{\gamma,s})(Q_{\gamma k,t}-\hat{u}_{\gamma,t}),
	\end{equation}
	for $\gamma=1,2$. Apparently, by the definition of $\hat{v}_{\gamma,s}$ in (\ref{def:vhat}) we have $\hat{v}_{\gamma,s}=
	m^2\hat{\sigma}_{\gamma,ss}$. By setting
	\[
	\hat{\bSigma}_{12}=\hat{\bSigma}_{1}/n_1+\hat{\bSigma}_{2}/n_2\hspace{1em}{\rm and}\hspace{1em}\hat{\Db}_{12}={\rm Diag}(\hat{\bSigma}_{12}),
	\]
	given $\mathcal{X}$ and $\mathcal{Y}$  we have
	\begin{equation}\label{def:NbhatR12}
	\bN^b=m^{-1}\hat{\Db}_{12}^{-1/2} (\hat{\bu}^b_1-\hat{\bu}^b_2)\sim N(\zero,\hat{\Rb}_{12}),
	\end{equation}
	where we set $\hat{\Rb}_{12}=\hat{\Db}_{12}^{-1/2}\hat{\bSigma}_{12}\hat{\Db}_{12}^{-1/2}$.
	
	{\bf Step (iii) (Sketch).} In this  step, we aim to obtain the approximation error between $\bN$ and $\bN^b|\mathcal{X},\mathcal{Y}$.  For this, we  analyze the estimation error between $\hat{\Rb}_{12}$ and  $\Rb_{12}$.  We then combine results from {\bf Steps (i)} and {\bf  (ii)} to finish the proof of  (\ref{ineq:CoreN1+}). The detailed proof is in Appendix \ref{proof:theoremCoreNW1+} of supplementary materials.
\end{proof} 

\begin{remark}
	Assumption {\bf (A)}  requires that $s_0^\zeta \log(q) = O(n^\delta)$ holds with $\zeta=2$ and $0<\delta<1/7$. However, $\zeta=2$ is not optimal for each individual  $p$. By the proof of Theorem \ref{therom:CoreWN1+}, $\zeta$ depends on the facet number of a polytope to approximate 
	$B_{(s_0,p)}(x) =\{\vb\in \reals^q : \|\vb\|_{(s_0,p)}\le x  \}$.  If $p=1$, $B_{(s_0,p)}(x)$ itself is a polytope, which makes $\zeta=1$ is enough for obtaining (\ref{ineq:CoreW1+}) and (\ref{ineq:CoreN1+}). Similarly, If $p=\infty$, $\zeta=0$ is sufficient. To make (\ref{ineq:CoreW1+}) and (\ref{ineq:CoreN1+}) hold for any $p\in[1,\infty]$, by Lemma \ref{lemma:approximstCovexSet} in Appendix A, we set $\zeta=2$ in Theorem \ref{therom:CoreWN1+}.
\end{remark}

As an implication of Theorems \ref{therom:CoreWN1+}, the following corollary shows that under mild moment conditions on the kernel functions of $U$-statistics, by using the multiplier bootstrap introduced in  Section \ref{sec:bootsProc}, the size of $(s_0,p)$-norm based test  is asymptotically $\alpha$, as desired.
\begin{corollary} \label{corollary:size}
	Suppose all assumptions in Section \ref{sec:assumptions} hold. For the one-sample problem in (\ref{def:hone-sampleu}), 
	under $\Hb_0$ of (\ref{def:hone-sampleu}) we have 
	\begin{equation}\label{size_one_sample}
	\P_{\Hb_0}\big(T^W_{\alpha,(s_0,p)}=1\big)\rightarrow \alpha~~\text{and}~~
	\hat{P}^W_{(s_0,p)}-P^W_{(s_0,p)}\rightarrow0, ~\text{as $n_1$, $B$}\rightarrow\infty.
	\end{equation}
	as $n_1,B\rightarrow\infty$.
	Similarly, for the two-sample problem in (\ref{def:htwo-sampleu}), under $\Hb_0$ of (\ref{def:htwo-sampleu}) we have 
	\begin{equation}\label{size_two_sample}
	\P_{\Hb_0}\big(T^N_{\alpha,(s_0,p)}=1\big)\rightarrow\alpha~~{\rm and}~
	\hat{P}^N_{(s_0,p)}-P^N_{(s_0,p)}\rightarrow0,~\text{as $n$, $B$}\rightarrow\infty.
	\end{equation}
\end{corollary}
The detailed proof of Corollary \ref{corollary:size} is in Appendix \ref{proof:corollary:size}  of supplementary materials. After analyzing the asymptotic size of the $(s_0,p)$-norm based test,  we now turn to the analysis of its power.
For this, we need the following notations: 
$\bD_1=(D_{1,1},\ldots,D_{1,q})^\top$ and $\bD_2=(D_{2,1},\ldots,D_{2,q})^\top$ with
\begin{equation}\label{def:powerAssumption}
\begin{aligned}
D_{1,s}&=|u_{1,s}-u_{0,s}|/\sqrt{m^2\sigma_{1,ss}/n_1},\\
D_{2,s}&=|u_{1,s}-u_{2,s}|/\sqrt{m^2\sigma_{1,ss}/n_1+
	m^2\sigma_{2,ss}/n_2},
\end{aligned}
\end{equation}
where $\sigma_{\gamma,ss}$ is defined in (\ref{def:sigmast}). We  need new Assumption {\bf (A)$'$} to describe the scaling between $s_0$, $q$, and $n$ for test statistics $W_{(s_0,p)}$ and $N_{(s_0,p)}$ to reject with overwhelming probability under the alternative.

\begin{itemize}
	\item{\bf (A)$'$} 
	For the one-sample problem in (\ref{def:hone-sampleu}), we  assume 
	$\log q = o(n_1^{1/3})$ and $ n_1 = O( q^{\delta_1})$ with  some $\delta_1>0$, as $n_1, q\rightarrow\infty$. For  the two-sample problem in (\ref{def:htwo-sampleu}),  we assume  $\log q = o(n^{1/3})$ and $n = O(q^{\delta_1})$ with some $\delta_1 > 0$, as
	$n, q\rightarrow\infty$. Moreover, we also assume that there is a constant $\delta_2 >0 $ such that $s_0 = O(\log^{\delta_2}(q))$ holds for both problems.
\end{itemize}
After the introduction of Assumption {\bf (A)$'$},   we then state the  theorem that characterizes the power of $W_{(s_0,p)}$ and $N_{(s_0,p)}$. 

\begin{theorem}\label{theorem:power}
	Suppose Assumptions {\bf (A)$'$}, {\bf (E)}, {\bf (M1)}, and {\bf (M2)} hold.  For the one-sample problem in (\ref{def:hone-sampleu}),  we assume $\varepsilon_{n_1}=o(1)$ with  $\varepsilon_{n_1}\sqrt{\log q}\rightarrow\infty$ as $n_1,q\rightarrow\infty$. If $\Hb_1$ of (\ref{def:hone-sampleu}) holds with 
	\begin{equation}\label{assump:powerOnesample}
	\|\bD_1\|_{(s_0,p)}\ge s_0 (1+\varepsilon_{n_1})
	\big(\sqrt{2\log q}+\sqrt{2\log (1/\alpha)}\big),
	\end{equation}
	we have $
	\P_{\Hb_1}\big(T^W_{(s_0,p)}=1\big)\rightarrow1$ as $n_1,q,B\rightarrow\infty.$
	Similarly, for the two-sample problem in (\ref{def:htwo-sampleu})  we assume $\varepsilon_n=o(1)$, and $\varepsilon_n\sqrt{\log q}\rightarrow\infty$ as $n,q\rightarrow\infty$. If $\Hb_1$ of (\ref{def:htwo-sampleu}) holds with 
	\begin{equation}\label{assump:powerTwosample}
	\|\bD_2\|_{(s_0,p)}\ge s_0 (1+\varepsilon_n)
	\big(\sqrt{2\log q}+\sqrt{2\log (1/\alpha)}\big),
	\end{equation}
	we have $
	\P_{\Hb_1}\big(T^N_{(s_0,p)}=1\big)\rightarrow1$ as $n,q,B\rightarrow\infty.
	$
	
\end{theorem}
The detailed proof of Theorem \ref{theorem:power} is presented in Appendix \ref{proof:theorem:power}  of supplementary materials.
The scaling of $q$ and $n$ in Theorem \ref{theorem:power} is weaker than  Assumption {\bf (A)}, allowing larger $q$ for proposed tests to correctly reject the null hypothesis. Moreover, by the proof of Theorem \ref{theorem:power}, for $m=1$, we can  further relax the conditions $\log q = o(n_1^{1/3})$ and $\log q = o(n^{1/3})$ by $\log q = o(n_1^{1/2})$ and $\log q = o(n^{1/2})$ in Assumption {\bf (A)$'$}. 

\subsection{Theoretical properties of $W_{\rm ad}$ and $N_{\rm ad}$}\label{sec:theroAd}
In Section \ref{sec:adTestProc}, we introduce the data-adaptive test by combining  the $(s_0,p)$-norm  based tests with $p\in\mathcal{P}$, where $\mathcal{P}\subset\{1,2\ldots,\infty\}$ is a finite fixed set specified by users. Intuitively, by combining tests with various norms, the data-adaptive test enjoys  high power across various alternative hypothesis scenarios. In  (\ref{def:Wad}) and (\ref{def:Nad}), we introduce the data-adaptive tests as
\begin{equation}\label{def:WadNad}
W_{\rm ad}=\min_{p\in\mathcal{P}}\hat{P}^W_{(s_0,p)}
\hspace{2em}{\rm and}\hspace{2em}
N_{\rm ad}=\min_{p\in\mathcal{P}}\hat{P}^N_{(s_0,p)},
\end{equation}
where $\hat{P}^W_{(s_0,p)}$ and $\hat{P}^N_{(s_0,p)}$
are defined in (\ref{def:PhatWN}).  By setting 
$F_{W,(s_0,p)}(z):=\P(W_{(s_0,p)}\le z)$ and $F_{N,(s_0,p)}(z):=\P(N_{(s_0,p)}\le z)$, we have that the oracle $P$-values of
$W_{(s_0,p)}$ and $N_{(s_0,p)}$ are 
\[
P^W_{(s_0,p)}:= 1-F_{W,(s_0,p)}(W_{(s_0,p)})\hspace{1em}{\rm and }
\hspace{1em}
P^N_{(s_0,p)}:= 1-F_{N,(s_0,p)}(N_{(s_0,p)}).
\] 
By the definitions of 
$\hat{P}^W_{(s_0,p)}$ and $\hat{P}^N_{(s_0,p)}$ in (\ref{def:PhatWN}), $\hat{P}^W_{(s_0,p)}$ and $\hat{P}^N_{(s_0,p)}$ estimate 
$P^W_{(s_0,p)}$ and $P^N_{(s_0,p)}$. Therefore,  by (\ref{def:WadNad})  $W_{\rm ad}$ and $N_{\rm ad}$ estimate
\begin{equation}\label{def:tildaWNad}
\tilde{W}_{\rm ad}=\min_{p\in \mathcal{P}}P^W_{(s_0,p)}\hspace{2em} {\rm and}
\hspace{2em}
\tilde{N}_{\rm ad}=\min_{p\in\mathcal{P}}P^N_{(s_0,p)}.
\end{equation}
By setting  $\tilde{F}_{W,{\rm ad}}(z):=\P(\tilde{W}_{\rm ad}\le z)$ and  $\tilde{F}_{N,{\rm ad}}(z):=\P(\tilde{N}_{\rm ad}\le z)$, considering that the small values of $W_{\rm ad}$ and $N_{\rm ad}$ yield the rejection of the  null hypotheses,  we have that the oracle $P$-values of $W_{\rm ad}$ and $N_{\rm ad}$ are $\tilde{F}_{W\rm ad}(\tilde{W}_{\rm ad})$ and $\tilde{F}_{N,\rm ad}(\tilde{N}_{\rm ad})$.

After introducing these notations, we aim to justify the bootstrap procedure in Section
\ref{sec:adTestProc} by showing that  $\hat{P}^W_{\rm ad}$ and   $\hat{P}^N_{\rm ad}$ (defined in (\ref{def:hatPWad}) and (\ref{def:hatPNad})) are consistent estimators of the oracle $P$-values  $\tilde{F}_{W\rm ad}(\tilde{W}_{\rm ad})$ and $\tilde{F}_{N,\rm ad}(\tilde{N}_{\rm ad})$. For this,
we introduce Assumption {\bf (A)$''$} to  specify the scaling between $s_0$, $q$ and $n$ for the data-adaptive combined test. 

To state Assumption {\bf (A)$''$}, we need some additional notations. For the two-sample problem,  we introduce $\bG^N\sim N(\zero,\Rb_{12} )\in \reals^q$ in (\ref{def:GNR12}) to approximate  $\bN$. 
We set $f_{\bG^N, (s_0,p)}(x)$ and $c_{\bG^N, (s_0,p)}(\alpha)$ as the probability density function  and the  $\alpha$-quantile of $\|\bG^N\|_{(s_0,p)}$. We then define $h_{q,N}(\epsilon)$ as
\[
h_{q,N}(\epsilon)=\max_{p\in \mathcal{P}}\max_{x\in I^N_{(s_0,p)}(\epsilon) } f^{-1}_{\bG^N, (s_0,p)}(x),
\]
where $ I^N_{(s_0,p)}(\epsilon) = [c_{\bG^N,(s_0,p)}(\epsilon), c_{\bG^N,(s_0,p)}(1-\epsilon)]$. For the one-sample problem, we define $h_{q,W}(\epsilon)$ similarly for $\bG^W\sim N(\zero, \Rb_1)\in \reals^q$, where 
\[
\Rb_1=(r_{1,st})\in\reals^{q\times q} \hspace{1em} \text{with}\hspace{1em} r_{1,st}={\rm Corr}(h_s(\bX),h_t(\bX)).
\]
By definition,  $\Rb_1$ and $\Rb_{12}$ are the asymptotic correlation
matrices of $\bW$ and $\bN$, where $\bW$ and $\bN$ are defined in (\ref{def:WN}). 
With these additional notations, we then state Assumption {\bf (A)$''$} as follows.
\begin{itemize}
	\item{\bf (A)$''$} Under (\ref{def:hone-sampleu}), as $n_1\rightarrow\infty$, we assume  that $h^{0.6}_{q,W}(\epsilon)s_0^2\log q= o(n_1^{1/10})$ holds for any $0<\epsilon<1$.  Under  (\ref{def:htwo-sampleu}), as $n\rightarrow\infty$, we assume that $h^{0.6}_{q,N}(\epsilon)s_0^2\log q= o(n^{1/10})$ holds for any $0<\epsilon<1$.  
\end{itemize} 
Compared to Assumption {\bf {(A)}},  the required  scaling in Assumption {\bf {(A)}$''$} is more stringent. This is because when analyzing the combined test, we need not only the convergence of  distribution functions of the test statistics but also   their  uniform convergence of the quantile functions on $[\epsilon, 1-\epsilon]$. 
\begin{remark}\label{remark:Rs0p}
	Let $1\le s_0,\#(\mathcal{P})<\infty$. 
	If there are $0<C_0<\infty$ and $0<\eta<1$ such that $C_0^{-1}<\lambda_{\min}(\Rb_{12})\le \lambda_{\max}(\Rb_{12})<C_0$  and $\max_{i\neq j}|r_{ij}|<\eta $, we have $h_{q,N}(\epsilon)=O(1)$ for any $\epsilon\in(0,1)$, as $q\rightarrow\infty$. Similarly, if  $C_0^{-1}<\lambda_{\min}(\Rb_{1})\le \lambda_{\max}(\Rb_{1})<C_0$  and $\max_{i\neq j}|r_{ij}|<\eta $, we also have $h_{q,W}(\epsilon)=O(1)$ for any $\epsilon\in(0,1)$, as $ q\rightarrow\infty$.  The detailed proof is in Appendix \ref{proof:rmark:Rs0p} of supplementary materials.
\end{remark}
The detailed proof of Remark \ref{remark:Rs0p} is in Appendix \ref{proof:rmark:Rs0p} of supplementary materials, in which
we obtain a joint asymptotic distribution for the order statistics of nonindependent Gaussian random variables.  This  result  is nontrivial
and of independent technical interest.
After introducing additional assumptions, we then justify the data-adaptive combined test by the following theorem.

\begin{theorem}\label{thm:AdSize}
	Suppose Assumptions  {\bf (A)$''$},  {\bf (E)}, {\bf (M1)} and {\bf (M2)} hold.  For the one-sample problem, 
	under $\Hb_0$ of (\ref{def:hone-sampleu})  we have 
	\begin{equation}\label{limit:Wad}
	\P_{\Hb_0}\big(T^W_{\rm ad}=1\big)\rightarrow \alpha~\text{and}~	\tilde{F}_{W,{\rm ad}}(\tilde{W}_{\rm ad})-\hat{P}^W_{\rm ad}\rightarrow0~\text{as}~n_1,B\rightarrow\infty.
	\end{equation}
	Similarly, for the two-sample problem, under $\Hb_0$
	of (\ref{def:htwo-sampleu}) we have 
	\begin{equation}\label{limit:Nad}
	\P_{\Hb_0}\big(T^N_{\rm ad}=1\big)\rightarrow \alpha~\text{and}~
	\tilde{F}_{N,{\rm ad}}(\tilde{N}_{\rm ad})-\hat{P}^N_{\rm ad}\rightarrow0 ~\text{as}~n,B\rightarrow\infty.
	\end{equation}
\end{theorem}
The detailed proof of Theorem \ref{thm:AdSize} is in Appendix \ref{proof:thm:AdSize}  of  supplementary materials. 

\begin{remark}\label{remark:P}
	   To prove Theorem \ref{thm:AdSize}, we first show that for any fixed $0<\epsilon<1$,  not only the  distribution  function of $N_{\rm ad}$  but also its  quantile function on $[\epsilon, 1-\epsilon]$  converge to those of $\tilde{N}_{\rm ad}$. By choosing $\epsilon$ sufficiently small, we then prove that the probability of $\tilde{N}_{\rm ad} \in (0,\epsilon)$ is negligible to finish the proof. If 
	  $\#(\mathcal{P})\rightarrow\infty$, we cannot  guarantee $\tilde{N}_{\rm ad} \in (0,\epsilon)$ is negligible any more. Moreover, it is also very hard to prove
	  the convergence of  quantile functions on  $(\epsilon, 1-\epsilon)$ for $N_{\rm ad}$ with  $\epsilon\rightarrow 0$. Hence, when constructing the combined test, we  require $0<\#(\mathcal{P})<\infty$. By simulation, we recommend using   $\mathcal{P}=\{1,2,3,4,5,\infty\}$. The simulation also shows that there is no significant power advantage to  add more elements to $\mathcal{P}$  (see Appendix \ref{sec:simP}). Therefore, the assumption of  finite $\#(\mathcal{P})$ is enough for the practical usage.

\end{remark}

We now turn to the analysis of  the power of the combined test. For this, we have the following result.
\begin{theorem}\label{theorem:Adpower}
	Suppose Assumptions {\bf (A)$'$}, {\bf (E)}, {\bf (M1)}, and {\bf (M2)}) hold.  For the one-sample problem in (\ref{def:hone-sampleu}),  we  assume $\log q = o(n_1^{1/2})$,  $\varepsilon_{n_1}=o(1)$, and $\varepsilon_{n_1}\sqrt{\log q}\rightarrow\infty$ as $n_1,q\rightarrow\infty$. If $\Hb_1$ of (\ref{def:hone-sampleu}) holds with 
	\begin{equation}
	\|\bD_1\|_{(s_0,p)}\ge s_0 (1+\varepsilon_{n_1})
	\big(\sqrt{2\log q}+\sqrt{2\log (\#\{\mathcal{P}\}/\alpha)}\big),
	\end{equation}
	we have $
	\P_{\Hb_1}\big(T^W_{\rm ad}=1 \big)=1$  as $n_1,q, B\rightarrow\infty.$
	Similarly, for the two-sample problem in (\ref{def:htwo-sampleu})  we assume $\log q=o(n^{1/2})$, $\varepsilon_n=o(1)$, and $\varepsilon_n\sqrt{\log q}\rightarrow\infty$ as $n,q\rightarrow\infty$. If $\Hb_1$ of (\ref{def:htwo-sampleu}) holds with 
	\begin{equation}\label{def:AdpowerD2}
	\|\bD_2\|_{(s_0,p)}\ge s_0 (1+\varepsilon_n)
	\big(\sqrt{2\log q}+\sqrt{2\log (\#\{\mathcal{P}\}/\alpha)}\big),
	\end{equation}
	we have $
	\P_{\Hb_1}\big(T^{N}_{\rm ad}=1 \big)\rightarrow 1$ as $n, q, B\rightarrow\infty.$
\end{theorem}
The detailed proof of Theorem \ref{theorem:Adpower} is presented in Appendix \ref{proof:theorem:Adpower} of supplementary materials.

\begin{remark} 
	On one hand, by Theorems \ref{theorem:power} and \ref{theorem:Adpower}, we require
	$\|\bu_1-\bu_0\|_{(s_0,p)}\succeq s_0\sqrt{\log(q)/n_1}$ or $\|\bu_1-\bu_2\|_{(s_0,p)} \succeq s_0\sqrt{\log(q)/n}$ for our proposed methods  to reject the null hypothesis with  overwhelming probability. On the other hand, by Theorem 3 in \cite{cai2013, tony2014two}, Theorem 4.3 in \cite{han2014distribution}, and Theorem 3.5 in \cite{zhou2015extreme}, for both vector-based and matrix-based  high dimensional tests, any $\alpha$-level test is unable to reject the null hypothesis correctly uniformly over $\|\bmu_1-\bmu_0\|_\infty \ge c_0\sqrt{\log(d) /n}$ or  $\|\bmu_1-\bmu_0\|_\infty \ge c_0\sqrt{\log(d) /n}$ with $c_0$  sufficiently small. Therefore, we  have that our proposed methods with finite $s_0$ are rate-optimal for these sparse alternatives.
\end{remark} 
\section{Simulation results}\label{section:Exper}
The goal of this section is to investigate the  numerical performance of the proposed tests. For this, we compare our methods with several existing methods from the literature. In this section, we only  consider the high dimensional mean test under different settings. We put additional simulation results for  testing high-dimensional covariance/correlation coefficients in Appendix \ref{appendix:moreSimulation}  to  illustrate the proposed methods'  generality. Apart from simulated datasets,  Appendix \ref{appendix:moreSimulation} also includes the experimental results on real world fMRI datasets.

In the context of high dimensional mean test,  we compare the proposed tests with four existing methods: 
Hotelling's $T^2$ test, the $L_2$-type tests given in \cite{bai1993limit} and \cite{srivastava2008test}, and the $L_\infty$-type test give in \cite{tony2014two}. We refer these four tests as $T^2$, BY, SD, and CLX. For simplicity, we only consider the two-sample problem. We generate synthetic data from a wide range of covariance structure including both sparse and non-sparse settings.  We also consider a wide range of alternative scenarios including both sparse and dense settings to investigate the power of the proposed methods.

Under the null hypothesis, we sample $n_1+n_2$ data points from the following models.
\begin{itemize}
	\item {\bf Model 1.} (Gaussian distribution with block diagonal $\bSigma$) We set 
	$\bSigma^\star=(\sigma^\star_{ij})\in\reals^{d\times d}$ with $\sigma^\star_{ii}\overset{\rm i.i.d.}{\sim} {\rm U}(1,2)$,  $\sigma^\star_{ij}=0.5$ for
	$5(k-1)+1\le i\ne j\le 5k$, where $k=1,\ldots,\lfloor d/5\rfloor$, and $\sigma^\star_{ij} = 0$ otherwise. In this model, under the null hypothesis we generate $n_1+n_2$ random vectors from $N(\zero,\bSigma^\star)$.
	\item {\bf Model 2.} (Gaussian distribution with banded $\bSigma$) We set $\bSigma'=(\sigma'_{ij})\in\reals^{d\times d}$ with 
	$\sigma'_{ij}= 0.4^{|i-j|}$ for $1\le i,j\le d$. In this model, under the null hypothesis we generate $n_1+n_2$ random vectors from $N(\zero,\bSigma')$.
	\item {\bf Model 3.} (Gaussian distribution with non-sparse $\bSigma$) We set $\Fb=(f_{ij})\in \reals^{d\times d}$ with $f_{ii}=1$, $f_{ii+1}= f_{i+1i}=0.5$, and $f_{ij}=0$ otherwise.
	We also set  that $\Ub\sim {\rm U}(\Lambda_{d,k})$ follows the uniform distribution on the  Stiefel manifold $\Lambda_{d,k}$ (i.e., $\Lambda_{d,k}=\{\Hb\in\reals^{d\times k}:\Hb^\top\Hb = \Ib_k\}$). After introducing $\Fb$ and $\Ub$, we then set the correlation matrix as $\Rb=(\Db^f)^{-1/2}(\Fb+\Ub\Ub^\top)(\Db^f)^{-1/2}$ with $\Db^f= {\rm Diag}(\Fb+\Ub\Ub^\top)$.  By setting $\Db=(d_{ij})\in \reals^{d\times d}$ as a diagonal matrix with $d_{ii}\sim U(1,2)$, we generate $n_1+n_2$ random vectors from  $N(0,\bSigma)$ with $\bSigma= \Db^{1/2}\Rb\Db^{1/2}$.
	
	\item {\bf Model 4.} (Multivariate $t$ distribution) 
	We generate $n_1+n_2$ random vectors from the multivariate $t$ distribution $t(\nu,\bmu, \bSigma)$ according to 
	$\bmu +\bZ/\sqrt{W/\nu}$, where we have $W \sim \chi^2(\nu)$ and $\bZ \sim N(\zero, \bSigma)$ with $W$ and $\bZ$ independent of each other.
	In the simulation, we set $\bmu=\zero$, $\nu=5$, and $\bSigma=\bSigma^\star$.
\end{itemize}
We use the above models to show that the proposed methods are valid given a fixed size $\alpha$ under various covariance structures and distributions. To present the empirical power of the proposed methods, we introduce a random vector $\bV\in \reals^d$ with exactly $s$ nonzero entries, which are selected randomly from $d$ coordinates. Each nonzero entry follows an independent uniform distribution ${\rm U}(u_1,u_2)$. Under the alternative hypothesis, we set $\bmu_1=\zero$ and $\bmu_2=\bV$. By choosing different $s$, $u_1$, and $u_2$, we compare the power of the proposed methods with that the existing methods under both the sparse and non-sparse settings.

\begin{table}[!ht]
	\begin{center}
		\addtolength{\tabcolsep}{-2pt}
		\scriptsize
		\caption{Empirical sizes for {\bf Model 1} with $\alpha=0.05$, $B=300$, and $n_1=n_2=100$ based on 2000
			replications.} \vspace{0.2cm}\label{table:1}
		\begin{tabular*}{12.3cm}{cccccccccp{2mm}cccc}
			\toprule[2pt]
			\multicolumn{2}{c}{}&\multicolumn{12}{c}{ Empirical size (\%) }\\
			$d$&$s_0$&$p=1$&$p=2$&$p=3$&$p=4$&$p=5$&$p=\infty$&$T^N_{\rm ad}$&  &$T^2$&BY&SD&CLX\\\hline
			75&5& 5.50 & 5.85 &6.15 &6.35 &6.30 &6.60& 6.50& 
			& 5.05& 5.85& 4.65 & 5.25\\
			&30& 4.20 & 4.45 &4.90 &5.30 &5.70 &6.90 & 5.90& 
			& 5.05& 5.85& 4.65 & 5.25\\
			&75& 3.70& 3.95 &4.75 & 5.10& 5.65 &6.75 & 5.50& 
			& 5.05& 5.85& 4.65 & 5.25
			\\\hline
			200&10&  4.75 &4.50 &4.85 &5.20 &5.25 &6.55&5.75 & 
			& - & 4.85& 3.85& 5.35     
			\\
			&50& 2.80&2.90 &3.55 &3.80 &4.35 &6.45& 4.75& 
			& - & 4.85& 3.85& 5.35     
			\\
			&100& 1.90 &2.25 &2.50 &3.60 &3.85& 6.45 & 4.80 &
			& - & 4.85& 3.85& 5.35    
			\\
			&150& 2.35 &2.45 &2.75 &3.70 & 4.15 & 6.85 & 4.90& 
			& - & 4.85& 3.85& 5.35           
			\\
			&200& 2.30& 2.35 &2.95 &3.65 &4.35 & 7.10 & 5.15 &  
			& - & 4.85& 3.85& 5.35   
			\\\hline
			400&10& 4.20 &4.30& 4.70 &5.30 &5.40 &7.60&5.90 & 
			& - & 5.35& 4.55& 6.65     
			\\
			&50&2.45 &2.50 &2.80 &3.45 &4.25 &8.25& 5.00& 
			& - & 5.35& 4.55& 6.65 
			\\
			&100& 2.05& 2.30 &2.25 &2.65& 3.95& 7.90& 4.75 &
			& - & 5.35& 4.55& 6.65   
			\\
			&200& 1.45 &1.60& 1.90& 2.70 &3.60& 7.75 & 4.55& 
			& - & 5.35& 4.55& 6.65           
			\\
			&400&1.40 &1.40 &1.75 &2.70& 3.80 &7.85 & 4.70 &  
			& - & 5.35& 4.55& 6.65 
			
			\\\hline
			800&10& 4.75& 4.95 &5.20 &5.50& 5.95& 9.10&6.30 & 
			& - & 5.65& 4.65& 7.45    
			\\
			&100&0.75 &1.20& 1.40 &1.80 &2.65& 8.85& 4.45& 
			& - & 5.65& 4.65& 7.45   
			\\
			&200& 0.40& 0.50& 0.75& 1.40& 2.00 &8.85& 4.45 &
			& - & 5.65& 4.65& 7.45   
			\\
			&400& 0.55& 0.45& 0.70& 1.20 &2.10 &8.15& 3.95& 
			& - & 5.65& 4.65& 7.45        
			\\
			&600& 0.40& 0.35& 0.80& 1.20& 2.00& 8.70& 4.00 &  
			& - & 5.65& 4.65& 7.45   
			\\
			&800&0.40& 0.55 &0.75 &1.35 &1.85& 8.65 & 3.65 &  
			& - & 5.65& 4.65& 7.45   
			\\
			\bottomrule[2pt]
			
		\end{tabular*}
	\end{center}
\end{table}

\begin{table}[!ht]
	\begin{center}
		\addtolength{\tabcolsep}{-2pt}
		\scriptsize
		\caption{Empirical power of {\bf Model 1} with $\alpha=0.05$, $B=300$, and $n_1=n_2=100$ based on 2000
			replications.} \vspace{0.2cm}\label{table:2}
		\begin{tabular*}{12.6cm}{cccccccccp{2mm}cccc}
			\toprule[2pt]
			\multicolumn{14}{c}{{ Empirical power (\%) with $\bmu_1=\zero$ and $\bmu_2=\bV$  with $s=5$, $u_1=0$, and $u_2=4\sqrt{\log(d)/n}$  }}\\
			$d$&$s_0$&$p=1$&$p=2$&$p=3$&$p=4$&$p=5$&$p=\infty$&$T^N_{\rm ad}$&  &$T^2$&BY&SD&CLX\\\hline
			75&5& 82.10 &84.35 &85.70& 86.50 &86.80& 85.10& 86.90& 
			& 73.7& 67.85& 66.45 & 83.5\\
			&30& 49.10 &69.20& 78.75 &83.80 &85.40& 85.50 &84.50& 
			& 73.7& 67.85& 66.45 & 83.5\\
			&75& 32.70& 64.25& 78.20 &83.25 &85.15& 85.00& 83.70& 
			& 73.7& 67.85& 66.45 & 83.5\\
			\hline
			200&10& 75.85&81.05 &83.85& 84.95& 86.10 &86.40 &85.65& 
			& - & 55.65& 53.90& 85.25     
			\\
			&50& 36.20& 59.65 &75.35 &81.50 &84.00& 86.05& 84.40& 
			& - & 55.65& 53.90& 85.25  
			\\
			&100&  23.60& 48.90 &72.65& 80.75 &84.20& 86.45& 84.35 &
			& - & 55.65& 53.90& 85.25     
			\\
			&150& 18.70& 45.40& 72.10 &81.15& 84.45 &86.45 & 84.35& 
			& - & 55.65& 53.90& 85.25             
			\\
			&200&  17.55& 45.20& 72.40 &81.00& 84.20 &86.15 & 84.25 &  
			& - & 55.65& 53.90& 85.25   
			\\\hline
			400&10&  77.90 &82.15 &85.20& 87.25& 88.05& 87.80&87.90 & 
			& - & 44.25& 42.75& 87.05     
			\\
			&50&34.25 &56.40 &71.85& 79.65 &84.05& 87.90& 85.55& 
			& - & 44.25& 42.75& 87.05  
			\\
			&100& 18.45 &40.45& 65.15& 77.30 &83.50& 87.60& 85.55 &
			& - & 44.25& 42.75& 87.05  
			\\
			&200&9.85 &29.35& 60.25 &76.25 &83.35& 88.10& 85.45& 
			& - & 44.25& 42.75& 87.05         
			\\
			&400&6.95 &25.60 &60.00& 75.90 &83.60 &87.50 & 85.05 &  
			& - & 44.25& 42.75& 87.05 \\\hline
			\multicolumn{14}{c}{}\\
			\multicolumn{14}{c}{{ Empirical power (\%)  with $\bmu_1=\zero$ and $\bmu_2=\bV$  with $s=100$, $u_1=0$, and $u_2=3\sqrt{1/n}$  }}\\
			$d$&$s_0$&$p=1$&$p=2$&$p=3$&$p=4$&$p=5$&$p=\infty$&$T^N_{\rm ad}$&  &$T^2$&BY&SD&CLX\\\hline
			200&10&87.55& 86.05& 85.00& 83.50& 81.75 &51.25&82.30& 
			& - & 96.20& 96.05& 47.45
			\\
			&50& 93.25& 93.80& 93.40 &92.05 &89.40 &51.05& 90.20& 
			& - & 96.20& 96.05& 47.45 
			\\
			&100&  92.75 &93.85 &93.80& 92.55 &89.80& 51.55& 91.35 &
			& - & 96.20& 96.05& 47.45
			\\
			&150& 90.95& 93.55& 94.35& 92.35& 89.90 &52.75& 90.65& 
			& - & 96.20& 96.05& 47.45        
			\\
			&200&  90.05& 93.75& 93.80& 92.45& 89.55& 51.70 & 91.10&  
			& - & 96.20& 96.05& 47.45
			\\\hline
			400&10& 70.40 &69.95 &68.75& 67.90& 66.00 &42.70&64.00 & 
			& - & 85.10& 84.25& 38.70     
			\\
			&50&72.60& 73.95& 74.95& 74.85 &73.35 &42.30& 70.25& 
			& - & 85.10& 84.25& 38.70
			\\
			&100&69.85& 72.80 &73.95& 74.45& 73.30& 42.35& 69.45 &
			& - & 85.10& 84.25& 38.70 
			\\
			&200& 61.55& 68.25& 73.05&74.45& 73.90& 42.45& 68.30& 
			& - & 85.10& 84.25& 38.70       
			\\
			&400& 54.40 &67.15& 73.00& 74.50 &73.40 &43.15 & 67.15 &  
			& - & 85.10& 84.25& 38.70
			\\\bottomrule[2pt]
			
		\end{tabular*}
	\end{center}
\end{table}

In Table \ref{table:1}, we present the empirical sizes of  introduced methods  for {\bf Model 1}. We set $n_1=n_2=n=100$ and $q=d=75,200,400,800$. The nominal
significance level is $0.05$. We compare our methods with four other tests: $T^2$, BY, SD, and  CLX. Moreover, $T^2$, BY, and SD are $L_2$-type and CLX is $L_\infty$-type. The $T^2$ test requires $d<n$, so that we don't perform  $T^2$ test as $d>n$.  In the current setting, the four existing methods can control the size correctly, except  that CLX test suffers a size distortion as  $d$ is significantly larger ($d=800$) than $n$. For the $(s_0,p)$-norm based tests, when $s_0$ is
significantly smaller ($s_0=5,10$) than $d$, they can control the size correctly, except that the $(s_0,\infty)$-norm based test suffers  a size distortion as $d$ is significantly large ($d=800$). As $s_0$ increases, the  empirical size of $(s_0,p)$-norm based tests decreases dramatically especially for small $p$,  making the $(s_0,p)$-norm based tests with small $p$ overly conservative. Although the $(s_0,p)$-norm based tests perform differently with different $s_0$ and $p$, the data-adaptive combined test $T_{\rm ad}^N$ can control the size correctly under various settings of $d$ and $n$.

In Table \ref{table:2}, we compare these methods under different alternative scenarios. In  the sparse alternative setting, we set  $\bmu_2=\bV$ with $s=5$ nonzero entries. Each entry follows  independent uniform distribution ${\rm U}(0,4\sqrt{\log(d)/n})$. 
In this setting, the $L_\infty$-type test achieves a higher empirical power than  the $L_2$-type tests. In the dense alternative setting, we set $\bmu_2=\bV$ with $s=100$ nonzero entries of the magnitude ${\rm U}(0,3n^{-1/2} )$.  In this setting, the $L_2$-type tests are  more powerful. This similar pattern also appears in the $(s_0,p)$-norm based tests. As $p$ increases, the $(s_0,p)$-norm based test is more sensitive to the sparse alternative. The influence   of  $s_0$ is more complicated. However, by choosing $s_0$ close to $s$,
the tests always enjoy  good performance. For the data-adaptive  combined test $T_{\rm ad}^N$, we choose a balanced $\mathcal{P}$ including both small and large values of $p$. Hence, in various settings of the alternative scenarios, $d$, and $n$,
it always has a high power. Although $T_{\rm ad }^N$ with balanced $\mathcal{P}$ may not be the most powerful option for some  alternatives,  $T_{\rm ad}^N$ is adaptive to the alternative setting and powerful enough in various kinds of alternative scenarios.  Theoretically, there is no uniformly most powerful test in all the alternative scenarios \citep{cox1979theoretical}. If the alternative pattern is unknown, the data-adaptive test with balanced $\mathcal{P}$ (including small and large $p$) is a good choice. If the alternative pattern is known, by choosing $\mathcal{P}$ accordingly we can still construct a powerful test. For the choice of $s_0$, similarly to the $(s_0,p)$-norm based tests,  $T_{\rm ad}^N$ with $s_0$ close to $s$  is always powerful.

We put the numerical results of  {\bf Models 2-4} in  Appendix  \ref{appendix:moreSimulation} of supplementary materials. Their experimental results are similar to {\bf Model 1} and indicate that the proposed methods work well in various settings.

\section{Summary and discussion}\label{section:Diss}
This paper considers the problem of  testing  high dimensional $U$-statistic
based vectors. We construct a family of tests based on the $(s_0,p)$-norm. By the introduction of $s_0$, when $q$ is large, 
we can increase the power compared to the tradition $L_p$-norm based test (especially for small $p$). Moreover, by choosing $p$ properly, we can further enhance the power under different alternatives.
We also introduce a data-adaptive combined test, which is simultaneously powerful under a wide variety of alternatives. Moreover, We also develop a trick for avoiding the high computational cost of the double-loop bootstrap for the data-adaptive  combined test with theoretical guarantee in high dimensions. 

We then discuss the choice of $s_0$ and $\mathcal{P}$. Theoretically, for individual $(s_0,p)$-norm tests we generally require that $s_0^\zeta\log q = o(n^\delta)$ holds with $\zeta=2$ and $0<\delta<1/7$ for all $p\in [1,\infty]$.
We also point out that it is  possible to  reduce $\zeta$ for some specified $p$. For combined tests, we require $0<\#(\mathcal{P})<\infty$  to prevent the  test statistic from going to $0$. By simulation, we also see that  the proposed tests with $s_0$ close to $s$ (true unknown number of entries violating $\Hb_0$)  enjoy high power, which makes $s$ a good candidate for $s_0$. In practice, we recommend choosing $s_0$ and $s$ as close as possible without violating  theoretical conditions.

%

There are several possible future directions of this work. For instance, how to generalize the idea to the $k$-sample testing problems ($k>2$) has been for a future investigation. This may require a nontrivial extension of the theoretical
analysis. Moreover, our theory is based on the Gaussian approximation for the sum of high dimensional independent random vectors
from \cite{chernozhukov2013gaussian}. \cite{zhang2014bootstrapping} and \cite{zhang2015gaussian}  further 
study Gaussian approximations for high dimensional time series, which allow to generalize our methods for dependent data. As a significant
amount of additional work is still needed, we shall report the results elsewhere in the future.

\end{bibunit}

\begin{bibunit}
\newpage
\setcounter{page}{1}

\title{ Supplement Materials to
		``A Unified Framework for Testing High Dimensional Parameters: A Data-Adaptive Approach"
}

\begin{aug}
	\author{\fnms{Cheng} \snm{Zhou}\thanksref{t1}\ead[label=e1]{chengzhmike@gmail.com}},
	\author{\fnms{Xinsheng} \snm{Zhang}\thanksref{t1}\ead[label=e2]{xszhang@fudan.edu.cn}},
	\author{\fnms{Wenxin} \snm{Zhou}\thanksref{t2}\ead[label=e2]{wenxinz@princeton.edu}}
	\and
	\author{\fnms{Han} \snm{Liu}\thanksref{t2}\ead[label=e3]{hanliu@princeton.edu}}
	
	\affiliation{Department of Statistics, Fudan University\thanksmark{t1} and Department of Operation Research and
		Financial Engineering, Princeton University\thanksmark{t2}}
\end{aug}

\begin{center}
	{ ABSTRACT}
\end{center}
\begin{center}
\noindent\begin{minipage}[c]{.8\textwidth}
	\footnotesize\mdseries\upshape
	The supplementary materials contain additional details of  the paper  ``A Unified Framework for Testing High Dimensional Parameters: A Data-Adaptive Approach" authored by Cheng Zhou,  Xinsheng Zhang, Wen-Xin Zhou, and Han Liu. After introducing some useful lemmas in Appendix \ref{sec:usefullemmas},  We  prove main results in Appendix   \ref{SM:mainResults}.
In Appendices \ref{sec:appendix C} and \ref{sec:appendix D}, we prove lemmas required by the proofs in  Appendix \ref{SM:mainResults}. In Appendix \ref{sec:proof:usefullemma}, we  prove lemmas introduced in Appendix \ref{sec:usefullemmas}. In Appendix \ref{appendix:moreSimulation}, we present additional numerical experimental results. Throughout supplementary materials, we use $C$, $C_1$, $C_2,\ldots$ to denote constants which do not depend on $n,d,$ and $q$. These constants can vary from place to place.
\end{minipage}
\end{center}

\appendix
\section{Useful lemmas}\label{sec:usefullemmas}
In Appendix \ref{sec:usefullemmas}, we introduce some useful lemmas that will be used many times for proving main results. We put their proof in Appendix \ref{sec:proof:usefullemma}.
To present these lemmas, we need some additional notations.  
Let $\bZ_1, \ldots,\bZ_{n}$ be independent random vectors in $\reals^d$ with 
$\bZ_k=(Z_{k1},\ldots,Z_{kd})^\top$  and $\E[\bZ_k]=\zero$ for $k=1,\ldots,n$. Let $\bW_1,\ldots,\bW_{n}$ be independent  Gaussian random vectors in $\reals^d$ such that  $\bW_k$ has the same mean vector and covariance matrix as $\bZ_k$.
By setting $\mathcal{V}_{s_0}:=\{\vb\in \mathbb{S}^{d-1}: \|\vb\|_0\le s_0\}$,  we require the following conditions:
\begin{itemize}
	\item{\bf  (M1)$'$} $n^{-1}\sum_{k=1}^n\E\big[(\vb'\bZ_{k})^2\big]\ge b>0$ for any $\vb\in \mathcal{V}_{s_0}$;
	\item{\bf (M2)$'$}$n^{-1}\sum_{k=1}^n\E\big[|Z_{kj}|^{2+\ell}\big]\le K^\ell$  for $\ell=1,2$ and $j=1,\ldots,d$. 
	\item{\bf (E)$'$} $\E\big[\exp(|Z_{kj}|/K)\big]\le 2$ for $j=1,\ldots,d$ and $k=1,\ldots,n$.
\end{itemize}

\begin{lemma}\label{lemma:CCK2}
	Assume $s_0^2\log(dn)=O(n^\zeta)$ with $0<\zeta<1/7$. If $\bZ_1, \ldots,\bZ_{n}$ satisfy {\bf (M1)$'$}, {\bf (M2)$'$}, and {\bf (E)$'$}.
	By setting $S_{n}^{\bZ}=n_1^{-1/2}\sum\nolimits_{k=1}^{n}\bZ_{k}$ and $S_{n}^{\bW}=n_1^{-1/2}\sum\nolimits_{k=1}^{n}\bW_{k}$,  for $1\le p\le \infty$ and  sufficiently large $n$, there is a constant $\zeta_0>0$ such that
	\begin{equation}\label{ineq:HCLT_s0p}
	\sup_{z\in (0,\infty)}\Big|\P\Big(\|S_{n}^{\bZ}\|_{(s_0,p)}\le z\Big)-\P\Big(\|S_{n}^{\bW}\|_{(s_0,p)}\le z\Big)\Big|\le C{n^{-\zeta_0}},
	\end{equation}
	where $C$  depends on $b$ and $K$.
\end{lemma}

\begin{lemma}\label{lemma:polyAppr}(Corollary 1.2 in \cite{barvinok2014thrifty:app})
	For any compact  and symmetric convex set $\cC\in\mathbb{R}^d$ with non-empty interior  and  $\gamma>e/4\sqrt{2}$, there exist a polytope $\cP\in\mathbb{R}^d$ and  a constant $\epsilon_\gamma >0$ such that
	for any $0<\epsilon<\epsilon_\gamma$, we have 
	\[
	\cP\subset \cC \subset(1+\epsilon)\cP\hspace{2em}\text{and}\hspace{2em} V<\Big(\frac{\gamma}{\sqrt{\epsilon}}\ln \frac{1}{\epsilon}\Big)^d,
	\]
	where $V$ is the vertex number of $\cP$.
\end{lemma}
We call a  set $A^m$ is m-generated if it is the intersection of $m$ half-spaces. Therefore, $A^m$ is a polytope with at least $m$ facets. We then set $\mathcal{V}(A^m)$ as the set of  $m$ unit vectors that are outward normal to the facets of $A^m$. For $\epsilon>0$, we then  define
\[
A^{m,\epsilon} :=\cap_{v\in\mathcal{V}(A^m)}\{w\in \reals^d:  w^\top v \le \mathcal{S}_{A^m}(v)+\epsilon\}, 
\]
where $\mathcal{S}_{A^m}(v):=\sup\{w^\top v: w\in A^m \}$.

\begin{lemma}\label{lemma:approximstCovexSet}
	Let $\mathcal{E}^{R,d} = \{\xb\in\reals^d: \|\xb\|\le  R\}$ and
	$V^{z,d}_{(s_0,p)}=\{\xb\in\reals^d:\|\xb\|_{(s_0,p)}\le z\}$.
	For any $\gamma > e/ 4\sqrt{2}$, there is a $m$-generated convex set $A^m\in\reals^d$ and a constant $\epsilon_\gamma$ such that for any $0<\epsilon<\epsilon_\gamma$, we have
	\[
	A^m\subset \mathcal{E}^{R,d} \cap V^{z, d}_{(s_0,p)}\subset A^{m, R\epsilon}\hspace{2em}\text{and}\hspace{2em}
	m\le d^{s_0}\Big(\frac{\gamma}{\sqrt{\epsilon}}\ln \frac{1}{\epsilon}\Big)^{s_0^2}.
	\]
\end{lemma}
\begin{lemma}\label{lemma:gaussianDif}(Nazarov’s inequality in \cite{nazarov2003:app})
Let $\bW=(W_1,\ldots,W_d)^\top\in \reals^d$ be centered Gaussian random vector with $\inf_{k=1,\ldots, d}E[W_k^2]\ge b >0$. For any $\xb\in\reals^d$ and $a>0$, we then have 
\[
\P(\bW\le \xb + a) - \P(\bW \le \xb) \le Ca\sqrt{\log d},
\]
where $C$ only depends on $b$.

\end{lemma}

\begin{lemma}\label{lemma:foldedNormal}
	$\bW=(W_1,\dots, W_d)^\top$ is a random vector with
	the  marginal distribution $N(0,\sigma^2)$. For any $t>0$, we have
	\begin{equation}\label{ineq:fn0}
	\E\Big[\max_{1\le i\le d}|W_i|\Big]\le \frac{\log(2d)}{t} +\frac{t\sigma^2}{2}.
	\end{equation}
\end{lemma}

To estimate the covariance matrix of $U$-statistic based vector, we introduce  $\sigma_{\gamma,st}$ and $\hat{\sigma}_{\gamma,st}$  in (\ref{def:sigmast}) and (\ref{def:sigmahat}).
The following lemma then analyzes  the estimation error of $\hat{\sigma}_{\gamma,st}$. To analyze the correlation matrix, we also  provide the approximation error of $\hat{r}_{\gamma,st}$, where
\begin{equation}\label{def:R1R2hatR1R2}
r_{\gamma,st}=\sigma_{\gamma,st}/{\sqrt{\sigma_{\gamma,ss}\sigma_{\gamma,tt}}}
\hspace{2em}{\rm and}\hspace{2em}
\hat{r}_{\gamma,st}={\hat{\sigma}_{\gamma,st}}/{\sqrt{\hat{\sigma}_{\gamma,ss}\hat{\sigma}_{\gamma,tt}}}.
\end{equation}
\begin{lemma}\label{lemma:sigmaRErrors}
	Assumptions  {\bf (E)}, {\bf (M1)}, and {\bf (M2)} hold. For $\log(qn)=o(n^{1/3})$ and $m>1$, when $n$ is sufficiently large, 
	\begin{equation}\label{sigmaErrorM1+}
	\max_{1\le s,t\le q\atop \gamma=1,2} \max\Big(\big|\hat{\sigma}_{\gamma,st}-\sigma_{\gamma,st}\big|, \big|\hat{r}_{\gamma,st}-r_{\gamma,st}\big|\Big) \le C \frac{\log^{3/2}(qn)}{\sqrt{n}},
	\end{equation}
	holds with probability $1-C_1n^{-1}$. For $\log(qn)=o(n^{1/2})$ and $m=1$,  when $n$ is sufficiently large, 
	\begin{equation}\label{sigmaErrorM1+}
	\max_{1\le s,t\le q\atop \gamma=1,2} \max\Big(\big|\hat{\sigma}_{\gamma,st}-\sigma_{\gamma,st}\big|, \big|\hat{r}_{\gamma,st}-r_{\gamma,st}\big|\Big) \le C\sqrt{\frac{\log(qn)}{n}}+C\frac{\log ^2(qn)}{n},
	\end{equation}
	holds with probability $1-C_1n^{-1}$. 
\end{lemma}

\section{Proof of main results} \label{SM:mainResults}
In Appendix \ref{SM:mainResults}, we present the detailed proofs of main results including Proposition \ref{proposition:s0pnorm},  Theorems \ref{therom:CoreWN1+},  \ref{theorem:power}, \ref{thm:AdSize} \ref{theorem:Adpower}, Remarks \ref{remark:Rs0p}, and Corollary \ref{corollary:size}.

\subsection{Proof of Proposition \ref{proposition:s0pnorm} }\label{proof:s0pnorm}
\begin{proof}
	We need to prove that for any $1\le p\le \infty$, $a\in\reals$ and $\xb,\yb\in\reals^d$, we have (i) $\|a\xb\|_{(s_0,p)}=|a|\|\xb\|_{(s_0,p)}$; (ii) $\|\xb + \yb\|_{(s_0,p)}\le \|\xb\|_{(s_0,p)}+\| \yb\|_{(s_0,p)}$;
(iii) $\|\xb\|_{(s_0,p)}=0$ implies $\xb=\zero$. 
	By Definition \ref{def:s0pnorm}, for $\xb=(x_1,\ldots,x_d)^\top$ we have \[\|\xb\|_{(s_0,p)}=\Big(\sum\nolimits_{j=d-s_0+1}^d(x^{(j)})^p\Big)^{1/p},\]
	We use $\bk_1$ to denote the index of $x^{(d-s_0+1)}, x^{(d-s_0+2)},\ldots, x^{(d)}$. Therefore, we have $\|\xb\|_{(s_0,p)}=\|\xb_{\bk_1 }\|_p$, where $\xb_{\bk_1}\in \reals^{s_0}$. We then separately prove (i), (ii), and (iii). For (i), we have 
	\[\|a\xb\|_{(s_0,p)}=\|a\xb_{\bk_1}\|_p=|a|\|\xb_{\bk_1}\|_p=|a|\|\xb\|_{(s_0,p)}.\]
	For (iii), from $\|\xb\|_{(s_0,p)}=0$, we have $x^{(d)}=0$, which implies $\xb=\zero$. Therefore, to prove Proposition \ref{proposition:s0pnorm}, we only need to prove 
	\[
	\|\xb + \yb\|_{(s_0,p)}\le \|\xb\|_{(s_0,p)}+\| \yb\|_{(s_0,p)}.
	\]
	Similarly to the definition of $\bk_1$, we define $\bk_2$, $\bk_{12}$ for $\yb$ and $\xb+\yb$. We then have 
	\begin{equation}\label{neq:norm1}
	\|\yb\|_{(s_0,p)}=\|\yb_{\bk_2}\|_p\hspace{1em}\text{and}\hspace{1em}
	\|\xb + \yb\|_{(s_0,p)}=\|(\xb+\yb)_{\bk_{12}}\|_{p}.
	\end{equation}
	For $1\le p\le \infty$, $\|\cdot\|_p$ is a norm. Hence, we have 
	\begin{equation}\label{neq:norm2}
	\|(\xb+\yb)_{\bk_{12}}\|_{p}=\|\xb_{\bk_{12}}+\yb_{\bk_{12}}\|_{p}\le \|\xb_{\bk_{12}}\|_p+\|\yb_{\bk_{12}}\|_{p}.
	\end{equation}
	 By the definition of $\bk_1$ and $\bk_2$, we  have 
	 \begin{equation}\label{neq:norm3}
	 \|\xb_{\bk_{12}}\|_p\le \|\xb_{\bk_1}\|_p=\|\xb\|_{(s_0,p)}\hspace{1em} \text{and}\hspace{1em}  \|\yb_{\bk_{12}}\|_p\le \|\yb_{\bk_2}\|_p=\|\yb\|_{(s_0,p)}.
	 \end{equation}
	 Combining (\ref{neq:norm1}), (\ref{neq:norm2}), and (\ref{neq:norm3}), we have (iii), which finishes the proof.
\end{proof}
\subsection{Proof of Theorem \ref{therom:CoreWN1+}}\label{proof:theoremCoreNW1+}
\begin{proof}
	In Theorem \ref{therom:CoreWN1+}, we aim to prove (\ref{ineq:CoreW1+}) and (\ref{ineq:CoreN1+}).
	For simplicity,  we only present the  detailed proof of (\ref{ineq:CoreN1+}). According to the proof sketch in Section \ref{sec:theorys0lp}, the proof proceeds in three steps. In the first step, we obtain the approximate distribution of $\bN$. In the second step, given $\mathcal{X}$ and $\mathcal{Y}$ we  obtain the bootstrap sample $\bN^b$'s distribution. In the last step, we analyze the approximation error between $\bN$ and $\bN^b|\mathcal{X},\mathcal{Y}$ to yield (\ref{ineq:CoreN1+}).
	
	{\bf Step (i).} In this step, we  aim to obtain  the approximate distribution of $\bN$.  As $\hat{u}_{\gamma,s}$ is a  $U$-statistic, by the Hoeffding decomposition we approximate  $\bN$ by a sum of independent random vectors. Hence, we can further  approximate the sum by its Gaussian counterpart. In detail, under the null hypothesis we have $u_{1,s}=u_{2,s}$. Therefore, we rewrite  $N_{s}$ as
	\begin{equation}\label{def:N_sN1+}
	N_s = (\tilde{u}_{1,s}-\tilde{u}_{2,s})/{\sqrt{\hat{v}_{1,s}/n_1
			+\hat{v}_{2,s}/n_2}},
	\end{equation}
	where $\tilde{u}_{\gamma,s}:=\hat{u}_{\gamma,s}-u_{\gamma,s}$ is the centralized version of
	$\hat{u}_{\gamma,s}$. For introducing Hoeffding  decomposition, we define
	\[
	h_s(\bX_k)=\E[\Psi_s(\bX_{k_1},\ldots,\bX_{k_m})|X_k],
	\]
	where $\Psi_s$ are defined in (\ref{def:PsiAnd}).
	Hence, by the Hoeffding decomposition,  we decompose $\tilde{u}_{\gamma,s}$ as
	\begin{equation}\label{decomposition:tildeusforproof}
	\begin{aligned}
	\tilde{u}_{1,s}&=\frac{m}{n_1}\sum\limits_{k=1}^{n_1}h_{s}(\bX_k)+{\binom{n_1}{m}}^{-1}\Delta_{n_1,s},\\
	\tilde{u}_{2,s}&=\frac{m}{n_2}\sum_{k=1}^{n_2}h_{s}(\bY_k)+{\binom{n_2}{m}}^{-1}\Delta_{n_2,s},
	\end{aligned}
	\end{equation}
	where  we define $\Delta_{n_1,s}$ and $\Delta_{n_1,s}$ as 
	\[
	\begin{aligned}
	\Delta_{n_1,s}&=\sum\limits_{1\le k_1<k_2<\ldots<k_m\le n_1}\Big(\Psi_{s}(\bX_{k_1},\ldots,\bX_{k_m})-\sum_{\ell=1}^{m}h_{s}(\bX_{k_\ell})\Big),\\
	\Delta_{n_2,s}&=\sum\limits_{1\le k_1<k_2<\ldots<k_m \le n_2}\Big(\Psi_{s}(\bY_{k_1},\ldots,\bY_{k_m})-\sum_{\ell =1}^{m}h_{s}(\bY_{k_\ell})\Big).
	\end{aligned}
	\]
	We then  use $m\sum_{k=1}^{n_1}h_s(\bX_k)/n_1$ and  $m\sum_{k=1}^{n_2}h_s(\bY_k)/n_2$ to approximate $\tilde{u}_{1,s}$ and $\tilde{u}_{2,s}$. By setting $\bSigma_{1}:=(\sigma_{1,st}), \bSigma_{2}:=(\sigma_{2,st})\in \reals^{q\times q}$ with
	\begin{equation}\label{def:sigmastN1+}
	\sigma_{1,st}=\E\big(h_s(\bX)h_t(\bX)\big)\hspace{2em}{\rm and}\hspace{2em}\sigma_{2,st}=\E\big(h_s(\bY)h_t(\bY)\big),
	\end{equation}
	considering  $\hat{v}_{\gamma,s}=m^2\hat{\sigma}_{\gamma,ss}$, as $n\rightarrow\infty$
	we  have
	$\hat{v}_{\gamma,s}\rightarrow m^2\sigma_{\gamma,ss}$, which motivates us to define
	\begin{equation}\label{def:HsN1+}
	H_s^N={\Big(\dfrac{1}{n_1}\sum\limits_{k=1}^{n_1}h_s(\bX_k)-\dfrac{1}{n_2}\sum\limits_{k=1}^{n_2} h_s(\bY_k)\Big)}/{\sqrt{\sigma_{1,ss}/n_1+\sigma_{2,ss}/n_2}}.
	\end{equation}
	Moreover, by setting  $\bH^N=(H_1^N,\ldots,H_q^N)^\top$, we have that $\bH^N$  approximates  $\bN$, and the approximation error is characterized by the following lemma.  
	\begin{lemma}\label{lemma:NhatHApproximationN1+}
		Assumptions {\bf (A)}, {\bf (E)}, {\bf (M1)}, and {\bf (M2)} hold.
		Under $\Hb_0$ of (\ref{def:htwo-sampleu}) there is a constant $C>0$ such that as $n\rightarrow\infty$, we have
		\begin{align}
		\P\Big(\|\bN -\bH^N\|_{(s_0,p)}> \varepsilon\Big) = o(1),
		\end{align}
		where $\varepsilon=Cs_0\log^2(qn) n^{-1/2}$.
	\end{lemma}
	
	The proofs of Lemma \ref{lemma:NhatHApproximationN1+} is in Appendix
	\ref{proof:lemma:NhatHApproximationN1+}  of supplementary materials. By the definition of $H_s^N$ in (\ref{def:HsN1+}), 
	$\bH^N$ is a sum of  random vectors with zero mean and covariance matrix $\Rb_{12}$, where we set
	\begin{equation}\label{def:SigmaD12N1+}
	\Rb_{12}:=\Db_{12}^{-1/2}\bSigma_{12}\Db_{12}^{-1/2}
	\end{equation}
	with $\bSigma_{12}=\bSigma_1/n_1+\bSigma_2/n_2$ and $\Db_{12}={\rm Diag}(\bSigma_{12})$.
	Therefore, by the central limit theorem, we can use the Gaussian random vector  $\bG^N\sim N(\zero, \bR_{12})$ to  approximate $\bH^N$. To characterize the approximation error, considering 
	$A_z:=\{\vb,\|\vb\|_{(s_0,p)}\le z\}\in \mathcal{A}_{s_0}$, by Lemma \ref{lemma:CCK2}, there is  $\zeta_0>0$ such that 
	\begin{equation}\label{ineq:GaussApproxN1+}
	\sup_{z}\Big|\P(\|\bH^N\|_{(s_0,p)}\!\le\! z)\!\!-\!\!\P(\|\bG^N\|_{(s_0,p)}\!\le\! z)\Big|\!\le Cn^{-\zeta_0}
	\end{equation}
	where the constant $C$ only depends on $K$ and $b$.
	We then use $\bG^N$ as
	the approximation for $\bN$. 
	
	{\bf Step (ii).} In this step, we aim to obtain the distribution of  $\bN^b|\mathcal{X}, \mathcal{Y}$. For this, we rewrite $\hat{u}^b_{1,s}$ and $\hat{u}^b_{2,s}$  in (\ref{def:uhatbs}) as
	\begin{equation}
	\hat{u}^b_{1,s}=\frac{m}{n_1}\sum_{k=1}^{n_1}(Q_{1k,s}-\hat{u}_{1,s})\varepsilon^b_{1,k},~\hat{u}^b_{2,s}=\frac{m}{n_2}\sum_{k=1}^{n_2}(Q_{2k,s}-\hat{u}_{2,s})\varepsilon^b_{2,k},
	\end{equation}
	where $Q_{1k,s}$ and $Q_{2k,s}$ are defined in (\ref{def:qalphas}). Considering that $\varepsilon^b_{\gamma,1},\ldots,\varepsilon^b_{\gamma,n_\gamma}$ are i.i.d. standard normal random variables,  therefore given $\mathcal{X}$ and $\mathcal{Y}$, $\hat{\bu}^b_{\gamma}:=(\hat{u}^b_{\gamma,1},\ldots,\hat{u}^b_{\gamma, q})$ follows $N(\zero, m^2\hat{\bSigma}_{\gamma}/n_\gamma)$ with $	\hat{\bSigma}_{\gamma}:=(\hat{\sigma}_{\gamma,st})\in\reals^{q\times q}$ where
	\begin{equation}\label{def:sigmahatinProf}
\hat{\sigma}_{\gamma,st}=\frac{1}{n_1}\sum_{k=1}^{n_1}(Q_{\gamma k,s}-\hat{u}_{\gamma,s})(Q_{\gamma k,t}-\hat{u}_{\gamma,t}).
	\end{equation}
	Apparently, by the definition of $\hat{v}_{\gamma,s}$ in (\ref{def:vhat}) we have $\hat{v}_{\gamma,s}=
	m^2\hat{\sigma}_{\gamma,ss}$. Therefore, by setting
	$
	\hat{\bSigma}_{12}=\hat{\bSigma}_{1}/n_1+\hat{\bSigma}_{2}/n_2$   and $\hat{\Db}_{12}={\rm Diag}(\hat{\bSigma}_{12}),
	$
	we have
	\[
	\bN^b|\mathcal{X},\mathcal{Y}=m^{-1}\hat{\Db}_{12}^{-1/2} (\hat{\bu}^b_1-\hat{\bu}^b_2)|\mathcal{X},\mathcal{Y}\sim N(\zero,\hat{\Rb}_{12}),
	\]
	where  $\hat{\Rb}_{12}=\hat{\Db}_{12}^{-1/2}\hat{\bSigma}_{12}\hat{\Db}_{12}^{-1/2}$.
	
	{\bf Step (iii).} In this step, we combine results from previous two steps to justify the
	bootstrap procedure, i.e., we aim to prove 
	\[
	\sup_{z\in(0,\infty)}\Big|\P(N_{(s_0,p)}>z )-\P(N^b_{(s_0,p)}>z |\mathcal{X}, \mathcal{Y})\Big|=o_p(1).
	\]
	For this, we need both the lower and upper bounds of $\P\big(N_{(s_0,p)}> z\big)-\P\big(N^b_{(s_0,p)}>z|\mathcal{X},\mathcal{Y}\big)$. We first presents how to obtain the upper bounds. By the triangle inequality, we have
	\begin{equation}\label{ineq:defrho1}
	\P(\|\bN\|_{(s_0,p)}\!>\!z)\!\le\!
	\P(\|\bH^N\|_{(s_0,p)}\!>\!z-\varepsilon)+
	\underbrace{
		\P(\|\bN-\bH^N\|_{(s_0,p)}\!>\!\varepsilon)}_{\rho_1}.
	\end{equation}
	By Lemmas \ref{lemma:NhatHApproximationN1+}, we have $\rho_1=o(1)$. We then bound $\P(\|\bH^N\|_{(s_0,p)}>z-\varepsilon)$. For this, we have
	\begin{equation}\label{ineq:defrho2}
	\P(\|\bH^N\|_{(s_0,p)}>z-\varepsilon)\le  \rho_2 +
	\P(\|\bG^N\|_{(s_0,p)}>z-\varepsilon),
	\end{equation}
	where $
	\rho_2 = \sup_{x>0}\Big|\P(\|\bH^N\|_{(s_0,p)}> x)-\P(\|\bG^N\|_{(s_0,p)}> x)\Big|.
	$ By (\ref{ineq:GaussApproxN1+}),  we have $
	\rho_2\le Cn^{-\zeta_0}$	which yields
	\begin{equation}\label{ineq:D1D2+}
	\P(\|\bN\|_{(s_0,p)}> z)\le \underbrace{\P(\|\bG^N\|_{(s_0,p)}>z-\varepsilon)}_{\rho_3} + o(1),
	\end{equation}
	as $n\rightarrow\infty$. We then  decompose $\rho_3$ as $\rho_3=
	\P(\|\bG^N\|_{(s_0,p)}>z) + \rho_4
	$ with \[\rho_4=\P(z-\varepsilon<\|\bG^N\|_{(s_0,p)}\le z).\]
	To control $\rho_4$, by utilizing the anti-concentration inequality for the 
	the Gaussian random vector in Lemma \ref{lemma:gaussianDif}, we introducing the following lemma.
	\begin{lemma}\label{lemma:anticoncentrate}
		Assumptions {\bf (A)} and {\bf (M1)} hold. For any $z>0$ and $\varepsilon=O(s_0\log^2(qn)n^{-1/2})$, we have 
		$\P(z-\varepsilon<\|\bG^N\|_{(s_0,p)}\le z)=o(1)$ as $n\rightarrow\infty$.
	\end{lemma}
	The proof of Lemma \ref{lemma:anticoncentrate} is in Appendix
	\ref{proof:lemma:anticoncentrate}  of supplementary materials.
	By Lemma \ref{lemma:anticoncentrate}, we then  have
	\[
	\P(\|\bN\|_{(s_0,p)}>z)\le \P(\|\bG^N\|_{(s_0,p)}>z)+ o(1),
	\]
	as $n\rightarrow\infty$.	As is shown in {\bf Step (ii)}, under the null hypothesis we have $\bN^b|\mathcal{X},\mathcal{Y}\sim N(\zero,\hat{\Rb}_{12})$.  Considering $\bG^N\sim N(\zero,\Rb_{12})$, we have \begin{equation}\label{ineq:D5}
	\P(\|\bN\|_{(s_0,p)}>z)- \P(\|\bN^b\|_{(s_0,p)}>z|\mathcal{X},\mathcal{Y})\le \hat{D}_5+o(1).
	\end{equation}
	with $
	\hat{D}_5=\sup_{z>0}\Big|\P(\|\bG^N\|_{(s_0,p)}>z)
	-\P(\|\bN^b\|_{(s_0,p)}>z|\mathcal{X},\mathcal{Y})\Big|$. The following lemma presents the upper bound of $\hat{D}_5$.
	\begin{lemma}\label{lemma:boundhatD5}
		Assumptions {\bf (A)}, {\bf (E)}, {\bf (M1)} and {\bf (M2)} hold. With probability at least $1-C_1n^{-1}$,  we have $\hat{D}_5 = o_p(1)$ as $n\rightarrow\infty.$
	\end{lemma}
	The proof of Lemma \ref{lemma:boundhatD5} is in Appendix
	\ref{proof:lemma:boundhatD5}  of supplementary materials.
	Therefore, we have
	\begin{equation}\label{ineq:HalfBoundRes}
	\sup_{z>0}\Big(\P(\|\bN\|_{(s_0,p)}>z)- \P(\|\bN^b\|_{(s_0,p)}>z|\mathcal{X},\mathcal{Y})\Big)= o_p(1),
	\end{equation}
	uniformly for any $z>0$.  We can similarly construct  the lower bound and obtain 
	\[
	\sup_{z>0}\Big|\P(\|\bN\|_{(s_0,p)}>z)- \P(\|\bN^b\|_{(s_0,p)}>z|\mathcal{X},\mathcal{Y})|=o_p(1),
	\]
	which finishes the proof of (\ref{ineq:CoreN1+}) in Theorem \ref{therom:CoreWN1+}.
\end{proof}
\subsection{Proof of Corollary \ref{corollary:size} }
\label{proof:corollary:size}
\begin{proof}
	In Corollary \ref{corollary:size}, we aim to prove (\ref{size_one_sample}) and (\ref{size_two_sample}).  As the proof of (\ref{size_one_sample}) is similar, we only prove (\ref{size_two_sample}).  As $\hat{P}^N_{(s_0,p)}-P^N_{(s_0,p)}\rightarrow0$ implies $\P_{\Hb_0}\big(T^N_{\alpha,(s_0,p)}=1\big)\rightarrow\alpha$, for proving (\ref{size_two_sample}) we only need to prove that as $n, B\rightarrow\infty$, we have 
	\begin{equation}\label{proof:ColGoalTwo}
	\hat{P}_{(s_0,p)}^N-P_{(s_0,p)}^N\rightarrow0,
	\end{equation}
	where  $\hat{P}_{(s_0,p)}^N$ is defined in (\ref{def:PhatWN}) and $P_{(s_0,p)}^N$ is the 
	oracle $P$-value of 	$N_{(s_0,p)}$.
	By introducing  
	\begin{equation}\label{def:FnFbnHat}
	\begin{aligned}
	F_{N,(s_0,p)}(z)&=\P\big(\|\bN\|_{(s_0,p)}\le z\big)\\
	\hat{F}_{N^b,(s_0,p)}(z)&= (B+1)^{-1}\Big(\sum\nolimits_{b=1}^{B}\ind\big\{N^b_{(s_0,p)}\le z|\mathcal{X},\mathcal{Y}\big\}+1\Big),
	\end{aligned}
	\end{equation} 
	consider the definitions of $\hat{P}^N_{(_s0,p)}$ and $P^N_{(s_0,p)}$, 
	we have 
	\begin{equation}\label{def:PNPNhat}
	\hat{P}_{(s_0,p)}^N=1-\hat{F}_{N^b,(s_0,p)}\big(N_{(s_0,p)}\big),~	P_{(s_0,p)}^N=1-F_{N,(s_0,p)}\big(N_{(s_0,p)}\big).
	\end{equation}
	According to  Theorems \ref{therom:CoreWN1+}, under Assumptions {\bf (A)}, {\bf (S)}, {\bf (E)},  {\bf (M1)}, and {\bf (M2)}, by setting $T_1= \big|1-F_{N^b,(s_0,p)}\big(N_{(s_0,p)}\big)-P_{(s_0,p)}^N\big|$ with 
	\begin{equation}\label{def:Fbn}
	F_{N^b,(s_0,p)}(z):=\P\big(\|\bN^b\|_{(s_0,p)}\le z\big|\mathcal{X},\mathcal{Y}\big),
	\end{equation}
	we have $T_1\rightarrow 0$ as $n\rightarrow\infty$.
	Considering (\ref{def:FnFbnHat}) and (\ref{def:PNPNhat}),  we use the triangle inequality to obtain $\Big|P_{(s_0,p)}^N-\hat{P}_{(s_0,p)}^N\Big|\le T_1 +T_2$ with
	\[
	T_2 =\Big|F_{N^b,(s_0,p)}\big(N_{(s_0,p)}\big)-\hat{F}_{N^b,(s_0,p)}\big(N_{(s_0,p)}\big)\Big|
	\]
	By Massart's inequality (see Section 1.5 in \cite{dudley2014uniform:app}), we have 	
	\begin{equation}\label{lim:Flimit_part2}
	\sup_{z\in\reals}\Big|\hat{F}_{N^b,(s_0,p)}(z)-F_{N^b,(s_0,p)}(z)\Big|\rightarrow0,\hspace{2em}\text{as}~n, B\rightarrow\infty.
	\end{equation}
	Therefor, as $n$, $B\rightarrow\infty$, we have $T_2\rightarrow0$, which finishes the proof.

\end{proof}
\subsection{Proof of Theorem \ref{theorem:power}}
\label{proof:theorem:power}
\begin{proof} 
	For simplicity, we only consider the two-sample problem.
	The proof proceeds in two steps.  In the first step, we
	give an upper bound of the oracle critical  value 
	\[t^N_{\alpha,(s_0,p)}=\inf\Big\{t\in\reals:\P\big(\|\bN^b\|_{(s_0,p)}\le t|\mathcal{X},\mathcal{Y}\big)> \alpha\Big\}.\]
	 In the second step, with the obtained upper bound of $t^N_{\alpha,(s_0,p)}$, we construct a lower bound of
	$
	\P\big(N_{(s_0,p)}>t^N_{\alpha,(s_0,p)}\big)
	$. By  showing that this lower bound goes to $1$ under (\ref{assump:powerTwosample}), we have \[
	\P\big(N_{(s_0,p)}>t^N_{\alpha,(s_0,p)}\big)\rightarrow 1,
	\]
	as $n,q\rightarrow\infty$. Considering that $\hat{t}^N_{\alpha,(s_0,p)}$ is a bootstrap estimator for $t^N_{\alpha,(s_0,p)}$,  under (\ref{assump:powerTwosample}) we then have $\P\big(N_{(s_0,p)}>\hat{t}^N_{\alpha,(s_0,p)}\big)\rightarrow 1$, as $n,B\rightarrow\infty$.
	
	{\bf Step (i).} In this step, we give an upper bound of $t^N_{\alpha,(s_0,p)}$. By the definition of $\bN^b$ in (\ref{def:WNB}),  $\bN^b|\mathcal{X},\mathcal{Y}$ is a $q$-dimensional Gaussian random vector with standard normal entries. According to Lemma \ref{lemma:foldedNormal}, by setting $\sigma =1$ and $t=\sqrt{2\log q}$ we have 
	\begin{equation}\label{ineq:ENbInfty}
	\E\big[\|\bN^b\|_{\infty}|\mathcal{X},\mathcal{Y}\big]\le \sqrt{2\log q}+\frac{1}{\sqrt{2\log q}}=
	\sqrt{2\log q}\big(1+\{2\log q\}^{-1}\big).
	\end{equation}
	By Theorem 5.8 of \cite{boucheron2013concentration:app}, we have 
	\begin{equation}\label{ineq:NbInfty}
	\P\Big(\|\bN^b\|_{\infty}\ge \E\big[\|\bN^b\|_{\infty}|\mathcal{X},\mathcal{Y}\big]+u\Big|\mathcal{X},\mathcal{Y}\Big)<\exp(-u^2/2).
	\end{equation}
	By setting $c_\alpha$ as the $\alpha$-quantile of $\|\bN^b\|_\infty|\mathcal{X},\mathcal{Y}$, combining (\ref{ineq:ENbInfty}) and (\ref{ineq:NbInfty}), we have 
	\begin{equation}\label{ineq:qAlpha}
	c_{1-\alpha} \le  \sqrt{2\log q}\big(1+\{2\log q\}^{-1}\big)+ \sqrt{2\log (1/\alpha)}.
	\end{equation}
	Considering that $t^N_{\alpha,(s_0,p)}$ is the $1-\alpha$ quantile of $\|\bN^b\|_{(s_0,p)}|\mathcal{X}, \mathcal{Y}$,
	by the inequality $\|\bN^b\|_{(s_0,p)}\le s_0^{1/p}\|\bN^b\|_\infty$, we then have 
	$t^N_{\alpha,(s_0,p)}\le s_0^{1/p}c_{1-\alpha}$. Therefore, by (\ref{ineq:qAlpha}) we have 
	\begin{equation}\label{ineq:tAlpha}
	t^N_{\alpha,(s_0,p)}\le  s_0^{1/p}\Big(\sqrt{2\log q}\big(1+\{2\log q\}^{-1}\big)+\sqrt{2\log (1/\alpha)}\Big).
	\end{equation}

	{\bf Step (ii)} In this step, we aim to obtain an lower bound of $\P\big(N_{(s_0,p)}>t^N_{\alpha,(s_0,p)}\big)$.
	By (\ref{ineq:tAlpha}), we have  $	\P(N_{(s_0,p)}>t^N_{\alpha,(s_0,p)}\big)\ge L_1^N$, where
	\begin{equation}\label{ineq:LB0}
L_1^N = \P\bigg(N_{(s_0,p)}>s_0^{1/p}\Big(\sqrt{2\log q}\big(1+\{2\log q\}^{-1}\big)+\sqrt{2\log (1/\alpha)}\Big)\bigg).
	\end{equation}
	To obtain the lower bound of $L_1^N$, we need some  additional notations. By setting $N_s$ as 
	$
	N_s=(\hat{u}_{1,s}-\hat{u}_{2,s})/\sqrt{\hat{v}_{1,s}/n_1+\hat{v}_{2,s}/n_2},
	$
	in (\ref{def:WN}), 
	we define $N_{(s_0,p)}=\|\bN\|_{(s_0,p)}$,
	where $\bN=(N_1,\ldots,N_q)^\top$.
	Under the alternative hypothesis, $u_{1,s}=u_{2,s}$ cannot hold for all $s\in\{1,\ldots,q\}$, which motivates us to define 
	\begin{equation}\label{def:NS1N1}
	N_s^1=\frac{\hat{u}_{1,s}-\hat{u}_{2,s}-u_{1,s}+u_{2,s}}
	{\sqrt{\hat{v}_{1,s}/n_1+\hat{v}_{2,s}/n_2}}\hspace{1em}
	{\rm and}\hspace{1em}\bN^1=(N^1_1,\ldots,N^1_q)^\top.
	\end{equation}
	Considering that  $\hat{v}_{\gamma,s}$ is the variance estimator for  $\sqrt{n_{\gamma}}\hat{u}_{\gamma,s}$ and that $\hat{v}_{\gamma,s}$ has the limit $m^2\sigma_{\gamma,ss}$ as $n_\gamma\rightarrow\infty$, we introduce $\bD_2=(D_{2,1},\ldots,D_{2,q})^\top$ and $\hat{\bD}_2=(\hat{D}_{2,1},\ldots,\hat{D}_{2,q})^\top$, where 
	\begin{equation}\label{def:DshatDs}
	\begin{aligned}
	D_{2,s}&=|u_{1,s}-u_{2,s}|/\sqrt{m^2\sigma_{1,ss}/n_1+
		m^2\sigma_{2,ss}/n_2}\\
	\hat{D}_{2,s}&=|u_{1,s}-u_{2,s}|/\sqrt{\hat{v}_{1,s}/n_1+
		\hat{v}_{2,s}/n_2}.
	\end{aligned}
\end{equation} Without loss of generality, we assume that  largest $s_0$
	entries of $\bD_2$ is $k_1$,  $k_2$, $\ldots,$ $k_{s_0}$.
	Therefore, by setting $\bk=(k_1,\ldots,k_{s_0})^\top$  under (\ref{assump:powerTwosample}) we have 
	\begin{equation}\label{assumption:powerInproof}
	\|\bD_2\|_{(s_0,p)}= \|(\bD_{2})_{\bk}\|_p\ge s_0(1+\varepsilon_n)\Big(\sqrt{2\log q}+\sqrt{2\log (1/\alpha)}\Big),
	\end{equation} 
	where we set $\varepsilon_n\rightarrow 0$ and $\varepsilon_n\sqrt{\log q}\rightarrow\infty$ as $n\rightarrow\infty$.
	By the definition of $(s_0,p)$ distance and the triangle inequality, we have
	\begin{equation}\label{ineq:Nlowbound1}
	N_{(s_0,p)}\ge \|\bN_{\bk}\|_p\ge
	\|(\hat{\bD}_2)_{\bk}\|_p-\|\bN^1_{\bk}\|_p.
	\end{equation}
	As we impose conditions on $\bD_2$ not on $\hat{\bD}_2$, by the definitions of $\bD_2$ and $\hat{\bD}_2$ in (\ref{def:DshatDs})  we need the estimation error of 
	$\hat{v}_{\gamma,s}$. By Lemme \ref{lemma:sigmaRErrors}, considering Assumption {\bf (M1)}, 
	for $m>1$, with probability at least $1-C_1n^{-1}$, we have 
	\begin{equation}\label{ineq:vsigma1+}
	\max_{\gamma=1,2\atop s=1,\ldots,q}\Big|\sqrt{\frac{\hat{v}_{\gamma,s}}{m^2\sigma_{\gamma,ss}}}-1\Big|\le C\frac{\log^{3/2}(qn)}{\sqrt{n}},
	\end{equation}
	when $n$ is sufficiently large. Similarly, for $m=1$ and sufficiently large $n$ with probability at least $1-C_1n^{-1}$ we have 
	 \begin{equation}\label{ineq:vsigma1}
	\max_{\gamma=1,2\atop
		s=1,\dots,q}\Big|\sqrt{\frac{\hat{v}_{\gamma,s}}{\sigma_{\gamma,ss}}}-1\Big|\le C \sqrt{\frac{\log(qn)}{n}}+C\frac{\log^2(qn)}{n}.
	\end{equation}
	Therefore, we  introduce the event $\mathcal{E}_0(x)$ as
	\[
	\mathcal{E}_0(x)=\Bigg\{
	\max_{\gamma=1,2\atop s=1,\ldots,q}\Big|\sqrt{\frac{\hat{v}_{\gamma,s}}{m^2\sigma_{\gamma,ss}}}-1\Big|\le x\Bigg\}.
	\]
	We set $x\asymp\log^{3/2}(qn)/\sqrt{n}$ for $m>1$ and $x\asymp \sqrt{\log(qn)/n}+\log^2(qn)/n$ for $m=1$. We then have $\P(\mathcal{E}_0(x)^c)\preceq n^{-1}$.
	By (\ref{ineq:Nlowbound1}), under $\mathcal{E}_0(x)$ we have
	\begin{equation}\label{ineq:LB1}
	N_{(s_0,p)}\ge \underbrace{ \frac{1}{1+x}\|\bD_2\|_{(s_0,p)}-\|\bN^1_{\bk}\|_{p}}_{L_2^N}.
	\end{equation}
	Therefore, by partitioning the event based on $\mathcal{E}_0(x)$,  we use (\ref{ineq:LB1}) to obtain
	\[
	L^N_1 \ge \P\bigg(L_2^N>s_0\big(1+\{2\log q\}^{-1}\big)\Big(\sqrt{2\log q}+\sqrt{2\log (1/\alpha)}\Big),\mathcal{E}_0(x)\bigg).
	\]
	Considering  (\ref{assumption:powerInproof}),
	by choosing $u$ satisfying $(1+x)\big(1+u+\{2\log q\}^{-1}\big)=(1+\varepsilon_n)
	$ we have 
	\begin{equation}\label{ineq:LB2}
	L^N_1\ge\P\Big(\|\bN_{\bk}^1\|_p<s_0u\sqrt{2\log q}\Big).
	\end{equation}
	By the triangle inequality, for $i\in\{1,\ldots, s_0\}$ we have 
	\begin{equation}\label{ineq:LB3}
	\P\Big(\|\bN_{\bk}^1\|_p\ge s_0u\sqrt{2\log q}\Big)\le s_0\max_{1\le i\le s_0}\P\Big(|N^1_{k_i}|\ge u\sqrt{2\log q}\Big).
	\end{equation}
	Therefore, combining (\ref{ineq:LB2}) and (\ref{ineq:LB3}) we have 
	\[
	L^N_1\ge 1-\P\Big(\|\bN_{\bk}^1\|_p\ge s_0u\sqrt{2\log q}\Big)\ge 1-s_0\max_{1\le i\le s_0}\P\Big(|N^1_{k_i}|\ge u\sqrt{2\log q}\Big).
	\]
	By the definition of $L_1$ in (\ref{ineq:LB0}), to prove $L_1\rightarrow 1$  we only need to obtain 
	\begin{equation}\label{limit:Nupperbound}
	s_0\P\big(|N^1_{k_i}|\ge u\sqrt{2\log q}\big)\rightarrow 0,
	\end{equation}
	uniformly as $n,q\rightarrow\infty$.
	For this, we introduce the following lemma. 
	\begin{lemma}\label{lemma:NS1BMEM}
	Under Assumptions {\bf (A)$'$}, {\bf (E)}, {\bf (M1)}, and {\bf (M2)}, as $n, q\rightarrow\infty$, we have 
	\begin{equation}\label{limit:s0tail}
	s_0\max_{s=1,\ldots, q}\P\big(|N^1_{s}|\ge u\sqrt{2\log q}\big)\rightarrow 0.
	\end{equation}
	\end{lemma}
	The detailed proof of Lemma \ref{lemma:NS1BMEM} is in Appendix \ref{proof:lemma:NS1BMEM}  of supplementary materials. By Lemma \ref{lemma:NS1BMEM}, we finish the proof.
\end{proof}

\subsection{Proof of Theorem \ref{thm:AdSize}}\label{proof:thm:AdSize}
\begin{proof}
	In Theorem \ref{thm:AdSize}, we aim to prove 
	\eqref{limit:Wad} and \eqref{limit:Nad}. As the proof of \eqref{limit:Wad} is similar, we only prove (\ref{limit:Nad}). The proof proceeds in two steps.
	In the first step,  by setting
	$F_{N,\rm ad}(z)=\P(N_{\rm ad}\le z|\mathcal{X},\mathcal{Y})$ and  $\tilde{F}_{N,\rm ad}(z)=\P(\tilde{N}_{\rm ad}\le z)$, we prove that as $n, B\rightarrow\infty$, we have
	\begin{equation}\label{limit:AdsizeProofPart1}
	\tilde{F}_{N,{\rm ad}}(\tilde{N}_{\rm ad})-F_{N,\rm ad}(N_{\rm ad})\rightarrow0,
	\end{equation}
	where $N_{\rm ad}$ and $\tilde{N}_{\rm ad}$ are defined in (\ref{def:WadNad}) and (\ref{def:tildaWNad}). In the second step, we prove that
	\begin{equation}\label{limit:AdsizeProofPart2}
	F_{N,\rm ad}(N_{\rm ad})-\hat{P}^N_{\rm ad}\rightarrow 0,
	\end{equation}
	as $n,B\rightarrow\infty$.
	Combining (\ref{limit:AdsizeProofPart1}) and (\ref{limit:AdsizeProofPart2}), we can easily obtain (\ref{limit:Nad}).
	
	{\bf Step (i).}  In this step, we aim to prove (\ref{limit:AdsizeProofPart1}). For this, we need the following lemma to analyze the difference between  the cumulative distribution functions  of $N_{\rm ad}$ and $\tilde{N}_{\rm ad}$.

	\begin{lemma}\label{lemma:AdSize}
		Assumptions  {\bf (A)$''$}, {\bf (E)}, {\bf (M1)} and {\bf (M2)}  hold. Under $\Hb_0$ of (\ref{def:htwo-sampleu}),  we have that for any $\epsilon>0$
		\begin{equation}\label{dif:CDFNad}
		\sup_{z\in[\epsilon,1-\epsilon]}\Big|F_{N,\rm ad}(z)-\tilde{F}_{N,{\rm ad}}(z)\Big| =0,
		\end{equation}
		as $n,B\rightarrow\infty$.
	\end{lemma}
	The proof of Lemma \ref{lemma:AdSize} is in Appendix \ref{proof:lemma:AdSize}  of supplementary materials.
	After introducing Lemma \ref{lemma:AdSize}, we then prove (\ref{limit:AdsizeProofPart1}). In detail, we aim to prove that for any $\delta, \epsilon'>0$ we have 
	\begin{equation}\label{bound:Delta_1}
	\underbrace{\P\Big(| \tilde{F}_{N,{\rm ad}}(\tilde{N}_{\rm ad})-F_{N,\rm ad}(N_{\rm ad})|\ge \delta\Big)}_{\Delta_1} < \epsilon'
	\end{equation}
	as $n\rightarrow\infty$. By plugging in $F_{N,\rm ad}(\tilde{N}_{\rm ad})$, we use the triangle inequality to obtain $\Delta_1\le \Delta_2 +\Delta_3$, where
	\begin{equation}\label{def:Dleta123}
	\begin{aligned}
	\Delta_2 =&\P\Big(|\tilde{F}_{N,\rm ad}(\tilde{N}_{\rm ad})\!-\!F_{N,\rm ad}(\tilde{N}_{\rm ad})|\ge \delta/2\Big),
	\\\Delta_3=&
		\P\Big(|F_{N,\rm ad}(\tilde{N}_{\rm ad})\!-\!F_{N,\rm ad}(N_{\rm ad})|\ge \delta/2\Big).
	\end{aligned}
	\end{equation}
	We then separately bound $\Delta_2$ and $ \Delta_3$.
	To prove (\ref{bound:Delta_1}). We only need to show  both $\Delta_2<\epsilon'/2$ and  $\Delta_3<\epsilon'/2$ hold as $n$  and $B$
	are sufficiently large. For $\Delta_2$, by setting $\mathcal{E}_{\tilde{N},\rm ad}(\epsilon):=\{\tilde{N}_{\rm ad}\in [\epsilon,1-\epsilon]\}$, we  can bound $\Delta_2$ by 
	\begin{equation}\label{ineq:boundDleta2_0}
	\Delta_2\le \P\Big(|\tilde{F}_{N,\rm ad}(\tilde{N}_{\rm ad})-F_{N,\rm ad}(\tilde{N}_{\rm ad})|\ge \delta/2\cap \mathcal{E}_{\tilde{N},\rm ad}(\epsilon)\Big)+ \Delta_4,
	\end{equation}
	where $\Delta_4 = \P\big(\big(\mathcal{E}_{\tilde{N},\rm ad}(\epsilon)\big)^c\big)$.
	By the definition of $\tilde{N}_{\rm ad}$ in (\ref{def:tildaWNad}), by choosing $\epsilon$ small enough, we have $\Delta_4 \le \epsilon'/4$.
	By Lemma \ref{lemma:AdSize} and the definition of $\mathcal{E}_{\tilde{N},\rm ad}(\epsilon)$, we also have
	\begin{equation}\label{ineq:boundDelta2_2}
	\P\Big(|\tilde{F}_{N,\rm ad}(\tilde{N}_{\rm ad})-F_{N,\rm ad}(\tilde{N}_{\rm ad})|\ge \delta/2\cap \mathcal{E}_{\tilde{N},\rm ad}(\epsilon)\Big)\le \epsilon'/4,
	\end{equation}
	for sufficiently large $n$ and $B$. Hence,  we have $\Delta_2 \le \epsilon'/2$ holds as $n$ and $B$ are sufficiently large.  After the proof foe $\Delta_2$, we then bound $\Delta_3$. 
	By the definition of $\tilde{N}_{\rm ad}$ in (\ref{def:tildaWNad}) and Corollary \ref{corollary:size}, we have 
	\begin{equation}\label{limit:tildeNadNad}
	|\tilde{N}_{\rm ad}-N_{\rm ad}|\rightarrow 0,\hspace{2em} \text{as $n$, $B$}\rightarrow\infty. 
	\end{equation}
	Therefore, we obtain that $f_{N,\rm ad}(z)=F'_{N,\rm ad}(z)$ is uniformly bounded for sufficiently large $n,B$. Hence, there is a constant $C$ such that 
	\begin{equation}\label{ineq:LipFNad}
	|F_{N,\rm ad}(\tilde{N}_{\rm ad})-F_{N,\rm ad}(N_{\rm ad})|\le C|\tilde{N}_{\rm ad}-N_{\rm ad}|.
	\end{equation}
	Combining (\ref{def:Dleta123}), (\ref{limit:tildeNadNad}), and (\ref{ineq:LipFNad}), we have $\Delta_3\le \epsilon'/2$ for sufficiently large $n$ and $B$. Therefore, we finish the proof of (\ref{limit:AdsizeProofPart1}).
	
	{\bf Step (ii).} In this step, we aim to prove (\ref{limit:AdsizeProofPart2}). For this,  we introduce 
	\begin{equation}\label{def:Nbad}
	\begin{aligned}
	&F_{N^b,(s_0,p)}=\P\Big(N_{(s_0,p)}^b\le z|\mathcal{X},\mathcal{Y}\Big),\\
&	N_{\rm ad}^b=\min_{p\in\mathcal{P}}\Big(1-F_{N^b,(s_0,p)}(N_{(s_0,p)}^b)\Big),
	\end{aligned}
	\end{equation}
	where $N_{(s_0,p)}^b$ is defined in (\ref{def:WNbs0p}). Therefore, we define the cumulation distribution function of $N_{\rm ad}^b|\mathcal{X},\mathcal{Y}$ as 
	\begin{equation}\label{def:Fnbad}
	F_{N^b,\rm ad}(z)=\P(N_{\rm ad}^b\le z|\mathcal{X},\mathcal{Y}).
	\end{equation}
	Considering the definition of $\hat{P}^N_{\rm ad}$ in (\ref{def:hatPNad}),
	by setting 
	\begin{equation}\label{def:hatFNad'}
	\hat{F}_{N,\rm ad'}(z)=\Bigg(\sum_{b=1}^B\ind\{N_{\rm ad'}^b\le z|\mathcal{X},\mathcal{Y}\}+1\Bigg)\big/(B+1),
	\end{equation}
	we have $\hat{P}^N_{\rm ad}=\hat{F}_{N,\rm ad'}(N_{\rm ad})$.
	
	To prove $F_{N,\rm ad}(N_{\rm ad})-\hat{P}^N_{\rm ad}\rightarrow 0$, by plugging in $F_{N^b,\rm ad}(N_{\rm ad})$ and using the triangle inequality, it is sufficient to prove 
	\begin{equation}\label{limit:adsizeStepii}
F_{N,\rm ad}(N_{\rm ad})-F_{N^b,\rm ad}(N_{\rm ad})\rightarrow 0\hspace{1em}{\rm and}\hspace{1em} F_{N^b,\rm ad}(N_{\rm ad})-\hat{P}^N_{\rm ad}\rightarrow 0,
	\end{equation}
	as $n$, $B\rightarrow\infty$. 
	To prove (\ref{limit:adsizeStepii}), we introduce the following two lemmas.
	\begin{lemma}\label{lemma:limitFnadFnbad}
		Assumptions   {\bf (A)$''$}, {\bf (E)}, {\bf (M1)}, and {\bf (M2)} hold. Under $\Hb_0$ of (\ref{def:htwo-sampleu}), 
		by setting $F_{N,\rm ad}(z)=\P(N_{\rm ad}\le z|\mathcal{X},\mathcal{Y})$ and  $F_{N^b,\rm ad }(z)=\P(N^b_{\rm ad}\le z|\mathcal{X},\mathcal{Y})$,
		we have 
		\begin{equation}\label{eq:limitFnadFnbad}
		\sup_{z\in [\epsilon,1-\epsilon]}|F_{N,\rm ad}(z)-F_{N^b,\rm ad}(z)|\rightarrow 0,\hspace{2em}\text{as}~n, B\rightarrow\infty,
		\end{equation}
		for any $\epsilon>0$.
	\end{lemma}
	
	\begin{lemma} \label{lemma:limitFnbadhatFad'}
		For any $\epsilon>0$,  we have that as $n, B\rightarrow\infty$, 
		\begin{equation}\label{eq:limitFnbadhatFad'}
		\sup_{z\in[\epsilon,1-\epsilon]}|F_{N^b,\rm ad}(z)-\hat{F}_{N,\rm ad'}(z)|\rightarrow 0,
		\end{equation}
		where $\hat{F}_{N,\rm ad'}(z)$ is defined in (\ref{def:hatFNad'}).
	\end{lemma}
	The proofs of Lemmas \ref{lemma:limitFnadFnbad} and \ref{lemma:limitFnbadhatFad'} are in Appendices \ref{proof:lemma:limitFnadFnbad} and \ref{proof:lemma:limitFnbadhatFad'}  of supplementary materials. 
	Let $\mathcal{E}_{N,\rm ad}(\epsilon)=\{N_{\rm ad}\in [\epsilon,1-\epsilon]\}$. Considering Lemmas \ref{lemma:limitFnadFnbad} and \ref{lemma:limitFnbadhatFad'}, by replacing $\mathcal{E}_{\tilde{N},\rm ad}(\epsilon)$ with $\mathcal{E}_{N,\rm ad}$, similarly to (\ref{ineq:boundDleta2_0}) and (\ref{ineq:boundDelta2_2}) we can prove (\ref{limit:adsizeStepii}), which finishes the proof of Theorem \ref{thm:AdSize}.

\end{proof}

\subsection{ Proof of Remark \ref{remark:Rs0p}
}\label{proof:rmark:Rs0p}
\begin{proof}
	For $\bG \sim N(\zero,\Rb)\in\reals^q$ with $q \geq 1$ fixed, the distribution of $\| \bG \|_{(s_0,p)}$ is absolutely continuous with respect to the Lebesgue measure and its density function $f^{\bG}_{(s_0,p)}$ is positive everywhere. This implies that for any $\epsilon>0$, $\min_{  c^{\bG}_{\epsilon,(s_0,p)} \leq z \leq c^{\bG}_{1-\epsilon,(s_0,p)}  } f^{\bG}_{(s_0,p)}(z)>0$. To prove the result after taking infimum over all positive integers $q$, it suffices to show that as long as $\Rb \in \mathcal{R}$, the limiting distribution of $\| \bG \|_{(s_0,p)}$ as $q\to \infty$ exists with an absolutely continuous density function. For this, we prove a stronger result, which characterizes the joint asymptotic distribution of the top $s_0$ order statistics of weakly dependent standard normal random variables. In detail, let $v^{(1)}, v^{(2)},\ldots, v^{(q)}$ be an ascending sequence of the magnitudes of the coordinates of $\vb \in\reals^q$ such that $0\le v^{(1)}\le v^{(2)}\le \ldots\le v^{(q)}$. Set $\bG=(G_1,\ldots,G_q)^\top\sim N(\zero,\Rb)$ with $\Rb\in\mathcal{R}$ and $\bG^I=(G_1^I,\ldots,G_q^I)^\top\sim N(\zero,\Ib_q)$. Moreover, let $\varphi_j(\bG)=G^{(q- j +1)}$ for $j=1,\ldots, q$ and $a_q=2 \log q - \log( \log q )$. For any $\xb=(x_1,x_2,\ldots,x_{s_0})$ with $x_1>x_2>\cdots>x_{s_0}>0$, by setting $f_{\rm ext}(t_1, \ldots,t_{s_0})= \exp\big( -\frac{1}{2} \sum_{j=1}^{s_0 - 1} t_j \big) g(t _{s_0} )  I( t_1> t_2 >\cdots > t_{s_0} )$, where $g(t) = 2^{-1}\pi^{-1/2} \exp(-t/2- \pi^{-1/2} e^{- t/2 })$,  we shall prove that as $q\rightarrow\infty$,
	\begin{equation}\label{joint.asym}
	\begin{aligned}
	&  \P\Big(\varphi_1^2(\bG)\le x_1+ a_q,\ldots,\varphi_{s_0}^2(\bG)\le x_{s_0}+ a_q \Big)   \\ 
	& \longrightarrow   \bigg( \frac{1}{2\sqrt{\pi}} \bigg)^{s_0 - 1}  
	\int_{-\infty}^{x_1}  \cdots  \int_{-\infty}^{x_{s_0}} f_{\rm ext}(t_1, \ldots,t_{s_0}) \, d t_{s_0} \cdots d t_1,
	\end{aligned}
	\end{equation}
	holds uniformly for $\Rb \in \mathcal{R}$.

	For simplicity, we only prove \eqref{joint.asym} for $s_0=2$, as  the general case can be dealt with similarly. Let $y_{jq}=\sqrt{ x_j + a_q} $ for $j=1,2$, and note
	\[
	\begin{aligned}
	&\Big\{\varphi_1(\bG)> \sqrt{ x_1+ a_q } ,	 \varphi_2(\bG) > \sqrt{ x_2 + a_q } \Big\} \\&\hspace{2em}=\bigcup_{ 1\leq i \ne j \leq q } \Big\{(|G_i|,|G_j|) >(y_{1q},y_{2q}) \Big\}=\bigcup_{k=1}^{2\bar{q} }\Big\{(|G_{i_k}|,|G_{j_k}|) >(y_{1q},y_{2q}) \Big\},
	\end{aligned}
	\]
	where $\{ (i_k, j_k) \}_{k=1}^{2\bar{q}} = \{ (1,2),(1,3) \ldots, (1,q), (2,1), (2,3) \ldots, (2,q) , \ldots, (q,q-1) \}$ and $\bar{q} = q(q-1)/2$. By the Bonferroni inequality, for any fixed $k<\bar{q}$, we have
	\begin{equation}\label{in.ex.formula}
	\begin{aligned}  
	\sum_{\ell=1}^{2k}(-1)^{\ell-1}E_\ell &\le
	\P\Big(\varphi_1(\bG)> \sqrt{ x_1+ a_q } , \varphi_2(\bG)> \sqrt{ x_2+ a_q } \Big)\\&\le \sum_{\ell=1}^{2k-1}(-1)^{\ell-1}E_\ell,
	\end{aligned}
	\end{equation} 
	where
	\[
	E_\ell\hspace{-0.3em}=\hspace{-1em}\sum_{1\le k_1< \cdots<k_\ell\le 2\bar{q}	}\hspace{-1em}\P\Big(   | G_{i_{k_1}} | > y_{1q} , | G_{j_{k_1}} |  > y_{2q} , \ldots , |G_{i_{k_\ell}}| > y_{1q} ,|G_{j_{k_\ell}}|  > y_{2q}  \Big).
	\]
	Moreover, for every $2\leq t\leq 2\ell$, define
	\begin{equation}\label{def:Ellt}
	E_{\ell, t}   =\hspace{-1em}\sum_{1\le k_1 <\cdots<k_\ell\le 2\bar{q} \atop \#\{i_{k_1},j_{k_1},\ldots,i_{k_\ell},j_{k_\ell}\}=t }\underbrace{\P\bigg(   \min_{1\leq \nu \leq \ell } | G_{i_{k_\nu}} | > y_{1q} , \min_{1\leq \nu \leq \ell } | G_{j_{k_\nu }} |  > y_{2q}  \bigg) }_{P_{k_1,\ldots,k_\ell}}.
	\end{equation}
	We  define index sets $\mathcal{I}^c$, $\mathcal{I}$, and $\mathcal{I}_k$ in  the same way as in the proof of Lemma~6 in \cite{tony2014two:app}. Therefore, we have $\mathcal{I}  = \cup_{k=1}^{t-1} \mathcal{I}_k$. Further, for $1\le i_1< \cdots <i_t\le q$, by defining
	\[
	Q(i_1,\ldots,i_t)\!\!=\!\!\Big\{\!1\le k_1\!< \!\cdots<k_\ell \le  2\bar{q}: \{i_{k_1},j_{k_1},\ldots,i_{k_\ell},j_{k_\ell}\}=\{i_1, \ldots,i_t\}\Big \},
	\]
	with $\# Q(i_1,\ldots,i_t) \le   \binom{t(t-1)}{\ell} $, we have 
	\[
	E_{\ell, t}  =\underbrace{\sum_{(i_1,\ldots,i_t)\in\mathcal{I}^{{\rm c}}}\sum_{(k_1,\ldots,k_\ell)\atop\in Q(i_1,\ldots,i_t) }
		P_{k_1,\ldots,k_\ell}}_{M_1(\ell, t)}+\underbrace{\sum_{(i_1,\ldots,i_t)\in\mathcal{I}}\sum_{(k_1,\ldots,k_\ell)\atop\in Q(i_1,\ldots,i_t)} P_{k_1,\ldots,k_\ell}}_{M_2(\ell, t)}.
	\]
	For $(k_1,\ldots,k_\ell)\in Q(i_1,\ldots,i_t)$ with $(i_1,\ldots,i_t)\in\mathcal{I}^{{\rm c}}$, a straightforward adaptation of the arguments used to prove (20) in \cite{tony2014two:app} yields that, as $q\rightarrow\infty$, 
	\begin{equation}
	P_{k_1,\ldots,k_\ell}= \{ 1+o(1) \} P^I_{k_1,\ldots,k_\ell},
	\end{equation}
	where $P^I_{k_1,\ldots,k_\ell}$ is defined in the same way as $P_{k_1,\ldots,k_\ell}$ in (\ref{def:Ellt}) by replacing $G_i$ with $G_i^I$. 
	
	For $(k_1,\ldots,k_\ell)\in Q(i_1,\ldots,i_t)$ with $(i_1,\ldots,i_t)\in\mathcal{I}_k$ for some $1\leq k\leq t-1$, considering $y_{1q}> y_{2q}$, we have  
	\[
	P_{k_1,\ldots,k_\ell}\le \P(|G_{i_1}|>y_{2q},\ldots,|G_{i_t}|>y_{2q}) := \tilde{P}_{i_1,\ldots,i_t}
	.\]
	Now, it follows from (21) in \cite{tony2014two:app} with slight modification that, as $q\rightarrow\infty$,
	\begin{equation}\label{ineq:M2}
	M_2(\ell, t) \le \sum_{(i_1,\ldots,i_t)\in\mathcal{I}}\binom{t(t-1)}{\ell}\tilde{P}_{i_1,\ldots,i_t}\
	=\binom{t(t-1)}{\ell}\sum_{(i_1,\ldots,i_t)\in\mathcal{I}}\tilde{P}_{i_1,\ldots,i_t}\rightarrow0.
	\end{equation}
	We define $M_2^I(\ell, t)$ by replacing entries of $\bG$ with the corresponding entries of $\bG^I$ in $M_2(\ell, t)$. Similarly to (\ref{ineq:M2}), we have $M_2^I(\ell, t)=o(1)$ as $q\rightarrow\infty$.  
	Therefore, as $q\rightarrow\infty$, we have 
	\[
	E_{\ell}=\sum_{t=2}^{2\ell}  E_{\ell, t}   = \{ 1+o(1) \} \sum_{1\le k_1< \cdots<k_\ell \le 2\bar{q} }P^I_{k_1,\ldots,k_\ell}	+	o(1).
	\]
	This, together with \eqref{in.ex.formula} implies that, as $q \to \infty$,
	\begin{equation}\label{sandwich}
	\begin{aligned}
	 &\{ 1 + o(1) \} \sum_{\ell=1}^{2k}(-1)^{\ell-1}  \hspace{-1em}\sum_{1\le k_1< \cdots<k_\ell \le 2\bar{q} }\hspace{-1em}P^I_{k_1,\ldots,k_\ell}\!  +\! o(1) \\\le& \P\Big(\varphi_1(\bG)\!>\! \sqrt{ x_1+ a_q } , \varphi_2(\bG)\!>\! \sqrt{ x_2 + a_q } \Big)  \\ 
 \le & \{ 1+o(1) \} \sum_{\ell=1}^{2k-1}(-1)^{\ell-1} \hspace{-1em}\sum_{1\le k_1< \cdots<k_\ell \le 2\bar{q} }P^I_{k_1,\ldots,k_\ell}  + o(1) . 
	\end{aligned}
	\end{equation}
	On the other hand, observing 
	\begin{equation} \label{ind.asymp}
	\begin{aligned} 
\P\Big( \varphi_1(\bG^I )> \sqrt{ x_1+ a_q } , \varphi_2(\bG^I )> \sqrt{ x_2+ a_q }  \Big)\\ = \lim_{ k \to \infty } \sum_{\ell=1}^{2k}(-1)^{\ell-1}  \sum_{1\le k_1< \cdots<k_\ell \le 2\bar{q} }P^I_{k_1,\ldots,k_\ell},
\end{aligned}
	\end{equation}
	and $a_q = 2\log q - \log(\log q)$, by  \cite{DN03:app},  the bivariate vector $( \varphi_1^2(\bG^I ) - a_q, \varphi_2^2(\bG^I ) - a_q )$ has a limiting distribution with joint density function 
	$$
	g_2( t_1, t_2 ) =  \frac{ g( t_1) g( t_2 ) }{G( t_1 )} = \frac{ e^{- t_1/2 } }{2\sqrt{\pi}} g( t_2 ),   \hspace{3em}{\rm for}\hspace{1em} t_1 > t_2,
	$$
	where $G(t) = \exp(- \pi^{-1/2} e^{- t/2 })$ and $g(t) = G'(t)$. Therefore, the limit in \eqref{ind.asymp} is equal to $\int_{x_1}^\infty \int_{x_2}^\infty 	g_2( t_1, t_2 ) I( t_1 > t_2 ) \, dt_2  \, d t_1$, which together with \eqref{sandwich} proves \eqref{joint.asym} by letting $q \to \infty$ first and then $k\to \infty$.
\end{proof}
\subsection{Proof of Theorem \ref{theorem:Adpower}}
\label{proof:theorem:Adpower}
\begin{proof}
	For simplicity, we only consider the two-sample problem. In detail, we aim to prove 
	\begin{equation}\label{proof:targetAdpower}
\P_{\Hb_1}\big(T^{N}_{\rm ad}= 1 \big)\rightarrow 1,\hspace{2em}\text{as $n$, $B$}\rightarrow\infty.
	\end{equation}
	under (\ref{def:AdpowerD2}) and some assumptions. 
	By the definition of $T^N_{\rm ad}$  in (\ref{def:TWNAd}), for proving (\ref{proof:targetAdpower}), it is equivalent to prove
	\begin{equation}\label{proof:targerAdpower:equ}
	\P_{\Hb_1}\big(\hat{P}^N_{\rm ad}\le\alpha\big)\rightarrow 1, \hspace{2em}\text{as $n$, $B$}\rightarrow \infty.
	\end{equation} 
	By the definition of $\hat{P}^N_{\rm ad}$ and $\hat{F}_{N,\rm ad'}(z)$ in (\ref{def:hatPNad}) and 
	(\ref{def:hatFNad'}), (\ref{proof:targerAdpower:equ}) becomes
	\begin{equation}\label{target:equ}
\P_{\Hb_1}\big(\hat{F}_{N,\rm ad'}(N_{\rm ad})<\alpha\big)\rightarrow 1, \hspace{2em}\text{as $n$, $B$}\rightarrow \infty.
	\end{equation} 
	Therefore, to obtain (\ref{proof:targetAdpower}), it is sufficient to prove (\ref{target:equ}).
	By setting $\alpha'=\alpha/\#\{\mathcal{P}\}$, we can prove that $\alpha$ is also an upper bound of $\hat{F}_{N,\rm ad'}(\alpha')$, i.e.,
	\begin{equation}\label{def:LowboundInvFhatNad'}
	\P_{\Hb_1}\Big(\hat{F}_{N,\rm ad'}\big(\alpha'\big)\le \alpha\Big)\rightarrow 1,\hspace{2em} \text{as $n$, $B$}\rightarrow\infty.
	\end{equation}
	By (\ref{def:LowboundInvFhatNad'}), to obtain (\ref{target:equ}) it is sufficient to prove  
	\begin{equation}\label{proof:targetGoal}
\P_{\Hb_1}\big(N_{\rm ad}\le \alpha'\big)\rightarrow 1, \hspace{2em}\text{as $n$, $B$}\rightarrow\infty.
	\end{equation}
	By the definition of $N_{\rm ad}$ in (\ref{def:Nad}), we have 
	\begin{equation}\label{ineq:NadDmp}
	\P_{\Hb_1}\big(\hat{P}^N_{(s_0,p)}\le \alpha'\big)\le\P_{\Hb_1}\big(N_{\rm ad}\le \alpha'\big),
	\end{equation}
	for any $p\in\mathcal{P}$. By Theorem \ref{theorem:power}, under (\ref{def:AdpowerD2}) we have 
	\begin{equation}\label{limit:powers0p}
	\P_{\Hb_1}\big(\hat{P}^N_{(s_0,p)}\le \alpha'\big)\rightarrow 1, \hspace{2em}\text{as $n$, $B$}\rightarrow\infty.
	\end{equation}
	Combining (\ref{ineq:NadDmp}) and (\ref{limit:powers0p}), we prove (\ref{proof:targetGoal}).

	To complete the proof, we now prove (\ref{def:LowboundInvFhatNad'}). By Lemma \ref{lemma:limitFnbadhatFad'}, for any $0<\alpha<1$,  we have 
	\begin{equation}\label{limit:powerad}
\P\Big(\hat{F}_{N,\rm ad'}\big(\alpha'\big)\le \alpha\Big)=\P\Big(F_{N^b,\rm ad}(\alpha')\le \alpha\Big),\hspace{2em}\text{as $n$, $B$}\rightarrow\infty.
	\end{equation}
	Moreover, by the definition of $F_{N^b,\rm ad}(z)$ in (\ref{def:Fnbad}), we have that
	$
	F_{N^b,\rm ad}(\alpha')\le \alpha
	$ holds with probability $1$, which yields (\ref{def:LowboundInvFhatNad'}).
\end{proof}

\section{Proof of lemmas in Appendix B}\label{sec:appendix C}
\subsection{Proof of Lemma \ref{lemma:NhatHApproximationN1+}}
\label{proof:lemma:NhatHApproximationN1+}
\begin{proof}
	To prove Lemma \ref{lemma:NhatHApproximationN1+}, we need to  bound  $\P\Big(\|\bN -\bH^N\|_{(s_0,p)}> \varepsilon\Big)$, where $\varepsilon= Cs_0 \log^2 (qn) n^{-1/2}$.  We first prove for $m>1$.
	 For this,  we set $\hat{\bH}^N=(\hat{H}_1^N,\ldots,\hat{H}_q^N)^\top$ with 
	\begin{equation}\label{def:HhatN}
	\hat{H}_s^N={\Big(\dfrac{1}{n_1}\sum\limits_{k=1}^{n_1}h_s(\bX_k)-\dfrac{1}{n_2}\sum\limits_{k=1}^{n_2} h_s(\bY_k)\Big)}/{\sqrt{\hat{\sigma}_{1,ss}/n_1+\hat{\sigma}_{2,ss}/n_2}}.
	\end{equation}
	By plugging $\hat{\bH}^N$, we have   $\P(\|\bN-\bH^N\|_{(s_0,p)}\!>\!\varepsilon)\le D_1 + D_2$ with 
	\[
	D_1 = \P\Big(\|\bN-\hat{\bH}^N\|_{(s_0,p)}>\varepsilon/2\Big),~D_2 = \P\Big(\|\hat{\bH}^N-\bH^N\|_{(s_0,p)}>\varepsilon/2\Big).
	\] 
	Therefore, we only need to separately prove $D_1=o(1)$ and $D_2=o(1)$ as $n\rightarrow\infty$.
	
	For proving $D_1=o(1)$, by setting 
	$
	\mathcal{E}_{12}:=\{\min_{s,\gamma}\hat{\sigma}_{\gamma,ss}>  b/2 \}, 
	$ we have
	\begin{equation}\label{ineq:bound}
	D_1\le \underbrace{\P\Big(\big\|\bN -\hat{\bH}^N\|_{(s_0,p)}> 
		\varepsilon/2 \cap\mathcal{E}_{12}\Big)}_{I_1}+\P\big(\mathcal{E}_{12}^c\big).
	\end{equation}
	Considering Assumptions {\bf (A)} and {\bf (M1)}, by Lemma \ref{lemma:sigmaRErrors}, we have $\P\big(\mathcal{E}_{12}^c\big)=o(1)$ as $n\rightarrow\infty$. Hence, we only need to prove 
	$I_1=o(1)$ as $n\rightarrow\infty$. By the Hoeffding's decomposition, considering  $\hat{v}_{\gamma,s}=m^2\hat{\sigma}_{\gamma,ss}$ and $\|\vb\|_{(s_0,p)}\le s_0^{1/p}\|\vb\|_{\infty}$, we have
	\[
	I_1\le\P\Bigg(\max_{1\le s\le q}\Big|\binom{n_1}{m}^{-1}\Delta_{n_1,s}-\binom{n_2}{m}^{-1}\Delta_{n_2,s}\Big|>\frac{mb^{1/2}\varepsilon}{\sqrt{2}s_0^{1/p}}\sqrt{\frac{1}{n_1}+\frac{1}{n_2}}\Bigg),
	\]
	where $\Delta_{n_1,s}$ and $\Delta_{n_2,s}$ are residuals of the Hoeffding's decomposition. 
	For bounding the residuals, we  threshold the kernel by $B_n=C\log(qn)$. For this, we introduce
	\begin{equation}\label{def:VE}
	\begin{aligned}
	V_{1,s}^{i_1,\ldots,i_m} &= \Psi_s(\bX_{i_1},\ldots,\bX_{i_m})\ind\{| \Psi_s(\bX_{i_1},\ldots,\bX_{i_m})|\le B_n\},\\
	E_{1,s}& =\E\Big(\Psi_s(\bX_{i_1},\ldots,\bX_{i_m})\ind\{| \Psi_s(\bX_{i_1},\ldots,\bX_{i_m})|\le B_n\}\Big),
	\end{aligned}
	\end{equation}
	and denote the thresholded kernel and Hoeffding's projection by 
	\begin{equation}\label{def:threshKernel}
	\begin{aligned}
&\hat{\Psi}_s(\bX_{i_1},\ldots,\bX_{i_m})= V_{1,s}^{i_1,\ldots,i_m} -E_{1,s},\\
&	\hat{h}_s(\bX_i)= \E\Big(\hat{\Psi}_s(\bX_{i_1},\ldots,\bX_{i_m})|\bX_i\Big).
	\end{aligned}
	\end{equation}
	Hence, the corresponding residuals become
	\[
\hat{\Delta}_{n_1,s}=\sum\limits_{1\le i_1<\ldots<i_m\le n_1}\Big(\hat{\Psi}_{s}(\bX_{i_1},\ldots,\bX_{i_m})-\sum_{\ell=1}^{m}\hat{h}_{s}(\bX_{i_\ell})\Big),
\]
By the definitions of both $\Delta_{n_1,s}$ and $\hat{\Delta}_{n_1,s}$, we then have
\[
\begin{aligned}
|\Delta_{n_1,s}-\hat{\Delta}_{n_1,s}|\le& \Bigg|\Delta_{n_1,s} -\Big(\sum\limits_{1\le i_1<\ldots<i_m\le n_1}V_{1,s}^{i_1,\ldots,i_m}-\sum_{\ell=1}^m\E(V_{1,s}^{i_1,\ldots,i_m}|\bX_{i_\ell})\Big)\Bigg|\\
&+ (m-1)\binom{n_1}{m}|E_{1,s}|.
\end{aligned}
\]
Considering that $\varepsilon=Cs_0\log^2(qn)n^{-1/2}$, we have \[\frac{mb^{1/2}\varepsilon}{\sqrt{2}s_0^{1/p}}\sqrt{\frac{1}{n_1}+\frac{1}{n_2}}=O(\log^2 (qn)/n).\]
By choosing a proper constant $C$ in $B_n$, considering Assumption {\bf (E)},  we have   $\max_s(|E_{1,s}| + |E_{2,s}|) \prec \log^2 (qn)/n$. Hence, when $n$ is sufficiently large,
we use the triangle  inequality to get $ I_1\le I_{1,1} + I_{1,2}, $
 where 
 \[
 \begin{aligned}
 I_{1,1} &= P\Bigg(\max_{1\le s\le q}\Big|\binom{n_1}{m}^{-1}\hat{\Delta}_{n_1,s}-\binom{n_2}{m}^{-1}\hat{\Delta}_{n_2,s}\Big|>C \frac{\log^2(qn)}{n}\Bigg)\\
 I_{1,2} & = Cqn^m\!\!\!\! \max_{s, i_\ell, j_\ell \atop \ell =1,\ldots, m}\!\!\big(\P(| \Psi_s(\bX_{i_1},\ldots,\bX_{i_m})| \!>\! B_n) \!+\!\P(| \Psi_s(\bY_{j_1},\ldots,\bY_{j_m})| \!>\! B_n)\big)\\
 \end{aligned}
 \]
 By choosing a proper constant $C$ in $B_n$, considering Assumption {\bf (E)}, we have $I_{1,2}=o(1)$.
 For $I_{1,1}$,	by Proposition 2.3 (c) in \cite{arcones1993limit:app}, we obtain 
	\begin{equation}\label{ineq:boundI1I1}
	I_{1,1} \le Cq\exp\big(-C_1n^{1-\frac{2}{m}}\log^{\frac{2}{m} }(qn)\big).
	\end{equation}
	Considering $m\ge 2$ and Assumption {\bf (A)}, we then have $I_{1,1}=o(1)$. Therefore, we prove that $D_1=o(1)$, as $n\rightarrow\infty$.
	
	After the proof for $D_1$, we then prove that $D_2=o(1)$. Considering
	\begin{equation}\label{ineq:HNhatHNdiff}
	\|\hat{\bH}^N -\bH^N\|_{(s_0,p)}\le s_0^{1/p}\|\hat{\bH}^N-\bH^N\|_{\infty},
	\end{equation}
	we have $ D_2 \le \P(\|\hat{\bH}^N-\bH^N\|_{\infty}>0.5s_0^{-1/p}\varepsilon).$
	By the definitions of $\bH^N$ and $\hat{\bH}^N$ in (\ref{def:HsN1+}) and (\ref{def:HhatN}), we have $	\|\hat{\bH}^N-\bH^N\|_{\infty}\le I_2 I_3$ with
	\begin{equation}\label{ineq:NHmaxdecompN1+}
	\begin{aligned}
		I_2 &= \max_{1\le s\le q}
	\frac{|\sum_{k=1}^{n_1}h_s(\bX_k)-\rho\sum_{k=1}^{n_2} h_s(\bY_k)|}{\sqrt{n_1\sigma_{1,ss}+\rho^2n_2\sigma_{2,ss}}}, \\
	I_3 &=  
	\max_{1\le s\le q}\Big|1-\frac{\sqrt{\sigma_{1,ss}+\rho\sigma_{2,ss}}}{\sqrt{\hat{\sigma}_{1,ss}+\rho\hat{\sigma}_{2,ss}}}
	\Big|,
	\end{aligned}
	\end{equation}
	where $\rho=n_1/n_2$. By Assumption {\bf (E)} and exponential inequality, we have that 
	$
	I_2 \le C\sqrt{\log (qn)}
	$ holds  with probability $1-C_1n^{-1}$.

	For bounding $I_3$, we introduce the following lemma. 
		\begin{lemma}\label{lemma:inverse}
		$\xi_1, \ldots,\xi_s\in\reals$ are positive random variables with $\xi_s>0$. For $y\in (0,1]$, we have
		\begin{equation}\label{ineq:inverseRelationship}
		\P\Big(\max_{1\le s\le q}|1-\xi_s|\le y/2\Big)\le
		\P\Big(\max_{1\le s\le q}|1-\xi_s^{-1}|\le y\Big).
		\end{equation}
	\end{lemma}
The detailed proof of Lemma \ref{lemma:inverse} is in Appendix \ref{proof:lemma:inverse}.
	Motivated by Lemma \ref{lemma:inverse}, we introduce
	\[
	I'_3:=\max_{1\le s\le q}\Big|1-\frac{\sqrt{\hat{\sigma}_{1,ss}+\rho\hat{\sigma}_{2,ss}}}{\sqrt{\sigma_{1,ss}+\rho\sigma_{2,ss}}}
	\Big|.
	\]
	By Assumption {\bf (M1)}, considering   $(a+b)(a-b)=a^2-b^2$, we use the triangle inequality to obtain
	\[
	\begin{aligned}
	I_3'&\le \max_{1\le s\le q}{(\sigma_{1,ss}+\rho\sigma_{2,ss})^{-1}}{|\hat{\sigma}_{1,ss}+\rho\hat{\sigma}_{2,ss}-\sigma_{1,ss}-\rho\sigma_{2,ss}|}\\
	&\le {(1+\rho)^{-1}b^{-1}}\Big(\max_{1\le s\le q}|\hat{\sigma}_{1,ss}-\sigma_{1,ss}|+\rho\max_{1\le s\le q}|\hat{\sigma}_{2,ss}-\sigma_{2,ss}|\Big).
	\end{aligned}
	\]
	Therefore, by Lemma \ref{lemma:sigmaRErrors},
	 $I'_3\le C \log^{3/2}(qn)n^{-1/2}$  holds with probability $1-C_1n^{-1}$. By Lemma \ref{lemma:inverse}, we then have that $I_3\le C \log^{3/2}(qn)n^{-1/2}$  holds with probability 
	 $1-C_1n^{-1}$ for sufficiently large $n$. Combining (\ref{ineq:HNhatHNdiff}) and the bound for $I_2$ and $I_3$, we then have 
	 \[
	 \|\hat{\bH}^N -\bH^N\|_{(s_0,p)} \le C s_0 \frac{\log^2 (qn)}{\sqrt{n}}
	 \]
	 with probability $1-C_1n^{-1}$ for sufficiently large $n$. Therefore, we have $D_2=o(1)$ as $n\rightarrow\infty$, which finishes the proof for $m>1$.
	 
	By Lemma \ref{lemma:sigmaRErrors} and similar proof, we can also prove for $m=1$. As the proof is much easier and similar to the proof for $m>1$, we omit the proof here.
\end{proof}

\subsection{Proof of Lemma \ref{lemma:anticoncentrate}}
\label{proof:lemma:anticoncentrate}
\begin{proof}
 For notational simplicity, we set
	\[
	L_{z,\varepsilon}:= \P(\|\bG\|_{(s_0,p)}\le z+\varepsilon)-\P(\|\bG\|_{(s_0,p)}\le z),
	\]
	where $z>0$ and $\varepsilon=O(s_0\log^2(qn)n^{-1/2})$. Let $\mathcal{E}^{R, q} = \{\xb\in\reals^q: \|\xb\|\le  R\}$ and
  $V^{z, q}_{(s_0,p)}=\{\xb\in\reals^q:\|\xb\|_{(s_0,p)}\le z\}$. We then have 
  \[
  L_{z,\varepsilon} \le  \underbrace{\P\big(\bG\in \reals^q \backslash \mathcal{E}^{R,q}\big)}_{L_1}  + \underbrace{\P\big(\bG \in V^{z+\varepsilon, q}_{(s_0,p)}\cap \mathcal{E}^{R,q}\big) - \P\big(\bG \in V^{z,q}_{(s_0,p)}\cap \mathcal{E}^{R,q}\big)}_{L_2}.
  \]
   By the tail probability of Gaussian distribution, we have  \[L_1 \le q(2\pi R^2 q^{-1})^{-1/2}\exp(-R^2q^{-1}/2).\]
   For $L_2$, by Lemma \ref{lemma:approximstCovexSet}, there is a $m$-generated convex set $A^m\in\reals^q$ such that 
  \begin{equation}\label{subset:Am}
  A^m\subset V^{z, q}_{(s_0,p)} \cap  \mathcal{E}^{R,q} \subset A^{m, R\epsilon}\hspace{2em}\text{and}\hspace{2em}
  m\le q^{s_0}\Big(\frac{\gamma}{\sqrt{\epsilon}}\ln \frac{1}{\epsilon}\Big)^{s_0^2}.
  \end{equation}
  Hence, there is a constant $C$ such that
  \begin{equation}\label{subset:Am+}
  V_{(s_0,p)}^{z+\varepsilon} \cap \mathcal{E}^{R,q} \subset A^{m, R\epsilon + C\varepsilon}
  \end{equation}
  By setting  $R=qn$ and $\epsilon = (qn)^{-2}$, we have   $R\epsilon  \prec \varepsilon$ and $L_1=o(1)$. By  Lemma \ref{lemma:gaussianDif}, we 
  combine (\ref{subset:Am}) and (\ref{subset:Am+}) to obtain $L_2 \le C \varepsilon s_0\sqrt{\log (qn)} = O(s_0^2\log^{5/2} (qn)n^{-1/2}).$ By Assumption {\bf (A)}, we have $L_2=o(1)$, which finishes the proof.
  \end{proof}

\subsection{Proof of Lemma \ref{lemma:boundhatD5}}\label{proof:lemma:boundhatD5}
 \begin{proof}
 	In Lemma \ref{lemma:boundhatD5}, we aim to bound $\hat{D}_5$, where
	\[
	\hat{D}_5:=\sup_{z>0}\Big|\P(\|\bG^N\|_{(s_0,p)}>z)
	-\P(\|\bN^b\|_{(s_0,p)}>z|\mathcal{X},\mathcal{Y})\Big|.\]
	To bound  $\hat{D}_5$, we need to analyze the distributions of $\bG^N$ and $\bN^b|\mathcal{X},\mathcal{Y}$. Considering the definitions of $\bSigma_{\gamma}$ and $\hat{\bSigma}_{\gamma}$ in (\ref{def:sigmast}) and (\ref{def:sigmahat}), by setting 
	\[
	\bSigma_{12}=\bSigma_1/n_1+\bSigma_2/n_2\hspace{1em}{\rm and}\hspace{1em}
	\hat{\bSigma}_{12}=\hat{\bSigma}_1/n_1+\hat{\bSigma}_2/n_2,
	\] we have $\bG^N\sim N(\zero,\Rb_{12})$ and $\bN^b|\mathcal{X}, \mathcal{Y}\sim N(\zero, \hat{\Rb}_{12})$, where $\Rb_{12}$ and $\hat{\Rb}_{12}$ are defined as
	\begin{equation}\label{def:R12}
	\begin{array}{l}
	\Rb_{12}={\rm Diag}(\bSigma_{12})^{-1/2}\bSigma_{12}{\rm Diag}(\bSigma_{12})^{-1/2}=(r_{12,ij})_{1\le i,j\le q},\\
	\hat{\Rb}_{12}={\rm Diag}(\hat{\bSigma}_{12})^{-1/2}\hat{\bSigma}_{12}{\rm Diag}(\hat{\bSigma}_{12})^{-1/2}=(\hat{r}_{12,ij})_{1\le i,j\le q}.
	\end{array}
	\end{equation}
	
	\begin{lemma}\label{lemma:r12error}
	\end{lemma}
	After analyzing the distributions of $\bG^N$ and $\bN^b|\mathcal{X}, \mathcal{Y}$, we then bound $\hat{D}_5$. For this, we  rewrite $\hat{D}_5$ as
	$
	\hat{D}_5= \max\big(\sup\nolimits_{z\in (0,\tilde{R}]}I_z, \sup\nolimits_{z\in (\tilde{R},\infty)}I_z\big),
	$ where 
	\[I_z=\big|\P(\|\bG^N\|_{(s_0,p)}>z)
	-\P(\|\bN^b\|_{(s_0,p)}>z|\mathcal{X},\mathcal{Y})\big|,\]
	and $\tilde{R} = C s_0 \sqrt{n}$.
	For $\sup\nolimits_{z\in (\tilde{R},\infty)}I_z$, considering $\|\vb\|_{(s_0,p)}\le s_0^{1/p}\|\vb\|_{\infty}\le s_0\|\vb\|_{\infty}$, we have
	\begin{equation}\label{ineq:IzPart2N1+}
	\sup_{z\in(R,\infty)}I_z\le \P(\|\bG^N\|_{\infty}> C\sqrt{n})+\P(\|\bN^b\|_{\infty}> C\sqrt{n}|\mathcal{X},\mathcal{Y}).
	\end{equation}
	Considering  $r_{12,ii}=\hat{r}_{12,ii}=1$, by the tail probability of Gaussian distribution, we further have 
	\begin{equation}\label{ineq:IzPart2ConcenN1+}
	\sup_{z\in(\tilde{R},\infty)}I_z
	\le Cq\exp(-C_1 n) = o(1).
	\end{equation}
	We now  bound $\sup\nolimits_{z\in (0,\tilde{R}]}I^D_z$. Let $\mathcal{E}^{\tilde{R}, q} = \{\xb\in\reals^q: \|\xb\|\le  \tilde{R}\}$ and
	$V^{z, q}_{(s_0,p)}=\{\xb\in\reals^q:\|\xb\|_{(s_0,p)}\le z\}$. Hence, considering
	 $
	\|\xb\|\le q^{1/2}\|\xb\|_{\infty}\le q^{1/2}\|\xb\|_{(s_0,p)},$
	we have $V^{z, q}_{(s_0,p)}\subset \mathcal{E}^{\tilde{R}q^{1/2},q} $ for $z<\tilde{R}$. Therefore,
	Considering Lemma \ref{lemma:approximstCovexSet}, there is a m-generated convex set $A^m$ and $\epsilon>0$ such that 
	\[
A^m\subset V^{z, d}_{(s_0,p)}\subset A^{m, \tilde{R}q^{1/2}\epsilon}\hspace{2em}\text{and}\hspace{2em}
m\le d^{s_0}\Big(\frac{\gamma}{\sqrt{\epsilon}}\ln \frac{1}{\epsilon}\Big)^{s_0^2}.	
	\]
	Let $\varepsilon' = Rq^{1/2}\epsilon$. By setting $\epsilon=(qn)^{-3/2}$, we have $\varepsilon'=s_0(qn)^{-1}$.   We then have
	$
	I_z \le  L_{z,1} + L_{z,2}
	$ with
	\begin{equation}
	\begin{aligned}
	L_{z,1} = \max\big(&\P(\bG^N\in A^{m, \varepsilon'}\backslash A^m), \P(\bN^b\in A^{m, \varepsilon'}\backslash A^m)\big)\\
	L_{z,2}=\max\Big(& \big|\P(\bG^N\in A^{m,\varepsilon'})-\P(\bN^b\in A^{m,\varepsilon'}|\mathcal{X},\mathcal{Y})\big|,\\
	&\big|\P(\bG^N\in A^m )-\P(\bN^b\in A^m|\mathcal{X},\mathcal{Y})\big|\Big),
	\end{aligned}
	\end{equation}
	for $z<\tilde{R}$. We then separately bound $L_{z,1}$ and $L_{z,2}$.
	For $L_{z,1}$, by Lemma \ref{lemma:gaussianDif} and Assumption {\bf (A)}, we have 
	\begin{equation}\label{ineq:l1}
	L_{z,1} \le C\varepsilon'\sqrt{\log(m)} = Cs_0^2 (qn)^{-1}\sqrt{\log(qn)}=o(1).
	\end{equation}	
	Considering $\mathcal{V}_{s_0}:=\{\vb\in \mathbb{S}^{q-1}: \|\vb\|_0\le s_0\}$, we have 
	\begin{align*}
	\sup_{\vb_1,\vb_2 \in\mathcal{V}_{s_0} } 	|\vb_1^\top (\hat{\Rb}_{12}-\Rb_{12})\vb_2|&\le \|\hat{\Rb}_{12}-\Rb_{12}\|_\infty \|\vb_1\|_1\|\vb_2\|_1\\&\le s_0\|\hat{\Rb}_{12}-\Rb_{12}\|_\infty.
	\end{align*}
	Therefore,  combining Theorem 4.1 and Remark 4.1 in \cite{chernozhukov2014central:app}, by Lemma \ref{lemma:r12error}, with probability at least $1-C_1n^{-1}$,
	we have 
	\begin{equation}\label{ineq:l2}
	L_{z,2}\le C\Big(s_0\frac{\log^{3/2}(qn)}{\sqrt{n}}\Big)^{1/3}\log^{2/3} (mn) \le C\Big(\frac{s_0^{10}\log^7( qn)}{n}\Big)^{1/6}.
	\end{equation}
	From Assumption {\bf (A)}, we have $L_{z,2}=o(1)$, which finishes the proof.
\end{proof}

\subsection{Proof of Lemma \ref{lemma:NS1BMEM}}
\label{proof:lemma:NS1BMEM}
\begin{proof}
	We first prove for $m>1$.
	By the definition of $N_s^1$ in (\ref{def:NS1N1}), we have
	\begin{equation}\label{dcmp:Ns1}
	N_s^1=
	\underbrace{\frac{\hat{u}_{1,s}-\hat{u}_{2,s}-u_{1,s}+u_{2,s}}
	{\sqrt{m^2\sigma_{1,ss}/n_1+m^2\sigma_{2,ss}/n_2}}}_{\tilde{N}_s^1}\cdot \frac{\sqrt{m^2\sigma_{1,ss}/n_1+m^2\sigma_{2,ss}/n_2}}{\sqrt{\hat{v}_{1,s}/n_1+\hat{v}_{2,s}/n_2}}.
	\end{equation} 
    By Lemma \ref{lemma:sigmaRErrors} and Lemma \ref{lemma:inverse}, for sufficiently large $n$  with probability at least $1-C_1n^{-1}$ we 
	have
	\begin{equation}\label{ineq:1minus}
	\Bigg|1-\frac{\sqrt{m^2\sigma_{1,ss}/n_1+m^2\sigma_{2,ss}/n_2}}{\sqrt{\hat{v}_{1,s}/n_1+\hat{v}_{2,s}/n_2}}\Bigg|\le C\frac{\log^{3/2}(qn)}{\sqrt{n}}.
	\end{equation}
	By setting $$\mathcal{E}(z)=
	\big\{\big|1-{(\hat{v}_{1,s}/n_1+\hat{v}_{2,s}/n_2)^{-1/2}}{(m^2\sigma_{1,ss}/n_1+m^2\sigma_{2,ss}/n_2)^{1/2}}\big|\le z\big \}
	$$
	and  $z\asymp \log ^{3/2}(qn)/\sqrt{n}$,  we can  bound $\P(|N_s^1|\ge x )$ by
	\begin{equation}\label{dcmpProb:Ns1}
	\P(|N_s^1|\ge x)\le\P\big(|N_s^1|\ge x, \mathcal{E}(z)\big)
	+Cn^{-1}.
	\end{equation}
	By the definition of $\mathcal{E}(z)$ and (\ref{dcmp:Ns1}), we then have 
	\[
		\P\big(|N_s^1|\ge x,\mathcal{E}(z)\big)\le \P\bigg(\frac{|\hat{u}_{1,s}-\hat{u}_{2,s}-u_{1,s}+u_{2,s}|}
	{\sqrt{m^2\sigma_{1,ss}/n_1+m^2\sigma_{2,ss}/n_2}}
	\ge (1+z)^{-1} x\bigg)
	\]
	Considering  $z=o(1)$ and $x=u\sqrt{2\log q}$ in Lemma \ref{lemma:NS1BMEM}, to prove (\ref{limit:s0tail}), we only need to prove  that as  $n,q\rightarrow\infty$, we have
	\[
	s_0 \underbrace{ \P\bigg(\frac{|\hat{u}_{1,s}-\hat{u}_{2,s}-u_{1,s}+u_{2,s}|}
	{\sqrt{m^2\sigma_{1,ss}/n_1+m^2\sigma_{2,ss}/n_2}}
	\ge C\sqrt{\log q}\bigg)}_{A_1} \rightarrow 0,
	\]
	uniformly  for $s$. By triangle and Hoeffding's inequalities, we have
	\begin{align*}
	s_0A_{1}\le & s_0\underbrace{\P\bigg(\frac{|\frac{1}{n_1}\sum_{k=1}^{n_1}h_s(\bX_k)-\frac{1}{n_2}\sum_{k=1}^{n_2}h_s(\bY_k|}
	{\sqrt{m^2\sigma_{1,ss}/n_1+m^2\sigma_{2,ss}/n_2}}
	\ge \frac{C}{2}\sqrt{\log q}\bigg)}_{A_2} \\
	&+s_0\underbrace{\P\bigg(\frac{\big|\binom{n_1}{m}^{-1}\Delta_{n_1,s}-\binom{n_2}{m}^{-1}\Delta_{n_2,s}\big|}{\sqrt{m^2\sigma_{1,ss}/n_1+m^2\sigma_{2,ss}/n_2}}\ge \frac{C}{2}\sqrt{\log q}\bigg)}_{A_3}.
	\end{align*}
	By the exponential inequality for sub-exponential distribution, considering Assumption {\bf (A)$'$} we  have
	$
	s_0A_2 \le s_0 \exp(-C\log^{1/2}(q))\rightarrow 0
	$. As $A_3$ does not exist for $m=1$, we only need to deal with $m>1$. Similarly to (\ref{def:VE}), we threshold the kernel of $\hat{u}_{\gamma,s}-u_{\gamma,s}$ by $B_n=C\log(q)$
	and construct the threshold residual  $\hat{\Delta}_{n_\gamma,s}$. Similarly to the proof of bounding $|\binom{n_1}{m}^{-1}\Delta_{n_1,s}-\binom{n_2}{m}^{-1}\Delta_{n_2,s}|$ in Lemma \ref{lemma:NhatHApproximationN1+}, by setting 
	\begin{align*}
	A_{3,1}&=	\P\Big(\Big|\binom{n_1}{m}^{-1}\hat{\Delta}_{n_1,s}-\binom{n_2}{m}^{-1}\hat{\Delta}_{n_2,s}\Big| > C\sqrt{\frac{\log q}{n}}\Big)\\
	A_{3,2}&= \max_{i_\ell, j_\ell \atop \ell =1,\ldots, m}\!\!\big(\P(| \Psi_s(\bX_{i_1},\ldots,\bX_{i_m})| \!>\! B_n) \!+\!\P(| \Psi_s(\bY_{j_1},\ldots,\bY_{j_m})| \!>\! B_n)\big),
	\end{align*}
	For proving $s_0A_3\rightarrow0$, we only need to prove  $s_0A_{3,1}\rightarrow 0$, $s_0 n^mA_{3,2}\rightarrow 0$, and $|E_{1,s}|+|E_{2,s}|\prec \sqrt{\log q / n}$, where $E_{\gamma,s}$ is defined in (\ref{def:VE}).
	By by Proposition 2.3 (c) in \cite{arcones1993limit:app}, under Assumption {\bf (A)$'$}, we have
	\[
	s_0A_{3,1}\le C_1s_0\exp\Big(- (n^{\frac{m-1}{2}}\log^{-\frac{1}{2}} (q))^{\frac{2}{m}}\Big)\rightarrow 0.
	\]
	As $\Psi_s(\bX_{i_1},\ldots,\bX_{i_m})| $ and $\Psi_s(\bX_{i_1},\ldots,\bX_{i_m})| $ have sub-exponential tails from Assumption {\bf (E)},   similarly to the proof in Lemma \ref{lemma:NhatHApproximationN1+}, under Assumption {\bf (A)$'$} we  have $s_0 n^mA_{3,2}\rightarrow 0$, and $|E_{1,s}|+|E_{2,s}|\prec \sqrt{\log q / n}$, which finishes the proof.

\end{proof}

\subsection{Proof of Lemma \ref{lemma:AdSize}}
\label{proof:lemma:AdSize}
\begin{proof}
	In Lemma \ref{lemma:AdSize}, we aim to prove (\ref{dif:CDFNad}). For this, we need to  bound 
	\[
	\sup_{z\in[\epsilon
		,1-\epsilon]}\Big|1-\tilde{F}_{N,{\rm ad}}(z)-\P(N_{\rm ad}>z|\mathcal{X},\mathcal{Y})\Big|.
	\] 
	By the definition of $N_{\rm ad}$ in (\ref{def:WadNad}), 
	we have $N_{\rm ad}=\min_{p\in\mathcal{P}}\hat{P}^N_{(s_0,p)}$, where 
	$\mathcal{P}$ is a finite set. Therefore, without loss of generality,  we assume $\mathcal{P}=\{p_1,p_2\}$ with $1\le  p_1\neq p_2\le \infty$.  We then have 
$N_{\rm ad}=\min\big(\hat{P}^N_{(s_0,p_1)},
	\hat{P}^N_{(s_0,p_2)}\big)$.  We then  have 
	$
	\P(N_{\rm ad}> z|\mathcal{X},\mathcal{Y})=\P\Big(
	\big\{\hat{P}^N_{(s_0,p_1)}> z\big\} \cap\big\{
	\hat{P}^N_{(s_0,p_2)}> z
	\big\} \Big| \mathcal{X}, \mathcal{Y}
	\Big).
	$
	In (\ref{def:FnFbnHat}) and (\ref{def:Fbn}), we introduce $F_{N^b,(s_0,p_\ell)}(z)$ and $\hat{F}_{N^b,(s_0,p_\ell)}(z)$ as 
	\begin{equation}\label{def:FlbnFlbnHat}
	\begin{aligned}
	F_{N^b,(s_0,p_\ell)}(z)&=\P(\|\bN^b\|_{(s_0,p_\ell)}\le z|\mathcal{X},\mathcal{Y}),\\
	\hat{F}_{N^b,(s_0,p_\ell)}(z)&= \frac{\sum\nolimits_{b=1}^{B}\ind\big\{N^b_{(s_0,p_\ell)}\le z|\mathcal{X},\mathcal{Y}\big\}+1}{B+1},
	\end{aligned}
	\end{equation}
	for $\ell=1,2$. By 
	the definition of  $\hat{P}^N_{(s_0,p)}$ in (\ref{def:PhatWN}), we then have 
	$
	\hat{P}^N_{(s_0,p_\ell)}=1-\hat{F}_{N^b,(s_0,p_\ell)}\big(N_{(s_0,p_\ell)}\big).
	$ 
	Therefore, we can rewrite $\P(N_{\rm ad}>z|\mathcal{X},\mathcal{Y})$ as 
	\begin{equation}\label{trans:PNad}
	\P\Big(\hat{F}_{N^b,(s_0,p_1)}(N_{(s_0,p_1)})<1-z,\hat{F}_{N^b,(s_0,p_2)}(N_{(s_0,p_2)})<1-z|\mathcal{X},\mathcal{Y}\Big).
	\end{equation}
	Similarly, by setting
	$
	F_{N,(s_0,p_\ell)}(z)=\P\big(N_{(s_0,p_\ell)}\le z\big)
	$
	we can also  rewrite 	$1-\tilde{F}_{N,{\rm ad}}(z)$ as
	\begin{equation}\label{trans:PNadTilde}
\P\Big(F_{N,(s_0,p_1)}\big(N_{(s_0,p_1)}\big)<1-z,F_{N,(s_0,p_2)}\big(N_{(s_0,p_2)}\big)<1-z\Big).
	\end{equation}
	Combining (\ref{trans:PNad}) and (\ref{trans:PNadTilde}), by setting 
	\[
	\begin{aligned}
	D_1(z)&=\P\Big(F_{N,(s_0,p_1)}\big(N_{(s_0,p_1)}\big)<1-z,F_{N,(s_0,p_2)}\big(N_{(s_0,p_2)}\big)<1-z\Big),\\
	D_2(z)&=\P\Big(\hat{F}_{N^b,(s_0,p_1)}(N_{(s_0,p_1)})<1-z,\hat{F}_{N^b,(s_0,p_2)}(N_{(s_0,p_2)})<1-z\Big),
	\end{aligned}
	\]
	we have 
	$
	\Big|1-\tilde{F}_{N,{\rm ad}}(z)-\P(N_{\rm ad}>z|\mathcal{X},\mathcal{Y})\Big|=\Big|D_{1}(z)-D_2(z)\Big|.
	$
	By Glivenko-Cantelli Theorem, we have 
	$\lim_{B\rightarrow\infty}\sup_{z\in\reals}|\hat{F}_{N^b,(s_0,p_\ell)}(z)-F_{N^b,(s_0,p_\ell)}(z)|=0$ almost surely, which motives us to 
	introduce 
	\[
	D_3(z)=\P\Big(F_{N^b,(s_0,p_1)}(N_{(s_0,p_1)})<1-z,F_{N^b,(s_0,p_2)}(N_{(s_0,p_2)})<1-z\Big).
	\]
	We then use the triangle inequality to bound $\Big|1-\tilde{F}_{N,{\rm ad}}(z)-\P(N_{\rm ad}>z|\mathcal{X},\mathcal{Y})\Big|$ by 
	\begin{equation}\label{ineq:BoundFNadtildeFNad}
	\Big|1-\tilde{F}_{N,{\rm ad}}(z)-\P(N_{\rm ad}>z|\mathcal{X},\mathcal{Y})\Big|\le \big|D_1(z)-D_3(z)\big|+
	\big|D_3(z)-D_2(z)\big|.
	\end{equation}
	By (\ref{ineq:BoundFNadtildeFNad}), to prove (\ref{dif:CDFNad}), it is sufficient  to prove that as $n,B\rightarrow\infty$, we have
	\begin{equation}\label{limit:D1D2D3}
   \sup_{z\in [\epsilon,1-\epsilon]}|D_1(z)-D_3(z)|\rightarrow 0 
	~{\rm and}~\sup_{z\in[\epsilon,1-\epsilon]}|D_3(z)-D_2(z)|\rightarrow 0,
	\end{equation}
	for any fixed $\epsilon>0$.

	By Lemma 5 in \cite{bonnery2012uniform:app}, we can prove
	\begin{equation}\label{limit:D3D2}
	\lim_{B\rightarrow\infty}\sup_{z\in[\epsilon,1-\epsilon]}|D_3(z)-D_2(z)|=0.
	\end{equation}
	Hence, we only need to prove $\lim_{n\rightarrow\infty}\sup_{z\in [\epsilon,1-\epsilon]}|D_1(z)-D_3(z)|=0.$ For this, we introduce the following lemma.
	\begin{lemma}\label{lemma:D1D3}
		Assumptions {\bf (A)$''$}, {\bf (E)}, {\bf (M1)}, and {\bf (M2)} hold. 
		Under $\Hb_0$ of (\ref{def:htwo-sampleu}) for any $\epsilon>0$
		we have 
		\[
	\sup_{z\in [\epsilon,1-\epsilon]}|D_1(z)-D_3(z)|\rightarrow 0,\hspace{2em}\text{as $n$}\rightarrow\infty.
		\]
	\end{lemma}
	The proof of Lemma \ref{lemma:D1D3} is in Appendix \ref{proof:lemma:D1D3}  of supplementary materials.  Combining (\ref{limit:D3D2}) and Lemma \ref{lemma:D1D3}, we prove (\ref{limit:D1D2D3}), which finishes the proof of Lemma \ref{lemma:AdSize}.
\end{proof}

\subsection{Proof of Lemma \ref{lemma:limitFnadFnbad}}
\label{proof:lemma:limitFnadFnbad}
\begin{proof}
	In Lemma \ref{lemma:limitFnadFnbad}, we aim to prove (\ref{eq:limitFnadFnbad}). We set 
	\[
	F_{N,\rm ad}(z)=\P(N_{\rm ad}\le z|\mathcal{X},\mathcal{Y})\hspace{1em}{\rm and }\hspace{1em}F_{N^b,\rm ad}(z)=\P(N^b_{\rm ad}\le z|\mathcal{X},\mathcal{Y}),
	\] 
	where $N_{\rm ad}$ and $N^b_{\rm ad}$ are defined in (\ref{def:Nad}) and (\ref{def:Nbad}).
	Hence, to prove (\ref{eq:limitFnadFnbad}), it is sufficient to prove 
	\begin{equation}\label{goal:lemmaFnadFnbad0}
	\sup_{z\in[\epsilon,1-\epsilon]}\Big|\P(N_{\rm ad}>z|\mathcal{X},\mathcal{Y})-\P(N^b_{\rm ad}>z|\mathcal{X},\mathcal{Y})\Big|\rightarrow0~\text{as}~n,B\rightarrow\infty.
	\end{equation}
	Without loss of generality, we assume $\mathcal{P}=\{p_1,p_2\}$ with $1\le p_1\neq p_2\le \infty$. We can then rewrite $\P(N_{\rm ad}>z|\mathcal{X},\mathcal{Y})$ as 
	\begin{equation}\label{eq:SNad}
	\P\Big(\hat{F}_{N^b,(s_0,p_1)}(N_{(s_0,p_1)})<1-z,\hat{F}_{N^b,(s_0,p_2)}(N_{(s_0,p_2)})<1-z|\mathcal{X},\mathcal{Y}\Big),
	\end{equation}
	where $\hat{F}_{N^b,(s_0,p_\ell)}(z)$  is defined in (\ref{def:FlbnFlbnHat}). Similarly, we can rewrite $\P(N^b_{\rm ad}>z|\mathcal{X},\mathcal{Y})$
	\begin{equation}\label{eq:SNbad}
\P\Big(F_{N^b,(s_0,p_1)}(N^b_{(s_0,p_1)})<1-z,F_{N^b,(s_0,p_2)}(N^b_{(s_0,p_2)})<1-z|\mathcal{X},\mathcal{Y}\Big),
	\end{equation}
	where $F_{N^b,(s_0,p_\ell)}(z)$  is defined in  (\ref{def:Nbad}). Let 
	\[
	L = \P\Big(F_{N^b,(s_0,p_1)}(N_{(s_0,p_1)})\!<\!1\!-\!z,F_{N^b,(s_0,p_2)}(N_{(s_0,p_2)})\!<\!1\!-\!z|\mathcal{X},\mathcal{Y}\Big).
	\]
	 By	Massart's inequality (see Section 1.5 in \cite{dudley2014uniform:app}) and  Lemma 5 in \cite{bonnery2012uniform}, under Assumptions {\bf (A)$''$}, {\bf (E)}, {\bf (M1)}, and {\bf (M2)}, for any fix $\epsilon > 0$, we have
	\begin{equation}\label{goal:lemmaFnadFnbad1}
\sup_{z\in[\epsilon,1-\epsilon]}\Big|\P(N_{\rm ad}>z|\mathcal{X},\mathcal{Y})-L\Big|\rightarrow 0\hspace{2em}\text{as $n$, $B$}\rightarrow\infty.
	\end{equation}
	Similarly to  the proof of Theorems \ref{therom:CoreWN1+}, considering (\ref{eq:SNbad}), we also  have
	\begin{equation}\label{goal:lemmaFnadFnbad2}
	\sup_{z\in[0,1]}\Big|\P(N^b_{\rm ad}>z|\mathcal{X},\mathcal{Y})- L\Big|\rightarrow 0,\hspace{2em}\text{as $n$, $B$}\rightarrow\infty.
	\end{equation}
	Combining (\ref{goal:lemmaFnadFnbad1}) and (\ref{goal:lemmaFnadFnbad2}), we  use the triangle inequality to obtain (\ref{goal:lemmaFnadFnbad0}), which finishes the proof of Lemma \ref{lemma:limitFnadFnbad}.

\end{proof}

\subsection{Proof of Lemma \ref{lemma:limitFnbadhatFad'}}
\label{proof:lemma:limitFnbadhatFad'}
\begin{proof}
	In Lemma \ref{lemma:limitFnbadhatFad'}, we aim to prove (\ref{eq:limitFnbadhatFad'}).
	By the definitions of $F_{N^b,\rm ad}(z)$ and $\hat{F}_{N,\rm ad'}(z)$ in (\ref{def:Fnbad}) and 
	(\ref{def:hatFNad'}), we have 
	\begin{equation}\label{eq:surviveFF}
	\begin{aligned}
	1-F_{N^b,\rm ad}(z)&=\P(N^b_{\rm ad}>z|\mathcal{X},\mathcal{Y})\\
	1-\hat{F}_{N,\rm ad'}(z)&=\sum_{b=1}^B\ind\{N^b_{ad'}>z|\mathcal{X},\mathcal{Y}\}/(B+1),
	\end{aligned}
	\end{equation}
	where $N_{\rm ad}^b$ and $N^b_{\rm ad'}$ are defined in (\ref{def:Nbad}) and (\ref{def:hatbns0p}). Therefore, for (\ref{eq:limitFnbadhatFad'}) it is sufficient to prove 
	\begin{equation}\label{trans:lemma:FhadNad'}
\sup_{z\in[\epsilon,1-\epsilon]}\Bigg|\P\Big(N^b_{\rm ad}>z|\mathcal{X},\mathcal{Y}\Big)-\Big(1-\hat{F}_{N,\rm ad'}(z)\Big)\Bigg|\rightarrow 0,
	\end{equation}
	as $n,B\rightarrow\infty$.
	Without loss of generality, we assume $\mathcal{P}=\{p_1,p_2\}$ with $1\le p_1\neq p_2\le \infty$, which yields 
	\begin{equation}\label{def:Nbad'inproof}
	N^b_{\rm ad'}=\min\Big(\hat{P}^{b,N}_{(s_0,p_1)},\hat{P}^{b,N}_{(s_0,p_2)}\Big),
	\end{equation} 
	where $\hat{P}^{b,N}_{(s_0,p)}$ is defined in (\ref{def:hatbns0p}).
	Combining (\ref{eq:surviveFF}) and (\ref{def:Nbad'inproof}), we then have 
	\begin{equation}\label{eq:1-hatFNad'part1}
	1-\hat{F}_{N,\rm ad'}(z)=\frac{\sum_{b=1}^{B}\ind\{\hat{P}^{b,N}_{(s_0,p_1)}>z,\hat{P}^{b,N}_{(s_0,p_2)}>z|\mathcal{X},\mathcal{Y}\}}{B+1}.
	\end{equation}
	By setting
	$\hat{F}^{b,N}_{(s_0,p)}(z)=B^{-1}\bigg(\sum_{b_1\neq b}\ind\{N^{b_1}_{(s_0,p)}\le z|\mathcal{X},\mathcal{Y}\}+1\bigg),$
	considering the definition of $\hat{P}^{b,N}_{(s_0,p)}$ in (\ref{def:hatbns0p}), 
	we have $\hat{P}^{b,N}_{(s_0,p)}=1-\hat{F}^{b,N}_{(s_0,p)}(N^b_{(s_0,p)})$. Therefore, by (\ref{eq:1-hatFNad'part1}), we  rewrite $ 1-\hat{F}_{N,{\rm ad'}}(z)$ as
	\begin{equation}\label{def:1-hatFad'}
	\frac{\sum_{b=1}^B\ind\Big\{\hat{F}^{b,N}_{(s_0,p_1)}(N^b_{(s_0,p_1)})<1-z,
		\hat{F}^{b,N}_{(s_0,p_2)}(N^b_{(s_0,p_2)})<1-z|\mathcal{X},\mathcal{Y}\Big\}}{B+1},
	\end{equation}
	As $\hat{F}^{b,N}_{(s_0,p)}(z)\rightarrow F_{N^b,(s_0,p)}(z)$, to approximate $1-\hat{F}_{N,\rm ad'}(z)$ we introduce $S(z)$ as
	\begin{equation}\label{def:Sz}
	\frac{\sum_{b=1}^B\ind\Big\{F_{N^b,(s_0,p_1)}(N^b_{(s_0,p_1)})\!\!<\!\!1\!\!-\!\!z,F_{N^b,(s_0,p_2)}(N^b_{(s_0,p_2)})\!\!<\!\!1\!\!-\!\!z|\mathcal{X},\mathcal{Y}\Big\}}{B+1},
	\end{equation}
	where $F_{N^b,(s_0,p_\ell)}$ is defined in (\ref{def:Nbad}).
	To analyze the difference between $1-\hat{F}_{N,\rm ad'}(z)$ and $S(z)$, we introduce the following lemma.
	\begin{lemma}\label{lemma:FhatNad'S}
		Let $\epsilon$ be any positive real number, we have
		\[
		\sup_{z\in [\epsilon,1-\epsilon]}\big|1-\hat{F}_{N,\rm ad'}(z)-S(z)\big|=0, \hspace{1em}\text{as $n$, $B$}\rightarrow\infty.
		\]
	\end{lemma}
	The proof of Lemma \ref{lemma:FhatNad'S} is in Appendix \ref{proof:lemma:FhatNad'S}.
	Considering  
	$
	N_{\rm ad}^b=\min_{p\in\mathcal{P}}\Big(1-F_{N^b,(s_0,p)}(N_{(s_0,p)}^b)\Big)$ from (\ref{def:Nbad}), we can rewrite $S(z)$ as 
	\[
	S(z)=(B+1)^{-1}{\sum_{b=1}^{B}\ind \{N^b_{\rm ad}>z|\mathcal{X},\mathcal{Y}\}}.
	\]
	By Massart's inequality (see Section 1.5 in \cite{dudley2014uniform:app}), we have 
	\begin{equation}\label{limit:SNbad}
	\sup_{z\in [0,1]}|S(z)-\P(N^b_{\rm ad}>z|\mathcal{X},\mathcal{Y})|\rightarrow 0,
	\end{equation}
	as $n, B\rightarrow\infty.$
	Combining Lemma \ref{lemma:FhatNad'S} and (\ref{limit:SNbad}), we plug in $S(z)$ and use the triangle inequality to obtain (\ref{trans:lemma:FhadNad'}), which finishes the proof of Lemma \ref{lemma:limitFnbadhatFad'}.
\end{proof}
\section{Proofs of lemmas in Appendix C }\label{sec:appendix D}

\subsection{Proof of Lemma \ref{lemma:inverse}}
\label{proof:lemma:inverse}

\begin{proof}
	Under the event $\{\max_{1\le s\le q}|1-\xi_s|\le y/2\}$, we have $|1-\xi_s|\le y/2$ for any $s\in \{1,\ldots,q\}$. Considering $y\in(0,1]$, by the simple calculation,  $|1-\xi_s|\le y/2$ implies 
	\[
	|1-\xi_s^{-1}|\le\max\Big(\frac{y}{2+y},\frac{y}{2-y}\Big)\le y,
	\]
	for any $s\in\{1,\ldots,q\}$. Therefore, we have 
	\[
	\Big\{\max_{1\le s\le q}|1-\xi_s|\le y/2\Big\}\subseteq \Big\{\max_{1\le s\le q}|1-\xi_s^{-1}|\le y\Big\},
	\]
	which implies (\ref{ineq:inverseRelationship}). Hence, we finish the proof of Lemma \ref{lemma:inverse}.
\end{proof}

\subsection{Proof of Lemma \ref{lemma:D1D3}}
\label{proof:lemma:D1D3}
\begin{proof}
	Without loss of  generality, we assume $\mathcal{P}=\{p_1,p_2\}$ with $1\le p_1\neq p_2\le \infty$. We set 
	\begin{equation}\label{def:D1D3inProof}
	\begin{aligned}
	D_1(z)&\!=\!\P\Big(F_{N,(s_0,p_1)}\big(N_{(s_0,p_1)}\big)\!<\!1\!-\!z,F_{N,(s_0,p_2)}\big(N_{(s_0,p_2)}\big)\!<\!1\!-\!z\Big),\\
	D_3(z)&\!=\!\P\Big(F_{N^b,(s_0,p_2)}(N_{(s_0,p_1)})\!<\!1\!-\!z,F_{N^b,(s_0,p_2)}(N_{(s_0,p_2)})\!<\!1\!-\!z\Big),
	\end{aligned}
	\end{equation}
	where $F_{N,(s_0,p_\ell)}(z)$ and  $F_{N^b,(s_0,p_\ell)}(z)$ are defined in (\ref{def:FnFbnHat}) and  (\ref{def:Fbn}). In Lemma \ref{lemma:D1D3}, we aim to prove 
	\begin{equation}\label{limit:D1D3goal}
	\lim_{n\rightarrow\infty}\sup_{z\in [\epsilon,1-\epsilon]}|D_1(z)-D_3(z)|=0.
	\end{equation}
	By following the proof of Theorem \ref{therom:CoreWN1+}, under Assumptions {\bf (A)$''$}, {\bf (E)},  {\bf (M1)}, and {\bf (M2)}, by setting
	\begin{align*}
	F_{N,12}(z_1,z_2)&=\P\big(N_{(s_0,p_1)}\le z_1,N_{(s_0,p_2)}\le z_2\big),\\
	F_{G,12}(z_1,z_2)&=\P\big(\|\bG^N\|_{(s_0,p_1)}\le z_1,\|\bG^N\|_{(s_0,p_2)}\le z_2\big)\Big|,
	\end{align*}
	 with $\bG^N\sim N(\zero,\Rb_{12})$ with $\Rb_{12}$ defined in (\ref{def:SigmaD12N1+}), we have 
	\begin{equation}\label{limit:Ns0pGN}
\sup_{z_1,z_2\in(0,\infty)}\Big|F_{N,12}(z_1,z_2)-F_{G,12}(z_1,z_2)\Big|\rightarrow\!0,~\text{as $n$}\rightarrow\infty.
	\end{equation}
 (\ref{limit:Ns0pGN}) motives us to introduce 
	\begin{align*}
	D_4(z)&\!=\!\P\Big(F_{N,(s_0,p_1)}\big(\|\bG^N\|_{(s_0,p_1)}\big)\!<\!1\!-\!z,F_{N,(s_0,p_2)}\big(\|\bG^N\|_{(s_0,p_2)}\big)\!<\!1\!-\!z\Big),\\
	D_5(z)&\!=\!\P\Big(F_{N^b,(s_0,p_1)}\big(\|\bG^N\|_{(s_0,p_1)}\big)\!<\!1\!-\!z,F_{N^b,(s_0,p_2)}\big(\|\bG^N\|_{(s_0,p_2)}\big)\!<\!1\!-\!z\Big).
	\end{align*}
	Combining (\ref{def:D1D3inProof}) and (\ref{limit:Ns0pGN}), we then have 
	\[
\sup_{z\in(0,1)}|D_1(z)-D_4(z)|\rightarrow 0\hspace{2em}{\rm and}\hspace{2em}\sup_{z\in(0,1)}|D_3(z)-D_5(z)|\rightarrow 0,
	\]
	as $n\rightarrow\infty$.
	Therefore, by using the triangle inequality, to prove (\ref{limit:D1D3goal}) we only need to prove 
	\begin{equation}\label{limit:D4D5goal}
\sup_{z\in[\epsilon,1-\epsilon]}|D_4(z)-D_5(z)|\rightarrow 0,\hspace{2em}\text{as $n$}\rightarrow\infty.
	\end{equation}
	By Assumption {\bf (A)$''$}, considering Theorems \ref{therom:CoreWN1+},  for any $\epsilon>0$ and sufficiently large $n$, we have
	\[
	\sup_{z\in[\epsilon, 1-\epsilon]}\Big|F_{N,(s_0,p)}^-(z)-F_{N^b,(s_0,p)}^-(z)\Big|\le h_{q,N}(\epsilon)\sup_{t\in \reals} |F_{N,(s_0,p)}(t)-F_{N^b,(s_0,p)}(t)|.
	\]
	Moreover, by the proof of Lemma \ref{lemma:anticoncentrate} we have
	\[
	\sup_{z\in[\epsilon,1-\epsilon]}|D_4(z)-D_5(z)|
	\le C h_{q,N}(\epsilon) s_0\sqrt{\log(nq)}\sup_{t\in \reals} \Big|F_{N,(s_0,p)}(t)-F_{N^b,(s_0,p)}(t)\Big|,
	\]
	for sufficiently large $n$. By the proof of Lemma \ref{lemma:CCK2}, \ref{lemma:anticoncentrate},  \ref{lemma:boundhatD5} and Theorem \ref{therom:CoreWN1+}, under Assumption {\bf (A)$''$}, {\bf (E)}, {\bf (M1)}, and 
	{\bf (M2)}, we have 
	\[
	\sup_{z\in[\epsilon,1-\epsilon]}|D_4(z)-D_5(z)|
	\le C h_{q,N}(\epsilon)s_0\sqrt{\log(qn)}\Big(\frac{s_0^{14}\log^7 (qn)}{n}\Big)^{1/6}.
	\]
    In Assumption {\bf (A)$''$}, we set  $h^{0.6}_{q,N}(\epsilon)s_0^2\log(qn)=o(n^{1/10})$. Therefore, we have 
	\[
	\sup_{z\in[\epsilon,1-\epsilon]}|D_4(z)-D_5(z)|\rightarrow0,\hspace{2em}\text{as $n$}\rightarrow\infty,
	\]
	which finishes the proof.
	
\end{proof}

\subsection{Proof of Lemma \ref{lemma:FhatNad'S}}
\label{proof:lemma:FhatNad'S}
\begin{proof}
	In  Lemma \ref{lemma:FhatNad'S}, we aim to prove 
	$
	\sup_{z\in [\epsilon,1-\epsilon]}\big|1-\hat{F}_{N,\rm ad'}(z)-S(z)\big|\rightarrow 0,
	$
	as $n, B\rightarrow\infty$.
	For this, we need to prove that
	for any $\delta, \tilde{\varepsilon}>0$
	\begin{equation}\label{ineq:limitgoalpart0}
	\P\Bigg(\sup_{z\in [\epsilon,1-\epsilon]}\big|1-\hat{F}_{N,\rm ad'}(z)-S(z)\big|>\delta\Bigg)<\tilde{\varepsilon},
	\end{equation}
	holds for sufficient large $n$ and $B$. By setting $$\hat{F}^{b,N}_{(s_0,p)}(z)=B^{-1}\big(\sum_{b_1\neq b}\ind\{N^{b_1}_{(s_0,p)}\le z|\mathcal{X},\mathcal{Y}\}+1\big),$$
	considering Massart's inequality (Section 1.5 in \cite{dudley2014uniform:app}), we have 
	\[
	\sup_{1\le b\le B\atop z\in\reals}\Big|\hat{F}^{b,N}_{(s_0,p)}(z)-F_{N^b,(s_0,p)}(z)\Big|\rightarrow0, \hspace{2em}\text{as $n$, $B$}\rightarrow\infty.
	\]	
	Considering Lemma 5 in \cite{bonnery2012uniform:app}, for any fixed $\epsilon, \delta'>0$, by setting \[\mathcal{A}(\delta')=\Bigg\{\sup_{1\le b\le B\atop z\in[\epsilon,1-\epsilon]}\Big|\hat{F}^{b,N-}_{(s_0,p)}(z)-F^{-}_{N^b,(s_0,p)}(z)\Big|\le \delta'\Bigg\},\] as $n$ and $B$ are sufficiently large, we have $\P(\mathcal{A}(\delta')^c)\le \tilde{\varepsilon}/2.
	$ Therefore, considering that 
	$F^{-}_{N^b,(s_0,p)}(z)$ is Lipschitz continuous on $z\in [\epsilon,1-\epsilon]$,  by the definitions of $1-\hat{F}_{N,\rm ad'}$ and $S(z)$ in (\ref{def:1-hatFad'}) and (\ref{def:Sz}), under $\mathcal{A}(\delta')$ there is a constant $C$ such that 
	$
	S(z+C\delta')\le 1-\hat{F}_{N,\rm ad'}(z)\le S(z-C\delta')
	$
	holds for any $z\in [\epsilon,1-\epsilon]$ and sufficiently large $	n$ and $B$. Hence, under $\mathcal{A}(\delta')$ we have   $	\sup_{z\in[\epsilon,1-\epsilon]}\Big|1\!-\!\hat{F}_{N,\rm ad'}(z)\!-\!S(z)\big|\le \mathcal{L}$ with 
	\begin{equation}\label{ineq:1mhatFSpart1}
 \mathcal{L}=\max\Bigg(\sup_{z\in[\epsilon,1\!-\!\epsilon]}\Big|S(z\!+\!C\delta')\!-\!S(z)\Big|,\sup_{z\in[\epsilon,1\!-\!\epsilon]}\Big|S(z-C\delta')\!-\!S(z)\Big|\Bigg).
	\end{equation}
	Therefore,  to prove (\ref{ineq:limitgoalpart0}), we only need to prove
	\begin{equation}\label{ineq:limitLgoal}
	\P\big({\mathcal{L}}>\delta,\mathcal{A}(\delta')\big)\le \tilde{\mathcal{\varepsilon}}/2,
	\end{equation}
	for sufficiently large $n$ and $B$. By Massart's inequality (Section 1.5 in \cite{dudley2014uniform:app}) and the definition of $S(z)$ in (\ref{def:Sz}), we have 
	\begin{equation}\label{limit:SzFnbad}
	\sup_{z\in[0,1]}|S(z)-F_{N^b,\rm ad}(z)|\rightarrow 0, \hspace{2em}\text{as $n$, $B$}\rightarrow\infty, 
	\end{equation}
	where $F_{N^b,\rm ad}(z)$ is defined in (\ref{def:Fnbad}). By (\ref{limit:SzFnbad}),  the limit of $\mathcal{L}$ is 
	\[
	\max\Big(\sup_{z\in[\epsilon,1\!-\!\epsilon]}\Big|F_{N^b,\rm ad}(z\!+\!C\delta')\!-\!F_{N^b,\rm ad}(z)\Big|,\sup_{z\in[\epsilon,1\!-\!\epsilon]}\Big|F_{N^b,\rm ad}(z-C\delta')\!-\!F_{N^b,\rm ad}(z)\Big|\Big).
	\]
	As $F_{N^b,\rm ad}(z)$ is uniformly Lipschitz contentious on $[\epsilon,1-\epsilon]$,  there is a constant $C_1$ such that 
	\begin{equation}\label{limit:BL}
	0\le \mathcal{L}\le C_1\delta',
	\end{equation}  
	holds for sufficiently large $n$ and $B$.
	By setting $\delta'$ small enough and (\ref{limit:BL}), we  obtain (\ref{ineq:limitLgoal}), which finishes the proof of Lemma \ref{lemma:FhatNad'S}.
\end{proof}

\section{Proof of useful lemmas in Appendix  A}\label{sec:proof:usefullemma}
\subsection{Proof of Lemma \ref{lemma:CCK2}}\label{proof:lemma:CCK2}
By setting $\mathcal{E}^{R,d} = \{\xb\in\reals^d: \|\xb\| \le R\}$, from Assumption
{\bf (E)}$'$, we have 
\[
\P(S_n^{\bZ}\in (\mathcal{ E}^{R,d})^c)\vee\P(S_n^{\bW}\in (\mathcal{ E}^{R,d})^c)=C_1d\exp(-C_2Rd^{-1/2}).
\]
By setting $V^{z,d}_{(s_0,p)}=\{\xb\in\reals^d:\|\xb\|_{(s_0,p)}\le z\}$, we then have
\begin{equation}\label{ineq:boundR}
\sup_z\Big|\P\big(S_n^{\bZ}\in V^{z,d}_{(s_0,p)}\big)-P\big(S_n^{\bW}\in  V^{z,d}_{(s_0,p)}\big)\big|\le A_1 + A_2,
\end{equation}
where $A_1=C_1d\exp(-C_2Rd^{-1/2})$ and $A_2=\sup_z P_z$ with 
\[
P_z =|\P\big(S_n^{\bZ}\in \mathcal{E}^{R,d}\cap V^{z,d}_{(s_0,p)})-P\big(S_n^{\bW}\in \mathcal{E}^{R,d}\cap V^{z,d}_{(s_0,p)}\big)|.
\]
We then approximate $\mathcal{E}^{R,d}\cap V^{z,d}_{(s_0,p)}$ with $m$-generated convex set. According to Lemmas \ref{lemma:approximstCovexSet} and \ref{lemma:gaussianDif}, by setting
\[
\bar{\rho}=|\P(S_n^{\bZ}\in A^m) -\P(S_n^{\bW}\in A^m)|\vee|\P(S_n^{\bZ}\in A^{m,R\epsilon}) -\P(S_n^{\bW}\in A^{m, R\epsilon})|,
\]
we have $P_z\le CR\epsilon\log^{1/2}(m)+\bar{\rho}$, where $C$ only depends on $b$. By high dimensional CLT for Hyperreactangles in \cite{chernozhukov2014central:app}, we have 
\[
\bar{\rho}\le C\Big(\frac{\log^7(mn)}{n}\Big)^{1/6},
\]
where $C$ only depends on $b$. Considering (\ref{ineq:boundR}), we then
have
\begin{align*}
\sup_z\Big|\P\big(&S_n^{\bZ}\in V^z_{(s_0,p)}\big)-P\big(S_n^{\bW}\in  V^z_{(s_0,p)}\big)\big|\\\le & CR\epsilon\log^{1/2}(m)+C\Big(\frac{\log^7(mn)}{n}\Big)^{1/6}+C_1d\exp(-C_2Rd^{-1/2}).
\end{align*}
By setting $\epsilon = (dn)^{-3/2}$ and $R=(dn)^{1/2}$, considering $s_0^2\log (dn) = O(n^\zeta)$ with $0<\zeta<1/7$, we  have 
\[
R\epsilon\log^{1/2}(m) \preceq   \Big(\frac{\log^7(mn)}{n}\Big)^{1/6},~ d\exp(-C_2Rd^{-1/2})
\preceq \Big(\frac{\log^7(mn)}{n}\Big)^{1/6},
\]
which yields (\ref{ineq:HCLT_s0p}).

\subsection{Proof of Lemma \ref{lemma:approximstCovexSet}}
\begin{proof}
	By the definition of $m$-generated convex sets $A^m$ and $A^{m,\epsilon}$, Lemma \ref{lemma:approximstCovexSet} is an immediate corollary of Lemma \ref{lemma:polyAppr}.
\end{proof}

\subsection{Proof of Lemma \ref{lemma:foldedNormal}}
\begin{proof}
	By the Jensen's inequality, we have 
	\begin{equation}\label{ineq:fn1}
	\exp\Big(t\E\Big[\max_{1\le i\le d}|W_i|\Big]\Big)
	\le \E\Big[\exp\Big(t\max_{1\le i\le d}|W_i|\Big)\Big]
	\le d\E[\exp(t|W_i|)].
	\end{equation} 
	By  (23) of \cite{tsagris2014folded:app}, we have 
	\begin{equation}\label{ineq:fn2}
	\E[\exp(t|W_i|)]=2e^{\frac{\sigma^2 t^2}{2}}[1-\Phi(-\sigma t)]\le 2e^{\frac{\sigma^2 t^2}{2}}.
	\end{equation}
	Combining (\ref{ineq:fn1}) and (\ref{ineq:fn2}), we have
	\[
	\exp\Big(t\E\Big[\max_{1\le i\le d}|W_i|\Big]\Big)
	\le  2de^{\frac{\sigma^2 t^2}{2}},
	\]
	which yields (\ref{ineq:fn0}).

\end{proof}

\subsection{Proof of Lemma \ref{lemma:sigmaRErrors}}
\begin{proof}
	We first prove for $m>1$.
	For simplicity, we only present the proof for $\bX$.
	In (\ref{def:sigmast}), we set $\bSigma_{1}=(\sigma_{1,st})$   with 
	\begin{equation}\label{def:sigmaInProof}
	\sigma_{1,st}= \E[h_s(\bX)h_t(\bX)],
	\end{equation} where $h_s$ is defined in (\ref{def:PsiAnd}). To estimate $\bSigma_1$, in (\ref{def:sigmahat}) we introduce  $\hat{\bSigma}_1:=(\hat{\sigma}_{1,st})\in\reals^{q\times q}$, where $
\hat{\sigma}_{1,st}={n_1}^{-1}\sum_{k=1}^{n_1}(Q_{1k,s}-\hat{u}_{1,s})(Q_{1 k,t}-\hat{u}_{1,t}).
	$
	By setting $\tilde{u}_{1,s}=\hat{u}_{1,s}-u_{1,s}$  and $
	\tilde{Q}_{1 k,s}=Q_{1 k,s}-u_{1,s}$, we  rewrite $\hat{\sigma}_{1,st}$ as 
	\begin{equation}\label{def:sigmahatInProof1}
	\hat{\sigma}_{1,st}=n_1^{-1}
	\sum\nolimits_{k=1}^{n_1}
	\tilde{Q}_{1 k,s}\tilde{Q}_{1 k,t}
	-\tilde{u}_{1,s}\tilde{u}_{1,t}.
	\end{equation}
	To provide an upper bound for $\max_{1\le s,t\le q} \big|\hat{\sigma}_{1,st}-\sigma_{1,st}\big|$, by
	 combining (\ref{def:sigmaInProof}) and (\ref{def:sigmahatInProof1}),
	we use the triangle inequality to obtain
	\begin{align*}
	\max_{1\le s,t\le q} \big|\hat{\sigma}_{1,st}-\sigma_{1,st}\big|
	\le& \underbrace{\max_{1\le s,t\le q}\Big|n_{1}^{-1}
		\big(\sum\nolimits_{k=1}^{n_1}
		\tilde{Q}_{1k,s}
		\tilde{Q}_{1k,t}\big)-\E[h_s(\bX)h_t(\bX)]\Big|}_{L_1}\\
	&+\underbrace{\max_{1\le s,t\le q}\Big|\tilde{u}_{1,s}\tilde{u}_{1,t}\Big|}_{L_2}.
	\end{align*}
	We then bound $L_1$ and $L_2$ separately. 	For bounding $L_2$, we introduce 
	\[\tilde{u}'_{1,s}=\binom{n_1}{m}^{-1}\sum_{1\le i_1<\ldots,<i_m\le n_1}V_{1,s}^{i_1,\ldots,i_m} - E_{1,s},
	\]
	where $V_{1,s}^{i_1,\ldots,i_m}$ and $E_{1,s}$ are defined in (\ref{def:VE}) with threshold $B_n = C\log(qn)$. For any $\delta > 0$,  by choosing proper $C$, we have $E_{1,s} \prec  (qn)^{-\delta}$. We
	then have
	\[
	|\tilde{u}_{1,s}-\tilde{u}'_{1,s}| \le 
	\underbrace{\Big|\tilde{u}_{1,s} - \binom{n_1}{m}^{-1}\sum_{1\le i_1<\ldots,<i_m\le n_1}V_{1,s}^{i_1,\ldots,i_m} \Big|}_{L_{2,s}} + E_{1,s}.
	\]
    By setting $z \succ (qn)^{-\delta}$, we have 
	\begin{equation}\label{bound:L2_1}
	\max_{1\le s \le q}\P(|\tilde{u}_{1,s}| > z) \le \max_{1\le s \le q}\Big(\P(|\tilde{u}'_{1,s}|>z/3) +\P(L_{2,s} > z/3) \Big).
	\end{equation}
	By  using the exponential inequality for bounded $U$-statistics we have 
	\begin{equation}\label{bound:L2_2}
	\max_{1\le s\le q}\P(|\tilde{u}'_{1,s}|>z/3) \le  C \exp( - C_1  n z^2/B_n^2).
	\end{equation}
	By Assumption {\bf (E)},  we also have 
	\begin{equation}\label{bound:L2_3}
	\begin{aligned}
	&\max_{1\le s\le q}\P(L_{2,s} > z/3) \le C n_1^m \exp(-C_1 B_n)
	\end{aligned}
	\end{equation}{}
	Combining (\ref{bound:L2_1}), (\ref{bound:L2_2}), and (\ref{bound:L2_3}), we then have
	\begin{equation}\label{BoundL2}
	\begin{aligned}{}
	\P(L_2 >y)&\le q^2\max_{1\le s,t \le q}\P\big(|\tilde{u}_{1,s}\tilde{u}_{1,t}|>y\big)
	\le 2q^2\max_{1\le s\le q}\P\big(|\tilde{u}_{1,s}|>\sqrt{y}\big)\\
	&\le Cq^2\exp(-C_1ny/B_n^2) + Cq^2 n_1^m \exp(-C_1B_n).
	\end{aligned}
	\end{equation}
	Therefore, for sufficiently large $n_1$ with probability $1-Cn_1^{-1}$  we have 
	$
	L_2 \le \log^3 (qn) / n.
	$
	
	We now bound $L_1$. Considering that $n_1^{-1}\sum_{k=1}^{n_1}h_s(\bX_{k})h_t(\bX_{k})$ approximates $\E[h_s(\bX_{k})h_t(\bX_{k})]$, we use  triangle inequality again to bound $L_1$ by
	\begin{equation}\label{DecompL_1}
	\begin{aligned}
	L_1\le&\underbrace{\max_{1\le s,t\le q}\Big|n_{1}^{-1}
		\big(\sum\nolimits_{k=1}^{n_1}
		\tilde{Q}_{1k,s}
		\tilde{Q}_{1k,t}\big)-n_1^{-1}\sum_{k=1}^{n_1}h_s(\bX_{k})h_t(\bX_{k})
		\Big|}_{L_3}\\
	&+\underbrace{\max_{1\le s,t\le q}\Big|n_1^{-1}\sum_{k=1}^{n_1}h_s(\bX_{k})h_t(\bX_{k})
		-\E[h_s(\bX)h_t(\bX)]\Big|}_{L_4}.
	\end{aligned}
	\end{equation}
	By Assumption {(\bf E)}, $h_s(\bX)$  has sub-exponential tails.	 Therefore, by Theorem 6 in \cite{delaigle2011robustness:app}, we have
	\begin{equation}\label{decomp:sigmahat1st}
	\P(L_4 > z)
	\le Cq^2\exp(-C_1n_1z^2)+Cq^2\exp\big(-C_2(n_1z)^{1/2}\big).
	\end{equation}
	Therefore, for sufficiently large $n_1$, with probability $1-Cn_1^{-1}$, we have 
	\[
	L_4 \le C\sqrt{\frac{\log(qn)}{n}} + C_1\frac{\log^2(qn)}{n}.
	\]
	After bounding $L_4$, we now deal with $L_3$. For this,   we   decompose $\tilde{Q}_{1k,s}$ as
	\begin{align}\label{decomposition:tildeq}
	\tilde{Q}_{1k,s}=\binom{n_1-1}{m-1}^{-1}\Big(Ah_s(\bX_{k})+B S_{1,s}+\Upsilon_{1,s}^{(k)}\Big),
	\end{align} 
	with $A=\binom{n_1-1}{m_1-1}-\binom{n_1-2}{m-2}$, $B=\binom{n_1-1}{m-2}$ ,
	$S_{1,s}:=\sum_{\beta=1}^{n_1}h_{s}(\bX_{\beta})$ and 
	\begin{align}\label{def:Upsilon}
	\Upsilon_{1,s}^{(k)}&=\sum_{1\le \ell_1<\ldots<\ell_{m-1}\le n_1\atop \ell_j\neq k, j= 1,\ldots,m-1}^{n_1} \Gamma_{1,s}^{k,\ell_1\ell_2\ldots\ell_{m-1}},
	\end{align}
	with $\Gamma_{1,s}^{k,\ell_1\ell_2\ldots\ell_{m-1}}= \Psi_{s}(\bX_{k},\bX_{\ell_1}\ldots,\bX_{\ell_m-1})-\big(h_{s}(\bX_{k})+\sum\limits_{i=1}^{m-1} h_{s}(\bX_{\ell_i})\big)\Big)$. $\Psi_s(\bX_{k_1},\ldots,\bX_{k_m})$, the centralized version of $\Phi_s(\bX_{k_1},\ldots,\bX_{k_m})$, is defined in (\ref{def:PsiAnd}). For notational simplicity,  by setting
	\begin{align}\label{def:VGammaD}
	V_{1,st}^2:=\sum\limits_{k=1}^{n_1}
	h_s(\bX_{k})h_t(\bX_{k}),
	~\Lambda_{1,s}:=\sum\limits_{k=1}^{n_1}
	\Upsilon_{1,s}^{(k)},~
	\Lambda^2_{1,st}:=\sum\limits_{k=1}^{n_1}
	\Upsilon_{1,s}^{(k)}\Upsilon_{1,t}^{(k)},
	\end{align}
	and $D=\binom{n_1-1}{m-1}$.
	we have $L_3=\max_{1\le s,t \le q}L_{3,st}$, where $L_{3,st}$ is defined as
	\begin{align*}
	&\Big|\frac{1}{n_1}\Big(\frac{1}{D^2}\sum\limits_{k=1}^{n_1}  \big( Ah_s(\bX_{
		k})+BS_{1,s}+\Upsilon_{1,s}^{(k)} \big)\big( Ah_t(\bX_{
		k})+BS_{1,t}+\Upsilon_{1,t}^{(k)} \big)\\&-\sum\limits_{k=1}^{n_1}h_s(\bX_{k})h_t(\bX_{k})
	\Big)\Big|,
	\end{align*}
	After introducing these notations, we  can expand $L_{3,st}$  as  
	\begin{align*}
	L_{3,st}=&\Big|\frac{A^2-D^2}{n_1D^2}V_{1,st}^2+\frac{1}{n_1D^2}(2AB+n_1B^2)S_{1,s}S_{1,t}+\frac{1}{n_1D^2}\Lambda_{1,st}^2\\&+\frac{A}{n_1D^2}\sum_{k=1}^{n_1}(\Upsilon_{1,s}^{(k)}h_t(\bX_{k})+\Upsilon_{1,t}^{(k)}h_s(\bX_{k}))
	+\frac{B}{n_1D^2}(\Lambda_{1,s}S_{1,t}+\Lambda_{1,t}S_{1,s} )\Big|.
	\end{align*}
	By using the triangle inequality on $L_{3,st}$, we have   
	$
	L_{3,st}\le J_{1,st}+J_{2,st}+J_{3,st}+J_{4,st}+J_{5,st},
	$
	where   
	\[
	\begin{array}{ll}
	J_{1,st}:=\Big|\dfrac{A^2-D^2}{n_1D^2}V_{1,st}^2\Big|,~J_{2,st}:=\Big|\dfrac{2AB+n_1B^2}{n_1D^2}S_{1,s}S_{1,t}\Big|,\\
	J_{3,st}:=\Big| \dfrac{1}{n_1D^2}\Lambda_{1,st}^2\Big|, J_{4,st}:=\Big|\dfrac{A}{n_1D^2}\sum_{k=1}^{n_1}\big(\Upsilon_{1,s}^{(k)}h_t(\bX_{k})+\Upsilon_{1,t}^{(k)}h_s(\bX_{k})\big)\Big|,\\
	J_{5,st}:=\Big|\dfrac{B}{n_1D^2}(\Lambda_{1,s}S_{1,t}+\Lambda_{1,t}S_{1,s}  )\Big|.
	\end{array}
	\]
	
	We now bound $J_{1,st},\ldots, J_{5, st}$ separately.
	By the definitions of $A$ and $D$, we obtain 
	\[
	A=O(n_1^{m-1}),~ D=O(n_1^{m-1})~ {\rm and}
	~D-A=\binom{n_1-2}{m-2}=O(n_1^{m-2}).\]
	Thus, for $J_{1,st}$, by the definition of  $V_{1,st}$ in (\ref{def:VGammaD}), by Assumption {\bf (M2)}  we easily have that 
	$\max_{1\le s,t \le q}J_{1,st} = O_p(n_1^{-1})$.
	For $J_{2,st}$,  considering $B=O(n_1^{m-2})$, we use the exponential inequality  to have
	\begin{equation}\label{BoundS}
	\begin{aligned}
	\P(J_{2,st}>y)=\P\Big(\frac{S_{1,s}S_{1,t}}{n_1^2}\ge Cy\Big)
	\le C_1\exp(-C_2n_1\min (y, \sqrt{y})).
	\end{aligned}
	\end{equation}
    With probability $1-Cn_1^{-1}$, we then have 
	$\max_{1\le s, t\le  q}J_{2,st}\le \log(qn_1) n_1^{-1}$ for sufficiently large $n_1$.
	We then bound $J_{3,st}$. Recalling  $\Lambda_{1,st}^2:=\sum_{k=1}^{n_1}\Upsilon_{1,s}^{(k)}\Upsilon_{1,t}^{(k)}$ in (\ref{def:VGammaD}), we have
	\begin{align*}
	\P(J_{3,st}>y)&=\P\Big(\frac{\Lambda_{1,st}^2}{n_1^{2m-1}}\ge Cy\Big)=\P\Big(\sum\limits_{k=1}^{n_1}\Upsilon_{1,s}^{(k)}\Upsilon_{1,t}^{(k)}\ge Cn_1^{2m-1}y\Big)\\
&	\le\sum\limits_{k=1}^{n_1}\P\Big(\Upsilon_{1,s}^{(k)}\Upsilon_{1,t}^{(k)}\ge Cn_1^{2m-2}y\Big).
	\end{align*}
	By the definition of $\Upsilon_{1,s}^{(k)}$ in (\ref{def:Upsilon}), given $\bX_{k}$, we can treat
	\[
	\Psi_{s}(\bX_{k},\bX_{\ell_1}\ldots,\bX_{\ell_m-1})-\big(h_{ij}(\bX_{k})+\sum_{r=1}^{m-1} h_{ij}(\bX_{\ell_r})\big),
	\] as a symmetric kernel function. Therefore, 
	$\Upsilon_{1,s}^{(k)}/D|\bX_k$ is  a $U$-statistic with  a kernel function of zero mean  and $m-1$ order.  Hence,  similarly to $L_2$, we  threshold the kernel with $C\log(qn_1)$ and use the exponential inequality for $U$-statistics to obtain that for sufficiently large $n_1$ with probability with $1-C_1n_1^{-1}$, we have 
	\[
	\max_{1\le s, t\le q} J_{3,st} \le C\frac{\log^2(qn)}{n}.
	\]

	We now bound $J_{4,st}$ and $J_{5,st}$. For $J_{4,st}$, we use the Cauchy-Swartz inequality on $\sum_{k=1}^{n_1}
	\Upsilon_{1,s}^{(k)}h_{t}(\bX_{k})$  and $\sum_{k=1}^{n_1}
	\Upsilon_{1,t}^{(k)}h_{s}(\bX_{k})$  to obtain
	\begin{equation}\label{ineq:boundJ4}
	J_{4,st}\le\big| \frac{A}{n_1D^2}(\Lambda_{1,ss}V_{1,tt}+\Lambda_{1,tt}V_{1,ss})\Big|.
	\end{equation}
	For $J_{5,st}$, by using the Cauchy-Swartz inequality on $\Lambda_{1,s}$ and $S_{1,s}$, we have 
	\begin{equation}\label{ineq:boundJ5}
	J_{5,st}\le \Big|\frac{B}{D^2}(\Lambda_{1,ss}V_{1,tt}+\Lambda_{1,tt}V_{1,ss})\Big|.
	\end{equation}
	Combining (\ref{ineq:boundJ4}) and (\ref{ineq:boundJ5}), we have 
	\begin{equation}\label{BoundJ45}
	J_{4,st}+J_{5,st}\le\underbrace{\Big| \frac{A+n_1B}{n_1D^2}(\Lambda_{1,ss}V_{1,tt}+\Lambda_{1,tt}V_{1,ss})
		\Big|}_{J_{6, st}}.
	\end{equation}
 Considering  $A=O(n_1^{m-1}),~B=O(n_1^{m-2}),$ and $D=O(n_1^{m-1})$, by the triangle inequality we have 
	\[
	\max_{1\le s,t \le q}J_{6, st}\le C\max_{1\le s,t\le q} \frac{\Lambda_{1,ss}V_{1,tt}}{n_1^m} = \Big(\max_{1\le s \le q}\underbrace{\frac{\Lambda^2_{1,ss}}{n_1^{2m-3/2}}}_{J'_{6,s}}\max_{1\le s\le q}\underbrace{\frac{V^2_{1,ss}}{n_1^{3/2}}}_{J''_{6,s}} \Big)^{1/2}.
	\]
	Similarly to $L_4$, from Assumption {\bf (M2)}, we have 
	$\max _{1\le s\le q} J''_{6,s}=O_p(n_1^{-1/2})$.
	For $J'_{6,s}$, we  have
	\begin{align}\label{BoundJ6_2_1}
	\P\Bigl(\frac{\Lambda_{1,ss}^2}{n_1^{2m-3/2}}\ge y \Bigr)
	\le \sum\limits_{k=1}^{n_1}\P\Big(\frac{|\Upsilon_{1,ss}^{(k)}|}{n_1^{m-1}}\ge Cn_1^{-1/4}y^{1/2}\Big).
	\end{align}
	By thresholding kernel with $C\log(qn)$ and exponential inequality for $U$-statistics, for sufficiently large $n_1$, 
	 $\max_{1\le s\le q}J'_{6,s} \le \log ^3(qn_1)n_1^{-1/2}
	 $
	 holds with probability $1-C_1n_1^{-1}$. Therefore, we have \[
	 \max_{1\le s,t \le q} J_{6,st} \le C \log^{3/2}(qn_1) n_1^{-1/2}.\]
	From all above results, for sufficiently large $n_1$, with probability $1-C_1 n_1^{-1}$, we have
	\begin{equation}\label{ineq:sigmaError}
	\max_{1\le, s, t\le q} |\hat{\sigma}_{1,st}-\sigma_{1,st}|\le C\frac{\log^{3/2} (qn_1)}{\sqrt{n_1}}.
	\end{equation}
	
	After analyzing the approximation error of $\hat{\sigma}_{1,st}$, we then prove for $\hat{r}_{1,st}$.
		By (\ref{def:R1R2hatR1R2}), we have $\hat{r}_{1,st}={\hat{\sigma}_{1,st}}/{\sqrt{\hat{\sigma}_{1,ss}
				\hat{\sigma}_{1,tt}}}$ and $r_{1,st}=\sigma_{1,st}/\sqrt{\sigma_{1,ss}\sigma_{1,tt}}$. Therefore, we have 
		\begin{align*}
		|\hat{r}_{1,st}-r_{1,st}|&=\Big|\frac{\hat{\sigma}_{1,st}}{\sqrt{\hat{\sigma}_{1,ss}\hat{\sigma}_{1,tt}}}-\frac{\sigma_{1,st}}{\sqrt{\sigma_{1,ss}\sigma_{1,tt}}}\Big|\\
		&\le\underbrace{\Big|\frac{\hat{\sigma}_{1,st}}{\sqrt{\hat{\sigma}_{1,ss}\hat{\sigma}_{1,tt}}}-\frac{\hat{\sigma}_{1,st}}{\sqrt{\sigma_{1,ss}\sigma_{1,tt}}}\Big|}_{A_1}+\underbrace{\Big|\frac{\hat{\sigma}_{1,st}}{\sqrt{\sigma_{1,ss}\sigma_{1,tt}}}-\frac{\sigma_{1,st}}{\sqrt{\sigma_{1,ss}\sigma_{1,tt}}}\Big|}_{A_2}.
		\end{align*}
		Hence, to bound $|\hat{r}_{1,st}-r_{1,st}|$  we  bound $A_1$ and $A_2$ separately.  For $A_1$, we rewrite it as
		\[
		A_1=\Big|\frac{\hat{\sigma}_{1,st}}{\sqrt{\hat{\sigma}_{1,ss}\hat{\sigma}_{1,tt}}}\Big|\Big|1-\frac{\sqrt{\hat{\sigma}_{1,ss}\hat{\sigma}_{1,tt}}}{\sqrt{\sigma_{1,ss}\sigma_{1,tt}}}\Big|.
		\]
		Considering $|\hat{r}_{1,st}|\le 1$ and $a^2-b^2=(a+b)(a-b)$, we have 
		\[
		A_1\le \sigma_{1,ss}^{-1}\sigma_{1,tt}^{-1} \big|{\hat{\sigma}_{1,ss}\hat{\sigma}_{1,tt}-\sigma_{1,ss}\sigma_{1,tt}}\big|.
		\]
		By Assumption {\bf (M1)} and {\bf (M2)}, there are constants $b$ and $B$,
		such that  $0<b\le\sigma_{1,ss}\le B< \infty$ for $s=1,\ldots,q$. Hence, we have
		\begin{equation}\label{ineq:A1}
		A_1\le b^{-2}\max_{1\le s\le q}|\hat{\sigma}_{1,ss}-\sigma_{1,ss}|^2+2Bb^{-2}\max_{1\le s\le q}|\hat{\sigma}_{1,ss}-\sigma_{1,ss}|.
		\end{equation}
		For $A_2$, by $\sigma_{1,ss}\ge b>0$ from Assumption {\bf (M1)} we have 
		\begin{equation}\label{ineq:A2}
		A_2\le 
		b^{-1}\max_{1\le s,t\le q}|\hat{\sigma}_{1,st}-\sigma_{1,st}|.
		\end{equation}
		Combining (\ref{ineq:sigmaError}), (\ref{ineq:A1}), and (\ref{ineq:A2}), we then have that
		\[
			\max_{1\le, s, t\le q} |\hat{r}_{1,st}-r_{1,st}|\le C\frac{\log^{3/2} (qn_1)}{\sqrt{n_1}},
		\]
		holds with the overwhelming probability,	which finishes the proof for $m>1$.
		
			We then prove for $m=1$. We  decompose $\hat{\sigma}_{1,st}$  as
		\[
		\hat{\sigma}_{1,st}=
		n_1^{-1}\sum_{k=1}^{n_1}\Psi_s(\bX_k)\Psi_t(\bX_k)-
		\overline{\Psi}_{1,s}\overline{\Psi}_{1,t},
		\]
		where $\Psi_{s}(\bX_k)=\Phi_{s}(\bX_k)-u_{1,s}$ and $\overline{\Psi}_{1,s}=n_1^{-1}\sum_{k=1}^{n_1}\Psi_{s}(\bX_k)$. Considering $\sigma_{1,st}=\E[\Psi_s(\bX)\Psi_t(\bX)]$,  by setting 
		\begin{align*}
		B_1 &= \P\Big(\max_{1\le s, t\le q}\big|{n_1}^{-1}\sum_{k=1}^{n_1}\Psi_s(\bX_k)\Psi_t(\bX_k)-
		\E[\Psi_s(\bX)\Psi_t(\bX)]\big|>x/2\Big)\\
		B_2&=\P\Big(\max_{1\le s, t\le q}\overline{\Psi}_{1,s}\overline{\Psi}_{1,t}>x/2\Big)
		\end{align*}
		we then have
		\begin{equation}\label{decomp:sigmahatst}
		\P\Big(\max_{1\le s, t\le q}\big|\hat{\sigma}_{\ \gamma,st}-\sigma_{\gamma,st}\big|>x\Big)\\
		 \le B_1 + B_2,
		\end{equation}
		By Theorem 6 in \cite{delaigle2011robustness:app}, we can bound $B_1$ by
		\begin{equation}\label{decomp:sigmahat1st}
		B_1
		\le Cq^2\exp(-C_1n_1x^2)+Cq^2\exp\big(-C_2(n_1x)^{1/2}\big).
		\end{equation}
		Similarly, for the term $B_2$ in (\ref{decomp:sigmahatst}), we use the same argument to obtain 
		\begin{equation}\label{decomp:sigmahat2st}
		\P\Big(\!\max_{1\le s, t\le q}\overline{\Psi}_{1,s}\overline{\Psi}_{1,t}\!>\!x/2\!\Big)\!\!\le\!\! Cq^2\exp(-C_1n_1x)\!+\!Cq^2\exp\!\big(\!-C_2(n_1\sqrt{x})\big).
		\end{equation}
		Combining (\ref{decomp:sigmahatst}), (\ref{decomp:sigmahat1st}), and (\ref{decomp:sigmahat2st}), for sufficiently large $n_1$, with probability $1-C_1 n_1^{-1}$, we have
		\begin{equation}\label{ineq:sigmaError:m1}
		\max_{1\le, s, t\le q} |\hat{\sigma}_{1,st}-\sigma_{1,st}|\le C\sqrt{\frac{\log (qn_1)}{n_1}} + C\frac{\log ^2(qn_1)}{n_1}.
		\end{equation}
		Similarly to $m>1$, we also have that
			\[
		\max_{1\le, s, t\le q} |\hat{r}_{1,st}-r_{1,st}|\le C\sqrt{\frac{\log (qn_1)}{n_1}} + C\frac{\log ^2(qn_1)}{n_1},
		\]
		holds with the overwhelming probability for $m=1$.
	\end{proof}

\section{More simulation results}\label{appendix:moreSimulation}
This section consists of three parts. Firstly, we  present the empirical size for high dimensional mean tests based on {\bf Models 2-4}, which are introduced in Section \ref{section:Exper}. Secondly, we  apply our methods to test high dimensional covariance/correlation coefficients to  illustrate the generality of proposed methods.
At last, we apply our methods to analyze  resting-state  functional magnetic resonance imaging (fMRI) data.

\subsection{Additional simulation results of testing   high dimensional mean values}
In Section \ref{section:Exper}, we introduce {\bf Models 1-4} for high dimensional mean tests.  In this section, we show the numerical results for {\bf Models 2-4} in Table \ref{table:Models2to4}. 
\begin{table}[!htbp]
	\begin{center}
		\addtolength{\tabcolsep}{-2pt}
		\scriptsize
		\caption{Empirical sizes of {\bf Models 2}, {\bf 3} and {\bf 4} with $\alpha=0.05$, $B=300$, and $n_1=n_2=100$ based on 2000
			replications.} \vspace{0.2cm}\label{table:Models2to4}
		\begin{tabular*}{12.3cm}{cccccccccp{2mm}cccc}
			\toprule[2pt]
			
			\multicolumn{2}{c}{}&\multicolumn{12}{c}{{\bf Model 2}}\\
			$d$&$s_0$&$p=1$&$p=2$&$p=3$&$p=4$&$p=5$&$p=\infty$&$T^N_{\rm ad}$&  &$T^2$&BY&SD&CLX\\\hline
			
			75 &5&6.20& 6.50 &6.55 &6.85 &7.00& 6.65&7.10 &
			& 5.25  &  6.50 & 5.40& 5.05
			\\
			&30&4.30 &4.75& 5.35 &6.00 &6.35 &6.75&6.25 &
			& 5.25  &  6.50 & 5.40& 5.05
			\\
			&75& 4.55& 4.75 &5.60 &6.00 &6.25 &6.50&6.30 &
			& 5.25  &  6.50 & 5.40& 5.05
			\\
			\hline
			200&10&5.20& 5.45 &5.75 &5.65 &6.20 &6.65& 6.30 &  
			&-& 5.35 & 4.60 &  6.10
			\\
			&50& 3.30 &3.40&3.80 &4.50 &5.30& 6.25& 5.30 & 
			&-& 5.35 & 4.60 &  6.10
			\\
			&100&2.85 &3.05 &3.35 &3.95& 4.75 &7.10& 5.10 &  
			&-& 5.35 & 4.60 &  6.10
			\\
			&150&3.00 &3.10 &3.55 &4.50& 5.10 &7.00& 5.50 & 
			&-& 5.35 & 4.60 &  6.10
			\\
			&200&2.70 &2.90 &3.40 &4.20 &5.05 &7.10& 5.15 & 
			&-& 5.35 & 4.60 &  6.10
			\\\hline
			
			400&10&4.85& 5.00 &5.45 &5.45 &5.95& 0.71&6.90 & 
			& - &5.10& 4.10& 6.25     
			\\
			&50&  1.90& 2.15 &2.60 &3.30 &3.90& 7.40& 5.45& 
			& - &5.10& 4.10& 6.25      
			\\
			&100& 1.35& 1.50&1.85& 2.80 &3.85 &7.20& 4.75&
			& - &5.10& 4.10& 6.25   
			\\
			&200&1.05 &1.15& 1.70 &2.65&3.70 &7.00 & 4.45& 
			& - &5.10& 4.10& 6.25         
			\\
			&400& 1.30 &1.65& 1.75 &2.70 &3.55 &7.10 & 4.50 &  
			& - &5.10& 4.10& 6.25    \\\hline
			\multicolumn{2}{c}{}&\multicolumn{12}{c}{{\bf Model 3}}\\
			$d$&$s_0$&$p=1$&$p=2$&$p=3$&$p=4$&$p=5$&$p=\infty$&$T^N_{\rm ad}$&  &$T^2$&BY&SD&CLX\\\hline
			
			75 &5& 5.25 &5.65 &6.25 &6.15 &6.30 &6.90& 6.75 &  
			&5.30& 6.10 & 5.40 &  5.90
			\\
			&30&4.70 &4.70 &5.35 &5.75 &6.20 &6.95&5.65 &
			&5.30& 6.10 & 5.40 &  5.90
			\\
			&75& 4.25& 4.80 &5.05& 5.10 &5.75 &7.00&5.75 &
			&5.30& 6.10 & 5.40 &  5.90
			\\
			\hline
			200&10&3.75 &4.05 &4.65& 5.20& 5.35 &7.05&5.85&  
			&-& 5.70 & 4.90 &  5.50
			\\
			&50& 2.80 &2.60 &3.20&3.50 &4.15 &6.70& 4.65 & 
			&-& 5.70 & 4.90 &  5.50
			\\
			&100&2.45& 2.50 &2.75 &3.50 &4.35 &6.60& 4.20 &  
			&-& 5.70 & 4.90 &  5.50
			\\
			&150&2.40& 2.55& 2.75 &3.70 &4.40 &7.05& 4.50 & 
			&-& 5.70 & 4.90 &  5.50
			\\
			&200&2.15& 2.30&2.75& 3.60 &4.35& 6.70& 4.65 & 
			&-& 5.70 & 4.90 &  5.50
			\\\hline
			400&10&3.95 &4.30 &4.80& 4.85 &5.30 &7.35& 6.05 &  
			&-& 5.25 & 3.95 &  6.25
			\\
			&50& 1.40 &1.80 &2.15& 2.55 &3.70 &7.15& 4.75 & 
			&-& 5.25 & 3.95 &  6.25
			\\
			&100&1.10 &1.20 &1.65& 2.25 &3.05 &7.05& 4.45 &  
			&-& 5.25 & 3.95 &  6.25
			\\
			&200& 0.90 &0.95 &1.25& 1.95 &3.20 &7.10& 4.35 & 
			&-& 5.25 & 3.95 &  6.25
			\\
			&400& 0.95 &0.75 &1.30 &2.10 &3.20 &7.15& 3.80 & 
			&-& 5.25 & 3.95 &  6.25
			\\\hline
			\multicolumn{2}{c}{}&\multicolumn{12}{c}{{\bf Model 4}}\\
			$d$&$s_0$&$p=1$&$p=2$&$p=3$&$p=4$&$p=5$&$p=\infty$&$T^N_{\rm ad}$&  &$T^2$&BY&SD&CLX\\\hline
			
			75 &5&4.10& 4.05 &4.05 &4.70 &4.95 &5.50&5.05 &
			& 4.10  & 3.90 & 3.60& 4.40
			\\
			&30&3.05&3.00 &3.20 &3.55& 3.90 &5.15&5.00 &
			& 4.10  & 3.90 & 3.60& 4.40
			\\
			&75& 2.75& 3.15& 3.30 &3.75 &4.10 &5.60 &4.65 &
			& 4.10  & 3.90 & 3.60& 4.40
			\\
			\hline
			200&10& 2.45 &2.75 &2.80& 3.10 &3.30 &5.30&4.20 &  
			&-& 1.75 & 1.50 &  4.35
			\\
			&50& 1.05 &1.05& 1.30 &1.75 &2.35& 5.50& 3.30 & 
			&-& 1.75 & 1.50 &  4.35
			\\
			&100& 1.10 &1.10& 1.20& 1.65 &2.35 &5.60& 3.00 &  
			&-& 1.75 & 1.50 &  4.35
			\\
			&150& 0.85& 0.90 &1.10& 1.45 &2.25 &5.65& 3.35 & 
			&-& 1.75 & 1.50 &  4.35
			\\
			&200&1.00 &1.10 &1.10 &1.65 &1.95 &5.65& 2.75& 
			&-& 1.75 & 1.50 &  4.35 
			\\\hline
			400&10& 2.85 &3.05 &3.35& 3.40 &4.15& 5.70& 4.20 &  
			&-& 0.85 & 0.45&  4.20
			\\
			&50&0.95 &0.95 &1.05 &1.30 &1.80 &5.65& 3.20 & 
			&-& 0.85 & 0.45&  4.20
			\\
			&100&0.45 &0.65 &0.60& 0.75& 1.20& 5.45& 2.80 &  
			&-& 0.85 & 0.45&  4.20
			\\
			&200&0.30 &0.30 &0.35 &1.00&1.60 &5.40& 2.70 & 
			&-& 0.85 & 0.45&  4.20
			\\
			&400&0.30 &0.30& 0.50 &0.70 &1.50 &5.50& 2.45 & 
			&-& 0.85 & 0.45&  4.20
			\\
			\bottomrule[2pt]
		\end{tabular*}
	\end{center}
\end{table}
\subsection{Simulation results of testing high dimensional covariance and correlation coefficients }
In this section, we carry out the simulation of the marginal  test using the Pearson's covariance and Kendall's tau correlation matrices. For simplicity, we  consider the one-sample problem. In the simulation, $Z$ and $\bX\in\reals^d$ are the response variable and the explanatory vector. We generate 
$n_1$ data points of $(Z, \bX^\top)^\top$
from the following models.
\begin{itemize}
	\item {\bf Model 5.} Let $\bSigma^L_0, \bSigma^L_1 \in\reals^{(d+1)\times(d+1)}$  to be
	\[
	\bSigma^L_0 \!\!=\!\!\left[\!\!\begin{array}{cc}
	1& \zero^\top\\
	\zero &(\Db^\star)^{-1/2}\bSigma^\star(\Db^\star)^{-1/2}
	\end{array}\!\!\right],~
	\bSigma^L_1\!\!=\!\!\left[\!\!\begin{array}{cc}
	1& \bV^\top\\
	\bV &(\Db^\star)^{-1/2}\bSigma^\star(\Db^\star)^{-1/2}
	\end{array}\!\!\right],
	\]
	where $\bV\in\reals^d$ has $s$ nonzero entries with the magnitude  ${\rm U}(u_1,u_2)$. Under the null hypothesis, we generate $n_1$ random vectors from $t(\nu, \bmu,\bSigma)$ with $\nu=5$, $\bmu=\zero$, $\bSigma=\bSigma^L_0$ as the samples of $(Z,\bX^\top)^\top$. Under the alternative hypothesis, we generate the samples of $(Z,\bX^\top)^\top$ from $t(5,\zero,\bSigma^L_1+\delta\bI_{d+1})$ with $\delta =|  
	\lambda_{\rm min}(\bSigma^L_1)|+0.5$.
\end{itemize}
The experimental results of {\bf Model 5} are in Table \ref{table:Model5}. In {\bf Model 5} we compare the proposed tests based on Pearson's covariance and Kendall' tau correlation matrices.
The pattern of empirical size and power for {\bf Model 5} is similar to {\bf Models 1-4}. Moreover,
the experiment
shows that Kendall's tau based test is more powerful than the Pearson's covariance based one for  distributions with the heavy tails and  strong tail dependence.

\begin{table}[!htbp]
	\begin{center}
		\addtolength{\tabcolsep}{-3pt}
		\tiny
		\caption{Empirical size and power of  {\bf Model 5} with $\alpha=0.05$, $B=300$, and $n_1=200$ based on 2000
			replications.} \vspace{0.2cm}\label{table:Model5}
		\begin{tabular*}{13.5cm}{cccccccccp{2mm}ccccccc}
			\toprule[2pt]
			\multicolumn{2}{c}{}&\multicolumn{15}{c}{
				Empirical size (\%) }\\
			\multicolumn{2}{c}{}&\multicolumn{7}{c}{Pesrson's sample covariance}&&\multicolumn{7}{c}{Kendall's tau}\\
			\cline{3-9}\cline{11-17}
			$d$&$s_0$&$p=1$&$p=2$&$p=3$&$p=4$&$p=5$&$p=\infty$&$T^N_{\rm ad}$&  &$p=1$&$p=2$&$p=3$&$p=4$&$p=5$&$p=\infty$&$T^N_{\rm ad}$\\\hline
			
			200& 10 &0.00& 0.00 &  0.00& 0.15& 0.15& 2.65
			&1.50&
			&3.75 &4.30 &4.75 &5.40& 6.05& 8.20
			&6.85\\
			&50& 0.00& 0.00& 0.00& 0.00& 0.05& 2.60
			&1.30&
			&1.20 &1.75 &1.90 &3.35 &4.35 &8.75
			&5.25\\
			&100& 0.00& 0.00& 0.00& 0.00& 0.00& 2.10
			&1.35&
			&0.60& 0.85& 1.70& 2.55& 3.95& 8.60
			&5.65\\
			&150& 0.00& 0.00& 0.00& 0.00& 0.00& 2.65
			&1.25&
			&0.60& 0.85& 1.50& 2.65& 3.75& 8.35
			&4.75\\
			&200& 0.00& 0.00& 0.00& 0.00& 0.00& 2.40
			&1.20&
			& 0.60& 0.90& 1.40& 2.80& 3.55& 8.20
			&5.25\\\hline
			\multicolumn{17}{c}{}\\
			\multicolumn{2}{c}{}&\multicolumn{15}{c}{ Empirical power (\%)  with $s=5$, $u_1=0$, and $u_2=4\sqrt{\log(d)/n_1}$ }\\
			$d$&$s_0$&$p=1$&$p=2$&$p=3$&$p=4$&$p=5$&$p=\infty$&$T^N_{\rm ad}$&  &$p=1$&$p=2$&$p=3$&$p=4$&$p=5$&$p=\infty$&$T^N_{\rm ad}$\\\hline
			200& 10&13.70 &20.7& 27.95& 34.00& 38.85& 55.20
			&50.60&
			&73.20& 78.75& 81.40& 83.45& 84.05& 84.20
			&84.00\\
			&50&  0.40& 1.50 &5.50& 14.00& 24.40& 55.95
			&49.05&
			& 26.90& 52.75& 70.55& 78.50& 81.80& 84.05
			&81.85\\
			&100&0.05& 0.40& 3.20& 11.55& 23.65& 55.65
			&48.85&
			&11.85& 39.20& 66.35 &77.05 &81.90& 84.15
			&82.00\\
			&150& 0.05& 0.30& 2.80& 12.00& 22.90& 55.20
			&48.25&
			& 8.30& 35.00& 65.65 &77.25 &81.55& 84.25
			&82.15\\
			&200&0.05& 0.40& 3.05& 11.85& 23.30& 55.20
			&47.55&
			& 6.95 &34.95& 65.55& 76.55 &81.70 &84.05
			&81.45\\\hline
			\multicolumn{17}{c}{}\\
			\multicolumn{2}{c}{}&\multicolumn{15}{c}{ Empirical power (\%) with $s=5$, $u_1=0$, and $u_2=3\sqrt{1/n_1}$ }\\
			$d$&$s_0$&$p=1$&$p=2$&$p=3$&$p=4$&$p=5$&$p=\infty$&$T^N_{\rm ad}$&  &$p=1$&$p=2$&$p=3$&$p=4$&$p=5$&$p=\infty$&$T^N_{\rm ad}$\\\hline
			200& 10& 10.25& 10.55& 11.35 &11.65 &12.80& 16.80
			&13.95&
			&75.85& 75.05 &74.15 &72.90& 70.65 &46.55
			&68.20\\
			&50& 4.90& 5.55 &6.85 &8.40 &9.6 &16.25
			&12.60&
			&78.30& 79.80 &80.60 &79.50& 77.10 &47.15
			&74.50\\
			&100&2.95& 4.05& 4.85 &6.45 &8.2 &17.20
			&11.15&
			& 73.80& 78.65& 80.90&80.35& 77.60 &46.85
			&75.10\\
			&150& 2.70& 4.00 &5.15 &6.50 &8.60 &16.90
			&11.15&
			&69.55 &78.00 &80.75 &79.70& 77.90 &47.45
			&73.75\\
			&200&2.75& 3.65& 5.35& 6.85& 8.60& 16.45
			&11.15&
			& 68.00& 78.45& 81.15& 80.70 &77.60& 46.95
			&74.25\\
			\bottomrule[2pt]
		\end{tabular*}
	\end{center}
\end{table}

\subsection{Simulation results of increasing $\#(\mathcal{P})$}\label{sec:simP}
In this section, we discuss the impact of $\#(\mathcal{P})$ by simulation. In Sections \ref{sec:adTestProc}  and \ref{sec:theroAd}, we require fixed $\mathcal{P}$ for the data-adaptive combined test. In Remark \ref{remark:P}, we discuss theoretical difficulties of increasing $\#(\mathcal{P})$.  In this section, we present the performance of proposed methods under various  $\mathcal{P}$.

 For this we generate the data based on {\bf Model 1} in Section \ref{section:Exper}. We consider various $\cP$. In detail, we set $\cP_1=\{1,2\}$, $\cP_2=\{1,2,\infty\}$, $\cP_3=\{1,2,3,4,5\}$, $\cP_4=\{1,2,3,4,5,\infty \}$, $\cP_5=\{1,2,\ldots,10,\infty\}$, and $\cP_6 =\{1,2,\ldots,20,\infty\}$. We also consider various alternatives with $s=5,50, 100$, from sparse to dense. The simulation results are in Table \ref{table:P}.

\begin{table}[!htbp]
	\begin{center}
		\addtolength{\tabcolsep}{-1pt}
		
		\caption{Empirical size and power of $T_{\rm ad}^N$ under {\bf Model 1} with $\alpha=0.05$, $B=300$, $d=400$, and $n_1=n_2=200$ based on 1000
			replications.} \vspace{0.2cm}\label{table:P}
		\begin{tabular*}{12.5cm}{cccccccccccccc}
			\toprule[2pt]
			&\multicolumn{6}{c}{} &&\multicolumn{6}{c}{Empirical power (\%) with}\\
		   &\multicolumn{6}{c}{Empirical size(\%) with}&&\multicolumn{6}{c}{$s=5$, $u_1=0$, $u_2=4\sqrt{\log(d)/n_1}$}\\\hline
			$s_0$&$\mathcal{P}_1$&$\mathcal{P}_2$&$\mathcal{P}_3$&$\mathcal{P}_4$&$\mathcal{P}_5$&$\mathcal{P}_6$&&$\mathcal{P}_1$&$\mathcal{P}_2$&$\mathcal{P}_3$&$\mathcal{P}_4$&$\mathcal{P}_5 $&$\mathcal{P}_6$\\\hline
			10&5.1&5.6&5.3&5.2&5.3&5.4&& 82.8&86.3&86.6&86.4&86.5 &86.5\\
			50&3.7&4.6&4.8&4.6&5.1&4.6&&60.4&84.0&82.4&84.7&85.0&85.1\\
			100&2.9&3.9&3.7&4.5&4.6&4.5&&44.2&83.7&81.6&84.5&84.8&85.0\\
			150&2.6&3.6&3.5&3.8&4.1&4.2&&33.6&83.3&80.8&83.6&84.6&84.5\\
			200&2.3&4.0&3.5&3.9&4.0&4.1&&29.1&83.5 &81.1 &84.1 &84.2 &84.8\\\hline
		     &\multicolumn{6}{c}{Empirical power (\%) with}&&\multicolumn{6}{c}{Empirical power (\%)with}\\
			&\multicolumn{6}{c}{$s=50$, $u_1=0$,$u_2=4\sqrt{1/n_1}$}&&\multicolumn{6}{c}{$s=100$, $u_1=0$, $u_2=3\sqrt{1/n_1}$}\\\hline
			$s_0$&$\mathcal{P}_1$&$\mathcal{P}_2$&$\mathcal{P}_3$&$\mathcal{P}_4$&$\mathcal{P}_5 $&$\mathcal{P}_6$&&$\mathcal{P}_1$&$\mathcal{P}_2$&$\mathcal{P}_3$&$\mathcal{P}_4$&$\mathcal{P}_5 $&$\mathcal{P}_6$\\\hline
			10&76.0 &75.2& 78.6& 77.3& 74.8& 75.1&& 71.1 &65.1& 0.1& 69.5& 64.1& 64.8\\
			50&75.8& 75.2 &79.0&78.2&78.0&78.0&& 79.6& 73.8& 77.9& 76.3& 72.7& 74.3\\
			100&70.1& 71.1& 79.4& 78.3& 77.3& 76.1&& 78.6& 72.0& 77.0& 76.0& 73.1& 74.2\\
			150& 65.0& 67.5& 78.1& 77.2& 75.3& 75.3	&& 75.8& 68.5& 76.9& 75.8& 73.1& 73.7\\
			200&60.2& 65.9& 76.8& 76.5& 74.6& 74.0&& 74.5& 68.5& 76.0& 75.1& 73.8& 73.8\\
			\bottomrule[2pt]
		\end{tabular*}
	\end{center}
\end{table}
From Table \ref{table:P}, we recommend using $\cP_4=\{1,2,3,4,5,\infty\}$. It has good performance for both sparse and dense alternatives. Table  \ref{table:P} also shows that there is no power advantage to add more elements to $\cP_4$.

\subsection{Real data example}
\label{sec:RealDataExample}

In this section, we apply our methods to analyze  resting-state  functional magnetic resonance imaging (fMRI) data.  We aim  to compare the resting-sate fMRI scans between the attention deficit hyperactivity disorder (ADHD)
and normal children. For each subject,  the resting-state fMRI scan is a high dimensional time series. 
Instead of dealing with the time series directly, we alternatively use an index named amplitude of low frequency fluctuation (ALFF) to yield a high dimensional vector for each subject. Each entry of ALFF is defined as the total power within the frequency range between 0.01 and 0.1 Hz of 
the corresponding  entry of the original fMRI time series, which reflects the slow fluctuation. In general, ALFF reflects the intensity of regional spontaneous brain activity. As for the detailed definition of ALFF, we refer to \cite{yu2007altered:app}.  Existing literature \citep{yu2007altered:app,zou2008improved:app} utilizes   univariate two-sample $t$-tests  to detect differentially experessed brain areas between the diseased and control groups based on ALFF. Before we conduct the univariate two-sample tests, it is a common practice to perform a global test to verify that there is significant difference  of ALFFs between two groups. By the definition of ALFF, we utilize the high dimensional mean test to perform the global test. 

\begin{figure}[!htbp]
	\begin{center}
		\resizebox{10.5cm}{12cm}{
			\includegraphics[width=0.7\textwidth,angle=0]{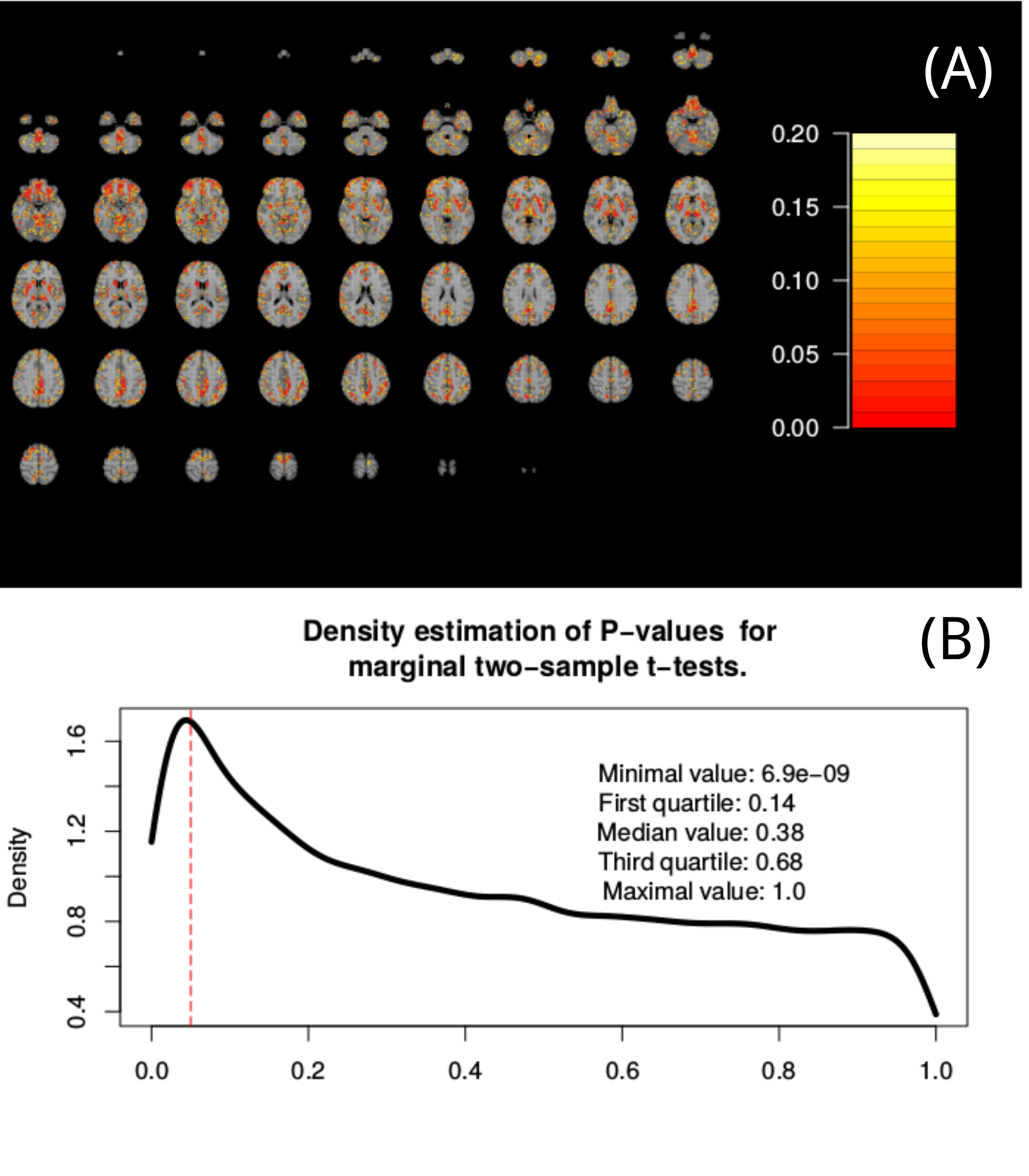}
		}
		\vspace*{-4em}
	\end{center}
	\caption{\small $P$-values of the marginal two-sample $t$-tests on ALFFs between   ADHD and control groups. (A) The $P$-value map on the standard MN152 brain template with the slice thickness 3mm at the given threshold ($P$-value $<0.2$). (B) The estimated density of the $P$-values  and some summary statistics. } \label{fig:pvalues}
	
\end{figure}

Our experiment is based on the first dataset of Peking University from the ADHD-200 sample.\footnote{ The website for ADHD-200 sample is  http://fcon\_1000.projects.nitrc.org/indi/adhd200/.}
The sample consists of 85 subjects, in which 24 subjects have ADHD. Therefore, the control group has 61 subjects. 
ALFF analysis is performed by using the C-PAC software. The C-PAC software preprocesses the  data by registering each person's fMRI scan to the standard MN152 template. To increase the signal-noise ratio, the software also performs  slice timing correction,
body motion correction,  nuisance signal correction, and
temporal filtering. Because of the difference of individual brain baseline activity,  we standardize the ALFF  for each subject. We then use the Gaussian kernel to perform the spatial  smoothing for each subject. Moreover, existing literature and psychological knowledge suggest that the ALFF of brain's gray matter  is related to the mental disease. Hence, we restrict the testing area to the gray matter of the brain. For detailed description of the processing procedure, we refer to \cite{yu2007altered:app}, \cite{zou2008improved:app}, and the user guide of  C-PAC software.\footnote{ The website for the C-PAC software is http://fcp-indi.github.io/.}

Figure  \ref{fig:pvalues} illustrates  $P$-values of univariate two-sample $t$-tests. Figure \ref{fig:pvalues}(A) illustrates  the $P$-value map to the standard MN152 brain template with the slice thickness 3mm at the given threshold ($P$-value $<0.2$). Moreover, Figure \ref{fig:pvalues}(B) illustrates the estimated density  of these $P$-values.   Figure \ref{fig:pvalues} shows there are significant   ALFF differences between the diseased and control groups in some brain areas. 

\begin{table}[!ht]
	\begin{center}
		\caption{$P$-values of the $(s_0,p)$-norm tests and  data-adaptive combined test with $s_0=40,400,4000, 8000$ and  $B=1000$ on the ALFF data.} \vspace{0.2cm}\label{table:realdata}
		\begin{tabular*}{9.5cm}{lccccccc}
			\toprule[2pt]
			\multicolumn{8}{c}{
				$P$-values of global tests between the ADHD and control groups }\\
			$s_0$&$p=1$&$p=2$&$p=3$&$p=4$&$p=5$&$p=\infty$&$T^N_{\rm ad}$\\
			\hline
			40& 0.001&0.001 &0.001 &0.001 &0.001 &0.001 & 0.000\\
			400&0.013 &0.013 &0.012 &0.011 &0.010 &0.000 & 0.000\\
			4000&0.016&0.015 &0.015 &0.015 &0.013 &0.000 & 0.000\\
			8000&0.016 &0.015 &0.013 &0.011 &0.008 &0.000 & 0.000
			
			\\\hline
			\multicolumn{8}{c}{}\\
			\multicolumn{8}{c}{ $P$-values of global tests within the  control group }\\
			$s_0$&$p=1$&$p=2$&$p=3$&$p=4$&$p=5$&$p=\infty$&$T^N_{\rm ad}$\\\hline
			40&  0.192& 0.192 &0.193 &0.195 &0.196& 0.254 &0.237\\
			400& 0.301 &0.295 &0.290 &0.288 &0.284& 0.299
			&0.355\\
			4000&0.373 &0.362 &0.352 &0.337 &0.323 &0.273
			&0.354\\
			8000& 0.406 &0.394 &0.387 &0.375 &0.360 &0.282
			&0.370\\
			\bottomrule[2pt]
		\end{tabular*}
	\end{center}
\end{table}

We then  use both the individual $(s_0,p)$-norm test and   data-adaptive combined test with balanced $\mathcal{P}=\{1,\ldots,5,\infty\}$ to perform the global test. We also  
randomly split the sample for the control group into  two subsamples with  $30$ and $31$ subjects.
We then  perform the global mean test between the two subsamples of the control group to confirm the validity of our proposed methods.   As is shown in Figure \ref{fig:pvalues}, at most  20\% of the
gray matter  is potentially different between the diseased and control groups. Therefore, considering that the voxel size is about 40000,  we set $s_0=40,400,4000,8000$ in the experiment. The experiment result is presented in Table \ref{table:realdata}, which shows that
our proposed methods are quite powerful to distinguish the ADHD and control groups.

\end{bibunit}

\begin{thebibliography}{76}
	\expandafter\ifx\csname natexlab\endcsname\relax\def\natexlab#1{#1}\fi
	\expandafter\ifx\csname url\endcsname\relax
	\def\url#1{\texttt{#1}}\fi
	\expandafter\ifx\csname urlprefix\endcsname\relax\def\urlprefix{URL }\fi
	
	\bibitem[{Anderson(2003)}]{anderson2003introduction}
	\textsc{Anderson, T.~W.} (2003).
	\newblock \textit{An Introduction to Multivariate Statistical Analysis (3rd)}.
	\newblock Wiley New York.
	
	\bibitem[{Arcones and Gine(1993)}]{arcones1993limit}
	\textsc{Arcones, M.~A.} and \textsc{Gine, E.} (1993).
	\newblock {Limit theorems for ${U}$-processes}.
	\newblock \textit{Annals of Probability} \textbf{21} 1494--1542.
	
	\bibitem[{Bai et~al.(2009)Bai, Jiang, Yao and Zheng}]{bai2009corrections}
	\textsc{Bai, Z.}, \textsc{Jiang, D.}, \textsc{Yao, J.} and \textsc{Zheng, S.}
	(2009).
	\newblock Corrections to {LRT} on large-dimensional covariance matrix by {RMT}.
	\newblock \textit{Annals of Statistics} \textbf{37} 3822--3840.
	
	\bibitem[{Bai and Saranadasa(1996)}]{bai1996effect}
	\textsc{Bai, Z.} and \textsc{Saranadasa, H.} (1996).
	\newblock Effect of high dimension: by an example of a two sample problem.
	\newblock \textit{Statistica Sinica} \textbf{6} 311--329.
	
	\bibitem[{Bai and Yin(1993)}]{bai1993limit}
	\textsc{Bai, Z.} and \textsc{Yin, Y.} (1993).
	\newblock Limit of the smallest eigenvalue of a large dimensional sample
	covariance matrix.
	\newblock \textit{Annals of Probability} \textbf{21} 1275--1294.
	
	\bibitem[{Bao et~al.(2015)Bao, Lin, Pan and Zhou}]{bao2015spectral}
	\textsc{Bao, Z.}, \textsc{Lin, L.}, \textsc{Pan, G.} and \textsc{Zhou, W.}
	(2015).
	\newblock Spectral statistics of large dimensional {S}pearman's rank
	correlation matrix and its application.
	\newblock \textit{Annals of Statistics} \textbf{43} 2588--2623.
	
	\bibitem[{Barbe and Bertail(2012)}]{barbe2012weighted}
	\textsc{Barbe, P.} and \textsc{Bertail, P.} (2012).
	\newblock \textit{The Weighted Bootstrap}, vol.~98.
	\newblock Springer Science and Business Media.
	
	\bibitem[{Basu and Pan(2011)}]{basu2011comparison}
	\textsc{Basu, S.} and \textsc{Pan, W.} (2011).
	\newblock Comparison of statistical tests for disease association with rare
	variants.
	\newblock \textit{Genetic Epidemiology} \textbf{35} 606--619.
	
	\bibitem[{Benjamini and Hochberg(1995)}]{benjamini1995controlling}
	\textsc{Benjamini, Y.} and \textsc{Hochberg, Y.} (1995).
	\newblock Controlling the false discovery rate: A practical and powerful
	approach to multiple testing.
	\newblock \textit{Journal of the Royal Statistical Society. Series B
		(Methodological)} \textbf{57} 289--300.
	
	\bibitem[{Birke and Dette(2005)}]{birke2005note}
	\textsc{Birke, M.} and \textsc{Dette, H.} (2005).
	\newblock A note on testing the covariance matrix for large dimension.
	\newblock \textit{Statistics and Probability Letters} \textbf{74} 281--289.
	
	\bibitem[{Bonn{\'e}ry et~al.(2012)Bonn{\'e}ry, Breidt and
		Coquet}]{bonnery2012uniform}
	\textsc{Bonn{\'e}ry, D.}, \textsc{Breidt, F.~J.} and \textsc{Coquet, F.}
	(2012).
	\newblock Uniform convergence of the empirical cumulative distribution function
	under informative selection from a finite population.
	\newblock \textit{Bernoulli} \textbf{18} 1361--1385.
	
	\bibitem[{Cai and Jiang(2012)}]{cai2012phase}
	\textsc{Cai, T.} and \textsc{Jiang, T.} (2012).
	\newblock Phase transition in limiting distributions of coherence of
	high-dimensional random matrices.
	\newblock \textit{Journal of Multivariate Analysis} \textbf{107} 24--39.
	
	\bibitem[{Cai and Liu(2011)}]{cai2011direct}
	\textsc{Cai, T.} and \textsc{Liu, W.} (2011).
	\newblock A direct estimation approach to sparse linear discriminant analysis.
	\newblock \textit{Journal of the American Statistical Association} \textbf{106}
	1566--1577.
	
	\bibitem[{Cai et~al.(2013)Cai, Liu and Xia}]{cai2013}
	\textsc{Cai, T.}, \textsc{Liu, W.} and \textsc{Xia, Y.} (2013).
	\newblock {Two-sample covariance matrix testing and support recovery in
		high-dimensional and sparse settings}.
	\newblock \textit{Journal of the American Statistical Association} \textbf{108}
	265--277.
	
	\bibitem[{Cai et~al.(2014)Cai, Liu and Xia}]{tony2014two}
	\textsc{Cai, T.}, \textsc{Liu, W.} and \textsc{Xia, Y.} (2014).
	\newblock Two-sample test of high dimensional means under dependence.
	\newblock \textit{Journal of the Royal Statistical Society: Series B
		(Statistical Methodology)} \textbf{76} 349--372.
	
	\bibitem[{Cai and Jiang(2011)}]{cai2011limiting}
	\textsc{Cai, T.~T.} and \textsc{Jiang, T.} (2011).
	\newblock Limiting laws of coherence of random matrices with applications to
	testing covariance structure and construction of compressed sensing matrices.
	\newblock \textit{Annals of Statistics} \textbf{39} 1496--1525.
	
	\bibitem[{Cai and Ma(2013)}]{cai2013optimal}
	\textsc{Cai, T.~T.} and \textsc{Ma, Z.} (2013).
	\newblock Optimal hypothesis testing for high dimensional covariance matrices.
	\newblock \textit{Bernoulli} \textbf{19} 2359--2388.
	
	\bibitem[{Cakici(2015)}]{cakici2015five}
	\textsc{Cakici, N.} (2015).
	\newblock The five-factor {F}ama-{F}rench model: International evidence.
	\newblock \textit{Available at SSRN 2601662} .
	
	\bibitem[{Castagna et~al.(2003)Castagna, Sun and
		Siegfried}]{castagna2003instantaneous}
	\textsc{Castagna, J.~P.}, \textsc{Sun, S.} and \textsc{Siegfried, R.~W.}
	(2003).
	\newblock Instantaneous spectral analysis: {D}etection of low-frequency shadows
	associated with hydrocarbons.
	\newblock \textit{The Leading Edge} \textbf{22} 120--127.
	
	\bibitem[{Chang et~al.(2014)Chang, Zhou and Zhou}]{chang2014simulation}
	\textsc{Chang, J.}, \textsc{Zhou, W.} and \textsc{Zhou, W.} (2014).
	\newblock Simulation-based hypothesis testing of high dimensional means under
	covariance heterogeneity.
	\newblock \textit{arXiv preprint arXiv:1406.1939} .
	
	\bibitem[{Chang et~al.(2015)Chang, Zhou and Zhou}]{chang2015}
	\textsc{Chang, J.}, \textsc{Zhou, W.} and \textsc{Zhou, W.} (2015).
	\newblock Bootstrap tests on high dimensional covariance matrices with
	applications to understanding gene clustering.
	\newblock \textit{arXiv preprint arXiv:1505.04493} .
	
	\bibitem[{Chen and Qin(2010)}]{chen2010two}
	\textsc{Chen, S.} and \textsc{Qin, Y.} (2010).
	\newblock A two-sample test for high-dimensional data with applications to
	gene-set testing.
	\newblock \textit{Annals of Statistics} \textbf{38} 808--835.
	
	\bibitem[{Chernozhukov et~al.(2013)Chernozhukov, Chetverikov and
		Kato}]{chernozhukov2013gaussian}
	\textsc{Chernozhukov, V.}, \textsc{Chetverikov, D.} and \textsc{Kato, K.}
	(2013).
	\newblock Gaussian approximations and multiplier bootstrap for maxima of sums
	of high-dimensional random vectors.
	\newblock \textit{Annals of Statistics} \textbf{41} 2786--2819.
	
	\bibitem[{Chernozhukov et~al.(2014)Chernozhukov, Chetverikov and
		Kato}]{chernozhukov2014central}
	\textsc{Chernozhukov, V.}, \textsc{Chetverikov, D.} and \textsc{Kato, K.}
	(2014).
	\newblock Central limit theorems and bootstrap in high dimensions.
	\newblock \textit{arXiv preprint arXiv:1412.3661} .
	
	\bibitem[{Comon(1994)}]{comon1994independent}
	\textsc{Comon, P.} (1994).
	\newblock Independent component analysis, a new concept?
	\newblock \textit{Signal Processing} \textbf{36} 287--314.
	
	\bibitem[{Cox and Hinkley(1979)}]{cox1979theoretical}
	\textsc{Cox, D.~R.} and \textsc{Hinkley, D.~V.} (1979).
	\newblock \textit{Theoretical {S}tatistics}.
	\newblock CRC Press.
	
	\bibitem[{Dembo and Shao(2006)}]{dembo2006large}
	\textsc{Dembo, A.} and \textsc{Shao, Q.} (2006).
	\newblock Large and moderate deviations for {H}otelling's ${T}^2$-statistic.
	\newblock \textit{Electronic Communications in Probability} \textbf{11}
	149--159.
	
	\bibitem[{Fama and French(1993)}]{fama1993common}
	\textsc{Fama, E.~F.} and \textsc{French, K.~R.} (1993).
	\newblock Common risk factors in the returns on stocks and bonds.
	\newblock \textit{Journal of Financial Economics} \textbf{33} 3--56.
	
	\bibitem[{Fama and French(2012)}]{fama2012size}
	\textsc{Fama, E.~F.} and \textsc{French, K.~R.} (2012).
	\newblock Size, value, and momentum in international stock returns.
	\newblock \textit{Journal of Financial Economics} \textbf{105} 457--472.
	
	\bibitem[{Fama and French(2015)}]{fama2015}
	\textsc{Fama, E.~F.} and \textsc{French, K.~R.} (2015).
	\newblock A five-factor asset pricing model.
	\newblock \textit{Journal of Financial Economics} \textbf{116} 1--22.
	
	\bibitem[{Fama and French(2016)}]{fama2016international}
	\textsc{Fama, E.~F.} and \textsc{French, K.~R.} (2016).
	\newblock International tests of a five-factor asset pricing model.
	\newblock \textit{Journal of Financial Economics} \textbf{123} 441--463.
	
	\bibitem[{Fan and Fan(2008)}]{fan2008high}
	\textsc{Fan, J.} and \textsc{Fan, Y.} (2008).
	\newblock High dimensional classification using features annealed independence
	rules.
	\newblock \textit{Annals of statistics} \textbf{36} 2605.
	
	\bibitem[{Fan et~al.(2012)Fan, Feng and Tong}]{fan2012road}
	\textsc{Fan, J.}, \textsc{Feng, Y.} and \textsc{Tong, X.} (2012).
	\newblock A road to classification in high dimensional space: The regularized
	optimal affine discriminant.
	\newblock \textit{Journal of the Royal Statistical Society: Series B
		(Statistical Methodology)} \textbf{74} 745--771.
	
	\bibitem[{Gombay and Horvath(2002)}]{gombay2002rates}
	\textsc{Gombay, E.} and \textsc{Horvath, L.} (2002).
	\newblock Rates of convergence for {$U$}-statistic processes and their
	bootstrapped versions.
	\newblock \textit{Journal of Statistical Planning and Inference} \textbf{102}
	247--272.
	
	\bibitem[{Han and Liu(2014)}]{han2014distribution}
	\textsc{Han, F.} and \textsc{Liu, H.} (2014).
	\newblock Distribution-free tests of independence with applications to testing
	more structures.
	\newblock \textit{arXiv preprint arXiv:1410.4179} .
	
	\bibitem[{Han et~al.(2013)Han, Zhao and Liu}]{han2013coda}
	\textsc{Han, F.}, \textsc{Zhao, T.} and \textsc{Liu, H.} (2013).
	\newblock {CODA: High dimensional copula discriminant analysis}.
	\newblock \textit{Journal of Machine Learning Research} \textbf{14} 629--671.
	
	\bibitem[{Ho et~al.(2008)Ho, Stefani, dos Remedios and
		Charleston}]{ho2008differential}
	\textsc{Ho, J.~W.}, \textsc{Stefani, M.}, \textsc{dos Remedios, C.~G.} and
	\textsc{Charleston, M.~A.} (2008).
	\newblock Differential variability analysis of gene expression and its
	application to human diseases.
	\newblock \textit{Bioinformatics} \textbf{24} 390--398.
	
	\bibitem[{Hu et~al.(2010)Hu, Qiu and Glazko}]{hu2010new}
	\textsc{Hu, R.}, \textsc{Qiu, X.} and \textsc{Glazko, G.} (2010).
	\newblock A new gene selection procedure based on the covariance distance.
	\newblock \textit{Bioinformatics} \textbf{26} 348--354.
	
	\bibitem[{Hu et~al.(2009)Hu, Qiu, Glazko, Klebanov and
		Yakovlev}]{hu2009detecting}
	\textsc{Hu, R.}, \textsc{Qiu, X.}, \textsc{Glazko, G.}, \textsc{Klebanov, L.}
	and \textsc{Yakovlev, A.} (2009).
	\newblock Detecting intergene correlation changes in microarray analysis: A new
	approach to gene selection.
	\newblock \textit{BMC Bioinformatics} \textbf{10} 20.
	
	\bibitem[{Huskova and Jansen(1993)}]{huvskova1993generalized}
	\textsc{Huskova, M.} and \textsc{Jansen, P.} (1993).
	\newblock Generalized bootstrat for studentized {$U$}-statistics: A rank
	statistic approach.
	\newblock \textit{Statistics and Probability Letters} \textbf{16} 225--233.
	
	\bibitem[{Huskova and Janssen(1993)}]{huskova1993consistency}
	\textsc{Huskova, M.} and \textsc{Janssen, P.} (1993).
	\newblock Consistency of the generalized bootstrap for degenerate
	{$U$}-statistics.
	\newblock \textit{Annals of Statistics} \textbf{21} 1811--1823.
	
	\bibitem[{Jiang(2004)}]{jiang2004asymptotic}
	\textsc{Jiang, T.} (2004).
	\newblock The asymptotic distributions of the largest entries of sample
	correlation matrices.
	\newblock \textit{Annals of Applied Probability} \textbf{14} 865--880.
	
	\bibitem[{Jiang and Yang(2013)}]{jiang2013central}
	\textsc{Jiang, T.} and \textsc{Yang, F.} (2013).
	\newblock Central limit theorems for classical likelihood ratio tests for
	high-dimensional normal distributions.
	\newblock \textit{Annals of Statistics} \textbf{41} 2029--2074.
	
	\bibitem[{Korolyuk and Borovskich(2013)}]{korolyuk2013theory}
	\textsc{Korolyuk, V.~S.} and \textsc{Borovskich, Y.~V.} (2013).
	\newblock \textit{Theory of $U$-Statistics}, vol. 273.
	\newblock Springer Science and Business Media.
	
	\bibitem[{Ledoit and Wolf(2002)}]{ledoit2002some}
	\textsc{Ledoit, O.} and \textsc{Wolf, M.} (2002).
	\newblock Some hypothesis tests for the covariance matrix when the dimension is
	large compared to the sample size.
	\newblock \textit{Annals of Statistics} \textbf{30} 1081--1102.
	
	\bibitem[{Lee(1990)}]{1990u}
	\textsc{Lee, A.} (1990).
	\newblock \textit{U-Statistics: Theory and Practice}.
	\newblock Statistics: A Series of Textbooks and Monographs, Taylor and Francis.
	
	\bibitem[{Lee et~al.(2012)Lee, Emond, Bamshad, Barnes, Rieder, Nickerson, Team,
		Christiani, Wurfel and Lin}]{lee2012optimal}
	\textsc{Lee, S.}, \textsc{Emond, M.~J.}, \textsc{Bamshad, M.~J.},
	\textsc{Barnes, K.~C.}, \textsc{Rieder, M.~J.}, \textsc{Nickerson, D.~A.},
	\textsc{Team, E. L.~P.}, \textsc{Christiani, D.~C.}, \textsc{Wurfel, M.~M.}
	and \textsc{Lin, X.} (2012).
	\newblock Optimal unified approach for rare-variant association testing with
	application to small-sample case-control whole-exome sequencing studies.
	\newblock \textit{The American Journal of Human Genetics} \textbf{91} 224--237.
	
	\bibitem[{Li and Rosalsky(2006)}]{li2006some}
	\textsc{Li, D.} and \textsc{Rosalsky, A.} (2006).
	\newblock Some strong limit theorems for the largest entries of sample
	correlation matrices.
	\newblock \textit{Annals of Applied Probability} \textbf{16} 423--447.
	
	\bibitem[{Li and Chen(2012)}]{li2012two}
	\textsc{Li, J.} and \textsc{Chen, S.} (2012).
	\newblock Two sample tests for high-dimensional covariance matrices.
	\newblock \textit{Annals of Statistics} \textbf{40} 908--940.
	
	\bibitem[{Lin and Tang(2011)}]{lin2011general}
	\textsc{Lin, D.} and \textsc{Tang, Z.} (2011).
	\newblock A general framework for detecting disease associations with rare
	variants in sequencing studies.
	\newblock \textit{The American Journal of Human Genetics} \textbf{89} 354--367.
	
	\bibitem[{Liu et~al.(2008)Liu, Lin and Shao}]{liu2008asymptotic}
	\textsc{Liu, W.}, \textsc{Lin, Z.} and \textsc{Shao, Q.} (2008).
	\newblock The asymptotic distribution and {B}erry-{E}sseen bound of a new test
	for independence in high dimension with an application to stochastic
	optimization.
	\newblock \textit{Annals of Applied Probability} \textbf{18} 2337--2366.
	
	\bibitem[{Liu and Shao(2013)}]{liu2013cramer}
	\textsc{Liu, W.} and \textsc{Shao, Q.} (2013).
	\newblock A cram{\'e}r moderate deviation theorem for {H}otellingâs
	${T}^2$-statistic with applications to global tests.
	\newblock \textit{Annals of Statistics} \textbf{41} 296--322.
	
	\bibitem[{Lo()}]{lo1987large}
	\textsc{Lo, A.~Y.} (1987).
	\newblock A large sample study of the {B}ayesian bootstrap.
	\newblock \textit{Annals of Statistics} \textbf{15} 360--375.
	
	\bibitem[{Mai and Zou(2013)}]{mai2013semiparametric}
	\textsc{Mai, Q.} and \textsc{Zou, H.} (2013).
	\newblock Semiparametric sparse discriminant analysis in ultra-high dimensions.
	\newblock \textit{Biometrika} \textbf{99} 29--42.
	
	\bibitem[{Mai et~al.(2012)Mai, Zou and Yuan}]{mai2012direct}
	\textsc{Mai, Q.}, \textsc{Zou, H.} and \textsc{Yuan, M.} (2012).
	\newblock A direct approach to sparse discriminant analysis in ultra-high
	dimensions.
	\newblock \textit{Biometrika} \textbf{99} 29--42.
	
	\bibitem[{Mason and Newton(1992)}]{mason1992rank}
	\textsc{Mason, D.~M.} and \textsc{Newton, M.~A.} (1992).
	\newblock A rank statistics approach to the consistency of a general bootstrap.
	\newblock \textit{Annals of Statistics} \textbf{20} 1611--1624.
	
	\bibitem[{Nagao(1973)}]{nagao1973some}
	\textsc{Nagao, H.} (1973).
	\newblock On some test criteria for covariance matrix.
	\newblock \textit{Annals of Statistics} \textbf{1} 700--709.
	
	\bibitem[{Pan et~al.(2014)Pan, Kim, Zhang, Shen and Wei}]{pan2014powerful}
	\textsc{Pan, W.}, \textsc{Kim, J.}, \textsc{Zhang, Y.}, \textsc{Shen, X.} and
	\textsc{Wei, P.} (2014).
	\newblock A powerful and adaptive association test for rare variants.
	\newblock \textit{Genetics} \textbf{197} 1081--1095.
	
	\bibitem[{Parzen et~al.(1994)Parzen, Wei and Ying}]{parzen1994resampling}
	\textsc{Parzen, M.}, \textsc{Wei, L.} and \textsc{Ying, Z.} (1994).
	\newblock A resampling method based on pivotal estimating functions.
	\newblock \textit{Biometrika} \textbf{81} 341--350.
	
	\bibitem[{Roy(1957)}]{roy1957some}
	\textsc{Roy, S.~N.} (1957).
	\newblock \textit{Some Aspects of Multivariate Analysis}.
	\newblock Wiley, New York.
	
	\bibitem[{Rubin et~al.(1981)}]{rubin1981bayesian}
	\textsc{Rubin, D.~B.} \textsc{et~al.} (1981).
	\newblock The {B}ayesian bootstrap.
	\newblock \textit{Annals of Statistics} \textbf{9} 130--134.
	
	\bibitem[{Schott(2005)}]{schott2005testing}
	\textsc{Schott, J.~R.} (2005).
	\newblock Testing for complete independence in high dimensions.
	\newblock \textit{Biometrika} \textbf{92} 951--956.
	
	\bibitem[{Schott(2007)}]{schott2007test}
	\textsc{Schott, J.~R.} (2007).
	\newblock A test for the equality of covariance matrices when the dimension is
	large relative to the sample sizes.
	\newblock \textit{Computational Statistics and Data Analysis} \textbf{51}
	6535--6542.
	
	\bibitem[{Shao et~al.(2011)Shao, Wang, Deng and Wang}]{shao2011sparse}
	\textsc{Shao, J.}, \textsc{Wang, Y.}, \textsc{Deng, X.} and \textsc{Wang, S.}
	(2011).
	\newblock Sparse linear discriminant analysis by thresholding for high
	dimensional data.
	\newblock \textit{Annals of statistics} \textbf{39} 1241--1265.
	
	\bibitem[{Shao and Zhou(2014)}]{shao2014necessary}
	\textsc{Shao, Q.} and \textsc{Zhou, W.} (2014).
	\newblock Necessary and sufficient conditions for the asymptotic distributions
	of coherence of ultra-high dimensional random matrices.
	\newblock \textit{Annals of Probability} \textbf{42} 623--648.
	
	\bibitem[{Srivastava(2009)}]{srivastava2009test}
	\textsc{Srivastava, M.~S.} (2009).
	\newblock A test for the mean vector with fewer observations than the dimension
	under non-normality.
	\newblock \textit{Journal of Multivariate Analysis} \textbf{100} 518--532.
	
	\bibitem[{Srivastava and Du(2008)}]{srivastava2008test}
	\textsc{Srivastava, M.~S.} and \textsc{Du, M.} (2008).
	\newblock A test for the mean vector with fewer observations than the
	dimension.
	\newblock \textit{Journal of Multivariate Analysis} \textbf{99} 386--402.
	
	\bibitem[{Srivastava and Yanagihara(2010)}]{srivastava2010testing}
	\textsc{Srivastava, M.~S.} and \textsc{Yanagihara, H.} (2010).
	\newblock Testing the equality of several covariance matrices with fewer
	observations than the dimension.
	\newblock \textit{Journal of Multivariate Analysis} \textbf{101} 1319--1329.
	
	\bibitem[{Tibshirani et~al.(2002)Tibshirani, Hastie, Narasimhan and
		Chu}]{tibshirani2002diagnosis}
	\textsc{Tibshirani, R.}, \textsc{Hastie, T.}, \textsc{Narasimhan, B.} and
	\textsc{Chu, G.} (2002).
	\newblock Diagnosis of multiple cancer types by shrunken centroids of gene
	expression.
	\newblock \textit{Proceedings of the National Academy of Sciences} \textbf{99}
	6567--6572.
	
	\bibitem[{Wu et~al.(2011)Wu, Lee, Cai, Li, Boehnke and Lin}]{wu2011rare}
	\textsc{Wu, M.~C.}, \textsc{Lee, S.}, \textsc{Cai, T.}, \textsc{Li, Y.},
	\textsc{Boehnke, M.} and \textsc{Lin, X.} (2011).
	\newblock Rare-variant association testing for sequencing data with the
	sequence kernel association test.
	\newblock \textit{The American Journal of Human Genetics} \textbf{89} 82--93.
	
	\bibitem[{Zang et~al.(2007)Zang, He, Zhu, Cao, Sui, Liang, Tian, Jiang and
		Wang}]{yu2007altered}
	\textsc{Zang, Y.}, \textsc{He, Y.}, \textsc{Zhu, C.}, \textsc{Cao, Q.},
	\textsc{Sui, M.}, \textsc{Liang, M.}, \textsc{Tian, L.}, \textsc{Jiang, T.}
	and \textsc{Wang, Y.} (2007).
	\newblock Altered baseline brain activity in children with {ADHD} revealed by
	resting-state functional {MRI}.
	\newblock \textit{Brain and Development} \textbf{29} 83--91.
	
	\bibitem[{Zhang and Wu(2015)}]{zhang2015gaussian}
	\textsc{Zhang, D.} and \textsc{Wu, W.} (2015).
	\newblock Gaussian approximation for high dimensional time series.
	\newblock \textit{arXiv preprint arXiv:1508.07036} .
	
	\bibitem[{Zhang and Cheng(2014)}]{zhang2014bootstrapping}
	\textsc{Zhang, X.} and \textsc{Cheng, G.} (2014).
	\newblock Bootstrapping high dimensional time series.
	\newblock \textit{arXiv preprint arXiv:1406.1037} .
	
	\bibitem[{Zhou et~al.(2015)Zhou, Han, Zhang and Liu}]{zhou2015extreme}
	\textsc{Zhou, C.}, \textsc{Han, F.}, \textsc{Zhang, X.} and \textsc{Liu, H.}
	(2015).
	\newblock An extreme-value approach for testing the equality of large
	{$U$}-statistic based correlation matrices.
	\newblock \textit{arXiv preprint arXiv:1502.03211} .
	
	\bibitem[{Zhou(2007)}]{zhou2007asymptotic}
	\textsc{Zhou, W.} (2007).
	\newblock Asymptotic distribution of the largest off-diagonal entry of
	correlation matrices.
	\newblock \textit{Transactions of the American Mathematical Society}
	\textbf{359} 5345--5363.
	
	\bibitem[{Zou et~al.(2008)Zou, Zhu, Yang, Zuo, Long, Cao, Wang and
		Zang}]{zou2008improved}
	\textsc{Zou, Q.}, \textsc{Zhu, C.}, \textsc{Yang, Y.}, \textsc{Zuo, X.},
	\textsc{Long, X.}, \textsc{Cao, Q.}, \textsc{Wang, Y.} and \textsc{Zang, Y.}
	(2008).
	\newblock An improved approach to detection of amplitude of low-frequency
	fluctuation ({ALFF}) for resting-state f{MRI}: Fractional {ALFF}.
	\newblock \textit{Journal of Neuroscience Methods} \textbf{172} 137--141.
	
\end{thebibliography}

\begin{thebibliography}{14}
	\bibitem[{Arcones and Gine(1993)}]{arcones1993limit:app}
	\textsc{Arcones, M.~A.} and \textsc{Gine, E.} (1993).
	\newblock {Limit theorems for ${U}$-processes}.
	\newblock \textit{Annals of Probability} \textbf{21} 1494--1542.
	
	
\bibitem[{Barvinok(2014)}]{barvinok2014thrifty:app}
\textsc{Barvinok, A.} (2014).
\newblock Thrifty approximations of convex bodies by polytopes.
\newblock \textit{International Mathematics Research Notices} \textbf{2014}
4341--4356.

\bibitem[{Bonn{\'e}ry et~al.(2012)Bonn{\'e}ry, Breidt and
	Coquet}]{bonnery2012uniform:app}
\textsc{Bonn{\'e}ry, D.}, \textsc{Breidt, F.~J.} and \textsc{Coquet, F.}
(2012).
\newblock Uniform convergence of the empirical cumulative distribution function
under informative selection from a finite population.
\newblock \textit{Bernoulli} \textbf{18} 1361--1385.

\bibitem[{Boucheron et~al.(2013)Boucheron, Lugosi and
	Massart}]{boucheron2013concentration:app}
\textsc{Boucheron, S.}, \textsc{Lugosi, G.} and \textsc{Massart, P.} (2013).
\newblock \textit{Concentration {I}nequalities: {A} {N}onasymptotic {T}heory of
	{I}ndependence}.
\newblock Oxford University Press.

\bibitem[{Cai and Liu(2011{\natexlab{a}})}]{cai2011adaptive:app}
\textsc{Cai, T.} and \textsc{Liu, W.} (2011{\natexlab{a}}).
\newblock Adaptive thresholding for sparse covariance matrix estimation.
\newblock \textit{Journal of the American Statistical Association} \textbf{106}
672--684.

\bibitem[{Cai et~al.(2014)Cai, Liu and Xia}]{tony2014two:app}
\textsc{Cai, T.}, \textsc{Liu, W.} and \textsc{Xia, Y.} (2014).
\newblock Two-sample test of high dimensional means under dependence.
\newblock \textit{Journal of the Royal Statistical Society: Series B
	(Statistical Methodology)} \textbf{76} 349--372.


\bibitem[{Chernozhukov et~al.(2014)Chernozhukov, Chetverikov and
	Kato}]{chernozhukov2014central:app}
\textsc{Chernozhukov, V.}, \textsc{Chetverikov, D.} and \textsc{Kato, K.}
(2014).
\newblock Central limit theorems and bootstrap in high dimensions.
\newblock \textit{arXiv preprint arXiv:1412.3661} .

\bibitem[{David and Nagaraja(2003)}]{DN03:app}
\textsc{David, H.~A.} and \textsc{Nagaraja, H.~N.} (2003).
\newblock \textit{Order {S}tatistics (3rd)}.
\newblock John Wiley.


\bibitem[{Delaigle et~al.(2011)Delaigle, Hall and Jin}]{delaigle2011robustness:app}
\textsc{Delaigle, A.}, \textsc{Hall, P.} and \textsc{Jin, J.} (2011).
\newblock Robustness and accuracy of methods for high dimensional data analysis
based on student's t-statistic.
\newblock \textit{Journal of the Royal Statistical Society: Series B
	(Statistical Methodology)} \textbf{73} 283--301.

\bibitem[{Dudley(2014)}]{dudley2014uniform:app}
\textsc{Dudley, R.~M.} (2014).
\newblock \textit{Uniform Central Limit Theorems (2rd)}.
\newblock Cambridge University Press.

\bibitem[{Nazarov(2003)}]{nazarov2003:app}
\textsc{Nazarov, F.} (2003).
\newblock On the maximal perimeter of a convex set in ${R}^n$ with respect to a
{G}aussian measure.
\newblock \textit{Geometric Aspects of Functional Analysis}  Lecture Notes in
Mathematics Volume 1807, 169--187.

\bibitem[{Tsagris et~al.(2014)Tsagris, Beneki and Hassani}]{tsagris2014folded:app}
\textsc{Tsagris, M.}, \textsc{Beneki, C.} and \textsc{Hassani, H.} (2014).
\newblock On the folded normal distribution.
\newblock \textit{Mathematics} \textbf{2} 12--28.

\bibitem[{Zang et~al.(2007)Zang, He, Zhu, Cao, Sui, Liang, Tian, Jiang and
	Wang}]{yu2007altered:app}
\textsc{Zang, Y.}, \textsc{He, Y.}, \textsc{Zhu, C.}, \textsc{Cao, Q.},
\textsc{Sui, M.}, \textsc{Liang, M.}, \textsc{Tian, L.}, \textsc{Jiang, T.}
and \textsc{Wang, Y.} (2007).
\newblock Altered baseline brain activity in children with {ADHD} revealed by
resting-state functional {MRI}.
\newblock \textit{Brain and Development} \textbf{29} 83--91.

\bibitem[{Zou et~al.(2008)Zou, Zhu, Yang, Zuo, Long, Cao, Wang and
	Zang}]{zou2008improved:app}
\textsc{Zou, Q.}, \textsc{Zhu, C.}, \textsc{Yang, Y.}, \textsc{Zuo, X.},
\textsc{Long, X.}, \textsc{Cao, Q.}, \textsc{Wang, Y.} and \textsc{Zang, Y.}
(2008).
\newblock An improved approach to detection of amplitude of low-frequency
fluctuation ({ALFF}) for resting-state f{MRI}: Fractional {ALFF}.
\newblock \textit{Journal of Neuroscience Methods} \textbf{172} 137--141.


\end{thebibliography}
\end{document}